\documentclass[12pt]{article}
\usepackage{amsmath}
\usepackage{amssymb}
\usepackage{amsfonts}
\usepackage{mathrsfs}
\usepackage{mathtools}
\usepackage[usenames]{color}
\usepackage[numbers]{natbib}
\usepackage{enumitem} 
\usepackage{amscd} %commutative diagramm

\usepackage[normalem]{ulem} % \sout (barrer)
\usepackage{comment}

\usepackage{hyperref}
\hypersetup{colorlinks=true,linkcolor=blue,
citecolor=blue} 

\usepackage{amsthm}
\def\ben#1{\begin{equation}#1\end{equation}}

\def\al#1{\begin{align*}#1\end{align*}}
\def\aln#1{\begin{align}#1\end{align}}

\usepackage{latexsym}
\newsavebox{\toy}
\savebox{\toy}{\framebox[0.65em]{\rule{0cm}{1ex}}}
\newcommand{\QED}{\usebox{\toy}\end{demo}}

\numberwithin{equation}{section}

\newtheorem{theorem}{Theorem}[section]
\newtheorem{lemma}[theorem]{Lemma}
\newtheorem{proposition}[theorem]{Proposition}
\newtheorem{cor}[theorem]{Corollary}
\newtheorem{rem}[theorem]{Remark}
\newtheorem{notation}[theorem]{Notation}

\newtheorem{example}[theorem]{Example}

\newcommand{\bd}{\begin{displaymath}}
\newcommand{\ed}{\end{displaymath}}
\newcommand{\N}{{\mathbb{N}}}

\newcommand{\R}{{\mathbb{R}}}
\newcommand{{\rd}}{\R^d}

\newcommand{\IP}{{\mathbb P}}
\newcommand{\IE}{\mathbb E}
\renewcommand{\P}{{\mathbb P}}
\newcommand{\E}{\mathbb E}
\newcommand{\DP}{{\mathrm P}}
\newcommand{\DE}{{\mathrm E}}
\newcommand{\lan}{\langle}
\newcommand{\ran}{\rangle}

\renewcommand{\b}{\beta}

\newcommand{\gm}{\gamma}

\newcommand{\D}{\Delta}

\newcommand{\e}{\varepsilon}
\newcommand{\ve}{\epsilon}

\newcommand{\kU}{\mathscr{U}}
\newcommand{\kV}{\mathscr{V}}
\newcommand{\kW}{\mathcal{W}}
\newcommand{\kM}{\mathcal{M}}

\newcommand{\dd}{\text{\rm d}}             % a straight d for differentials
\newcommand{\dB}{\xi}
\newcommand{\sZ}{{\bf{\mathcal{Z}}}}

\newcommand{\ssup}[1] {{\scriptscriptstyle{({#1}})}}

\newcommand{\dis}{\displaystyle}

\newcommand{\cvlaw}{\stackrel{(d)}{\longrightarrow}}
\newcommand{\cvIP}{\stackrel{\IP}{\longrightarrow}}
\newcommand{\cvLone}{\stackrel{L^1}{\longrightarrow}}
\newcommand{\cvLtwo}{\stackrel{L^2}{\longrightarrow}}
\newcommand{\eqlaw}{\stackrel{(d)}{=}}
\newcommand{\SN}{\color{red}}
\newcommand{\MN}{\color{magenta}}

\AtBeginDocument{% ceci supprime les numeros MR
   \def\MR#1{}  }

\author{
Shuta Nakajima\footnote{Department of Mathematics and Computer Science of the University of Basel. \texttt{shuta.nakajima@unibas.ch}}
   \and
   Makoto Nakashima\footnote{Graduate School of Mathematics, Nagoya University. \texttt{nakamako@math.nagoya-u.ac.jp}}
}
\begin{document}

%%%\pagestyle{myheadings}

%
%\pagestyle{myheadings}
%\markboth{FC-CC-CM}{EW limits}
%
%
%\title{A martingale approach to SHE and KPZ: the law of large numbers and fluctuations in the optimal weak disorder regime for $d\geq 3$}
\title{Fluctuations of two-dimensional stochastic heat equation and KPZ equation  in subcritical regime for general initial conditions}

%\title{Law of large number and fluctuations in the weak disorder and full $L^2$-region for SHE and KPZ equation in dimention $d\geq 3$}

  %Law of large number, Edwards-Wilkinson and Gaussian Free Field fluctuations for the SHE and KPZ equations and directed polymers in $d\geq 3$ for optimal temperature regions}
%
%\author{Francis Comets$^{1,}$\footnote{Corresponding author}, Cl\'ement Cosco$^{1}$, Chiranjib Mukherjee$^{2}$}

\date{}
\maketitle
\begin{center}
\date{\today}
\end{center}
\begin{abstract}
  
The solution of Kardar-Parisi-Zhang equation (KPZ equation) is solved formally via Cole-Hopf transformation $h=\log u$, where $u$ is the solution of multiplicative stochastic heat equation(SHE). In \cite{CD20,CSZ20,G20}, they consider the solution of two dimensional KPZ equation via  the solution $u_\e$ of SHE with flat initial condition and with noise which is mollified in space on scale in $\e$ and its strength is weakened as $\b_\e=\hat{\b} \sqrt{\frac{2\pi \e}{-\log \e}}$,  and they prove that when $\hat{\b}\in (0,1)$, $\frac{1}{\b_\e}(\log u_\e-\IE[\log u_\e])$ converges in distribution to a solution of Edward-Wilkinson model as a random field. 

In this paper, we consider a stochastic heat equation $u_\e$ with general initial condition $u_0$ and its transformation $F(u_\e)$ for $F$ in a class of functions $\mathfrak{F}$, which contains $F(x)=x^p$ ($0<p\leq 1$) and $F(x)=\log x$. Then, we prove that  $\frac{1}{\b_\e}(F(u_\e(t,x))-\IE[F(u_\e(t,x))])$ converges in distribution to Gaussian random variables  jointly in finitely many $F\in \mathfrak{F}$, $t$, and  $u_0$. In particular, we obtain  the fluctuations of solutions of stochastic heat equations and KPZ equations jointly converge to solutions of SPDEs which depends on $u_0$.

Our main tools are It\^o's formula, the martingale central limit theorem, and the homogenization argument as in \cite{CNN20}. To this end, we also prove the local limit theorem for the partition function of intermediate $2d$-directed polymers

%The local limit theorem for the partition function of intermediate $2d$-directed polymers similar to \cite{Si95,V06} is obtained.

%In \cite{CSZ20}, the problem was reduced to the case in the solution of SHE via approximating $\log u_\e$ by $``u_\e-1"$ and then the Gaussian fluctuations were obtained by the fourth-moment method. In \cite{G20}, Gaussian fluctuations were obtained by Malliavin calculus and the second-order Poincar\'e inequality.

%There have been recently several works studying the regularized stochastic heat equation (SHE) and Kardar-Parisi-Zhang (KPZ) equation in dimension $d\geq 2$ as the smoothing parameter is switched off, but most of the results did not hold in the full temperature regions where they should. Inspired by martingale techniques coming from the directed polymers literature, we first extend the law of large numbers for SHE obtained in \cite{MSZ16} to the full weak disorder region of the associated polymer model and to more general initial conditions. We further extend the Edwards-Wilkinson regime of the SHE and KPZ equation studied in \cite{GRZ18,MU17,DGRZ20} to the full $L^2$-region, along with multidimensional convergence and general initial conditions for the KPZ equation (and SHE), which were not proven before. To do so, we %follow the strategy of \cite{CN19} and 
%rely on a martingale CLT combined with a refinement of the local limit theorem for polymers.
\end{abstract}

\noindent \textbf{Keywords}: KPZ equation, Stochastic heat equation, Edwards-Wilkinson equation,  Local limit theorem for polymers, Stochastic calculus.
\\
\noindent \textbf{AMS 2010 subject classifications}: Primary 60K37. Secondary 60F05, 60G42, 82D60.

%\pagenumbering{roman}
      \setcounter{tocdepth}{2}
 \tableofcontents

{\footnotesize 
% \hspace{-10cm}

%
%\noindent$^{~1}$Universit\'e de Paris,
%Laboratoire de Probabilit\'es, Statistique et Mod\'elisation,\\
% LPSM (UMR 8001 CNRS, SU, UP)
%B\^atiment Sophie Germain, 8 place Aur\'elie Nemours, 75013 Paris\\
%\noindent {\tt comets@lpsm.paris,  ccosco@lpsm.paris}
%
%
%\noindent$^{~2}$Universit\"at M\"unster,
%Fachbereich Mathematik und Informatik,
%Einsteinstra\ss e 62, M\"unster, D-48149\\
%\noindent{\tt chiranjib.mukherjee@uni-muenster.de}

%\input{Sec1}

\section{Introduction and Main result}
%\subsection{Regularization scheme of the KPZ equation and SHE}
KPZ equation is an SPDE formally given by \begin{equation} \label{eq:formalKPZ}
\frac{\partial}{\partial t} h(t,x) = \frac12 \D h(t,x) +  \frac12   |\nabla h(t,x) |^2+
 \b  \dot\xi(t,x),
\end{equation}
where $\dB$ is a time-space white noise on $[0,\infty)\times \mathbb{R}^d$. This SPDE is  ill-posed due to the non-linear term $\nabla h$ which should be a generalized function. 

For $d=1$, Bertini and Giacomin formulated the solution of \eqref{eq:formalKPZ} via Cole-Hopf solution $h=\log u$ \cite{BG97}, where  $u$ is the solution of stochastic heat equation \begin{align}\label{eq:formalSHE}
\frac{\partial }{\partial t}u(t,x)=\frac{1}{2}\D u(t,x)+\b u(t,x)\dot{\dB}(t,x).
\end{align}

In  dimension $d=2$, we consider a space-regularized multiplicative stochastic heat equation but we need to scale  the disorder strength:  \begin{align}
\frac{\partial u_\e}{\partial t}=\frac{1}{2}\D u_\e+\b_\e u_\e \dot\xi_\e,\quad u_\e(0,x)=u_0(x)
\end{align}
where $\beta_\e=\hat{\b}\sqrt{\frac{2\pi}{-\log \e}}$ with $\hat{\b}\geq 0$, and  $\xi_\e$ is a mollification in space of $\xi$ such that $\xi_\e \Rightarrow \xi$ as $\e\to 0$, i.e.,\
\[
\dot\xi_{\e}(t,x) = (\dot\xi(t,\cdot)\star \phi_\e)(x)=  \int \phi_\e(x - y) \dot{\xi}(t,y) \dd y,
\]
with $\phi_\e({ x}) = \e^{-2} \phi(\e^{-1} x)$ and $\phi$ being a smooth, non-negative, compactly supported, symmetric function on $\mathbb R ^2$, so that $\int \phi(x) \dd x =1$, and $\phi_\e$ converges in distribution to the Dirac mass $\delta_0$.  Let $h_\e=\log u_\e$. Then,  we find by It\^{o}'s formula that $h_\e$ satisfies the SPDE\begin{equation}\label{eq:KPZe}
 \frac{\partial}{\partial t} h_{\e} = \frac12 \D h_{\e} +  \bigg[\frac12   |\nabla h_\e |^2  - C_\e\bigg]+
 \b_\e   \dot\xi_{\e} \;,\quad\,\,   h_{\e}(0,x) = h_0(x),
\end{equation}
where $C_\e$ is a diverging parameter: 
\begin{equation}\label{C-eps}
C_\e= \frac{\b_\e^2 V(0) }{2\e^2},
\end{equation}
where $V(x)=\int_{\R^2}\phi(x-y)\phi(y)\dd y $.
Caravenna, Sun, and Zygouras proved that if the initial condition is flat, that is $h_\e(0,x)= h_0(x)\equiv 0$ and $\hat{\b}\in (0,1)$,  then $\dis {\b_\e^{-1}}\left(h_\e-\IE[h_\e]\right)$ converges in distribution to the solution to Edward-Wilkinson equation as a random field \cite{CSZ20}. We should remark Chatterjee and Dunlap addressed the tightness of ${\b_\e^{-1}}\left(h_\e-\IE[h_\e]\right)$ \cite{CD20} and Gu obtained Edward-Wilkinson limit in $\hat{\b}\in (0,\beta_0)$ for some $\beta_0\leq 1$ \cite{G20}.

On the other hand, the one-point distribution $h_\e(t,x)$ converges to a random variable as follows: %{\SN It's a bit difficult to find the theorem since a number is not assigned.}
\begin{theorem}\label{thm:CSZ17b}{\cite[Theorem 2.15]{CSZ17b}} 
For any $t>0$ and $x\in \R^2$, \begin{align*}
h_\e(t,x)\Rightarrow \begin{cases}
X_{\hat{\beta}}:=\sigma(\hat{\b})Z-\frac{1}{2}\sigma^2(\hat{\b}),\quad &0\leq \hat{\b}<1\\
0,&\hat{\b}\geq 1
\end{cases},
\end{align*}
where $Z$ is a random variable with standard normal distribution and $\sigma(\hat{\b})=\sqrt{\log \frac{1}{1-\hat{\b}^2}}$.
\end{theorem}

%\section{Statement of the results}\label{sec:statResults}
%\subsection{Gaussian fluctuation for 2d KPZ equation}
%\label{sec:GaussianKPZ}
%Let $\phi$ be a smooth, compactly supported, symmetric function on $\mathbb R ^2$, so that $\phi_\e$ converges in distribution to the Dirac mass $\delta_0$. 

We will look at the fluctuation of $u_\e$ for general initial conditions in our main results. Let $\mathfrak{C}$  be a set of continuous functions which satisfies \begin{equation}\label{eq:u0cond}
0<\inf_{x\in \mathbb{R}^2}u_0(x)\leq \sup_{x\in\mathbb{R}^2}u_0(x)<\infty,
\end{equation}
or equivalently
\begin{align*}
\|\log u_0\|_\infty<\infty.
\end{align*}

%Then, we consider a space-regularized stochastic heat equation: \begin{equation*}
%\frac{\partial}{\partial t}u_\e=\frac{1}{2}\D u_\e+\b_\e u_\e\xi_\e,\quad u_\e(0,x)=u_0(x),
%\end{equation*}
%where $\xi_\e(t,x)=\int_{\R^2}\phi_\e(x-y)\dot{\xi}(t,y)\dd y$ with $\phi_\e = \e^{-2} \phi(\e^{-1} x)$, $\b_\e= \hat{\b}\sqrt{\frac{2\pi}{-\log \e}}$ with $\hat{\b}>0$.

{ Let $\mathfrak{F}$ be a set of functions $F\in C^3((0,\infty))$ such that there exists a constant $C=C_F>0$  such that for any $x\in (0,\infty)$\begin{align}
%&F(x)\leq C(x^{-p}+x)\\
&|F'(x)|\leq C(x^{-1}+1),\quad 
|F''(x)|\leq C(x^{-2}+1),\quad |F'''(x)|\leq C(x^{-3}+1).\label{eq:Fass}
%&F''(x)\leq C(x^{-p}+1)
\end{align}
% of logarithms or power functions with power less than or equal to $1$, that is \begin{align*}
%\mathfrak{F}=\{\log x,x^p\,(p\leq 1)\}
%\end{align*} %
Then, $\mathfrak{F}$ contains $x^p$ ($0<p\leq 1$), $\log x$, $\sin x$, $\cos x$, $e^{-x}$ $\dots$ In this paper, we focus on the fluctuation of $F(u_\e)$. 
\begin{example}
  In particular cases, we find  that  $u_\e^{(F)}=F(u_\e)$ satisfies SPDEs.
  \begin{itemize}
\item If $F(x)=x^{p}$ $(0<p\leq 1)$, then $u_\e^{(F)}(0,x)=u_0(x)^p$ and \begin{align*}
\partial_t u_\e^{(F)}(t)=\frac{1}{2}\Delta u_\e^{(F)}-\frac{(p-1)}{2p}\frac{|\nabla u_\e^{(F)}|^2}{u_\e^{(F)}}+\frac{\b_\e^2V(0)p(p-1)u_\e^{(F)}}{2\e^2}+\b_\e u_\e^{(F)}\dot\xi_\e.
\end{align*}
\item If $F(x)=\log x$ $($KPZ equation$)$, then $u_\e^{(F)}(0,x)=\log u_0(x)$ and \begin{align*}
\partial_t u_\e^{(F)}(t)=\frac{1}{2}\Delta u_\e^{(F)}+\frac{1}{2}{|\nabla u_\e^{(F)}|^2}-\frac{\b_\e^2V(0)}{2\e^2}+\b_\e \dot\xi_\e. 
\end{align*}
\end{itemize}
\end{example}

We remark that $u_\e$ is a process indexed by $u_0$ and $\hat{\b}$ so we should write $u_\e=u_\e^{(\hat{\b},u_0)}$ and  $u_\e^{(F)}=u_\e^{(F,\hat{\b},u_0)}$. However, we omit $\hat{\b}$ and $u_0$ for simplicity of notation when it is clear from the context.
}

%\subsection{On the fluctuations for SHE} \label{sec:ResultsSHE}
We denote by $\mathcal C^\infty_c$ the set of infinitely differentiable, compactly supported functions on $\mathbb R^2$.
\begin{theorem}\label{th:EWlimit} Suppose $u_0^{(1)},\cdots,u_0^{(n)}\in \mathfrak{C}$, $\hat{\b}_1,\cdots,\hat{\b}_n\in (0,1)$ and $F_1,\cdots,F_n\in \mathfrak{F}$.
%Suppose $u_\e(0,\cdot)= u_0$ where $u_0$ satisfies \eqref{eq:u0cond}. 
For $t_{1}, \cdots, t_{n}\geq 0$ and $f_{1},\cdots,f_{n}\in\mathcal C^\infty_c$, the following convergence holds jointly  as $\e\to 0$,
\begin{align*} %\label{eq:EWGICSHEResult}
&\b_\e^{-1}\int_{\R^2} f_{i}(x) \left(u_\e^{(F_i,u_0^{(i)},\hat{\b}_i)}(t_{i},x) - \IE\left[{u}_\e^{(F_i,u_0^{(i)},\hat{\b}_i)}(t_{i},x)\right]  \dd x\right)\\
&\hspace{4em} \cvlaw \kU_{t_i}(f_i,F_i,\hat{\b}_i,u_0^{(i)}) \;,
\end{align*}
where $\dis \left\{\kU_{t_i}(f_i,F_i,\hat{\b}_i,u_0^{(i)})\right\}_{i=1}^n$ is centered Gaussian random variables  
with covariance {\begin{align}
& \frac{1}{1-\hat{\b}_i\hat{\beta}_j}\int_{0}^{{t_i}\wedge t_j}\dd \sigma\int\dd x\dd yf_i(x)f_j(y)
%\notag\\
%&\hspace{3em}\times 
I^{(i)}(x)I^{(j)}(y)%\notag\\
%\IE\left[F_i'\left(\frac{e^{X_{\hat{\b}_i}}\bar{u}^{(i)}(t,x)}{\sqrt{1-\hat{\b}_i^2}}\right)\right]
%\IE\left[F_j'\left(\frac{e^{X_{\hat{\b}_j}}\bar{u}^{(j)}(t,y)}{\sqrt{1-\hat{\b}_j^2}}\right)\right]
%\notag\\
%&\hspace{3em}\times 
%\\
%&\hspace{8em}\times
%&\hspace{5em} \times 
\int\dd z\rho_\sigma(x,z)\rho_\sigma(y,z)\bar{u}^{\ssup i}(t_i-\sigma,z)\bar{u}^{\ssup j}(t_j-\sigma,z),\label{eq:CovSt}
  \end{align}
 % {\SN Have we defined $\bar{u}^{(i)}$?}{\MN No! Fixed.}
with  \begin{align*}
I^{(t,F,\hat{\b},u_0)}(x)=I(x)=
\IE\left[
F'\left(e^{X_{\hat{\b}}}
\bar{u}(t,x)\right)
e^{X_{\hat{\b}}}\right]=
\IE\left[F'\left(e^{X_{\hat{\b}}+{\sigma^2(\hat{\b})}}
\bar{u}(t,x)\right)
\right],
\end{align*}
{  $X_{\hat{\b}}$ is a Gaussian random variable defined in Theorem \ref{thm:CSZ17b}, $\rho_t(x)=(2\pi t)^{-1}e^{-\frac{|x|^2}{2t}}$ is the heat kernel and $\rho_t(x,y)=\rho_t(x-y)$, and $\bar{u}(t,x)=\int \rho_t(x,y)u_0(y)\dd y$}. We write $I^{(i)}$ for $I^{(t_i,F_i,\hat{\b}_i,u_0^{(i)})}(x)$ to make notation simple.
}
%with $\kU$ defined in \eqref{eq:EW_GRZ} and $\bar u$ in \eqref{eq:Def_ubar}.
\end{theorem}

\begin{rem}
The centered Gaussian field $\dis \left\{\kU_{t}(f,F,\hat{\b},u_0):t\in [0,\infty),F\in \mathfrak{F},\hat{\b}\in [0,1),u_0\in \mathfrak{C}\right\}$ with covariance \eqref{eq:CovSt} can be constructed explicitly.  Let $\left\{\kV^{(\hat{\b},u_0)}(t,x):(t,x)\in  [0,\infty)\times \R^2\right\}$ be solutions of the following SPDE: $\kV^{(\hat{\b},u_0)}(0,x)\equiv 0$ and  \begin{align*}
\partial_t \kV^{(\hat{\b},u_0)}(t,x)=\frac{1}{2}\Delta \kV^{(\hat{\b},u_0)}(t,x)+{\bar{u}(t,x)}\dot{\xi}_{\hat{\b}}(t,x),
\end{align*}
where $\xi_{\hat{\b}}=\dis \sum_{n=0}^\infty \hat{\b}\xi^{(n)}$ with an independent sequence  time-space white noises $\{\xi^{(n)}\}$ for $\hat\b\in (0,1)$. We remark that $\xi_\b$ is a time-space white noise with strength $\dis \frac{1}{1-\b^2}$ and  $\dis \IE\left[\dot{\xi}_{{\b}}(t,x)\dot{\xi}_{\b'}(t',x')\right]=\frac{1}{1-\b\b'}\delta_{t,t'}\delta_{x,x'}$.
Then, by Duhamel's principle, $\kV^{(\hat{\b}_i,u_0^{(i)})}$ is given by
\begin{align}\label{explicit construction}
\kV^{(\hat{\b}^{(i)},u_0^{(i)})}(t,x)=\int_0^t \int_{\R^2}\rho_{t-s}(x,y)\bar{u}^{(i)}(s,y)\xi_{\hat{\b}_i}(\dd s,\dd y)
\end{align}
and the centered Gaussian field given by\begin{align*}
\kV_{t_i}(f_i,F_i,\hat{\b}_i,u_0^{(i)})=\int_{\R^2}\dd xI^{\left(i\right)}(x)f_i(x)F_i'(\bar{u}^{(i)}(t_i,x))\kV^{(\hat{\b}_i,u_0^{(i)})}(t,x)
\end{align*}
has the covariance structure \eqref{eq:CovSt}.
\end{rem}

%\begin{example}
%We denote by $\lambda(F,\hat\b)=\begin{cases}
%(1-\hat{\b}^2)^{\frac{p(1-p)}{2}},\quad &F(x)=x^p\\
%1,\quad&F(x)=\log x. 
%\end{cases}$
%\end{example}

\begin{rem} When $F$ is a power function or the logarithm, the limit 
%First, we remark that if $F\in \mathfrak{F}$, then  \begin{itemize}
%\item $(F(x)=x^p)$ $\IE\left[F'\left(\frac{e^{X_{\hat{\b}}-\frac{1}{2}\sigma^2(\hat\b)}\bar{u}^{(i)}(t,x)}{\sqrt{1-\hat{\b}^2}}\right)\right]=p{\bar{u}(t,x)^{p-1}}{(1-\hat\b^2)^{\frac{p(1-p)}{2}}}$
%\item $(F(x)=\log x)$ $\IE\left[F'\left(\frac{e^{X_{\hat{\b}}-\frac{1}{2}\sigma^2(\hat\b)}\bar{u}^{(i)}(t,x)}{\sqrt{1-\hat{\b}^2}}\right)\right]=\frac{1}{\bar{u}(t,x)}$.
%\end{itemize}
$\kU^{(F,u_0,\hat{\b})}$ is a solution of an SPDE:
\begin{itemize}
\item If $F(x)=x^p$, then $\kU^{(F,u_0,\hat{\b})}(0,x)\equiv 0$ and \begin{align*}
\partial_t \kU=\frac{1}{2}\Delta \kU+\frac{p(p-1)}{2}|\nabla \log \bar{u}|^2\kU+(1-p)\nabla \log \bar{u}\cdot \nabla\kU+\frac{\bar{u}(t,x)}{(1-\hat{\b}^2)^{p^2-p+1}}\dot{\xi}(t, x).
\end{align*}
\item If $F(x)=\log x$, then $\kU^{(F,u_0,\hat{\b})}(0,x)\equiv 0$ and \begin{align*}
\partial_t \kU=\frac{1}{2}\Delta \kU+\nabla \log \bar{u}\cdot \nabla\kU+\frac{\bar{u}(t,x)}{(1-\hat{\b}^2)}\dot{\xi}(t, x).
\end{align*}
\end{itemize}
%{\SN Remark as $p\to 0$, $\mathcal{U}^{x^p}$ converges to $\mathcal{U}^{\log{x}}$.}
\end{rem}
From the above remarks, we have the following.
\begin{cor}
  Suppose $\hat{\beta}<1$. As $p\to 0$, $\kU_{t}(f,x^p,\hat{\b},u_0)$ converges to  $\kU_{t}(f,\log{x},\hat{\b},u_0)$.
  \end{cor}

\begin{rem}
In \cite{DGRZ20}, they study the fluctuations of the transformation $F(u_\e)$ for higher dimensional case $d\geq 3$ with $F$, its derivative and second derivative growing at most $x^{-p}+x^p$. They proved that there exists a constant $\beta_p$ such that the Gaussian fluctuation holds for $\hat{\beta}\in (0,\b_p)$. Our assumption on $F$ is slightly different from theirs but we can show the Gaussian fluctuations for all $\hat{\b}$ up to critical point.
\end{rem}

To analyze $u_\e$, we use the Feynman-Kac representation given  in \cite[Section 2]{BC95} where they considered the case $d= 1$ but it is easy to be modified  for $d\geq 2$:\begin{align*}
u_\e(t,x)=\DE_x\left[\exp\left(\b_\e\int_0^t \int_{\R^2}\phi_\e(B_s-y)\dot{\xi}(t-s,\dd y)\dd s-\frac{\b_\e^2tV(0)}{2\e^2}\right)u_0\left(B_t\right)\right],
\end{align*}
where  we denote by $\DP_x$ and $\DE_{{ x}}$ the law and the expectation with respect to two dimensional Brownian motion $B=\{B_t\}_{t\geq 0}$ starting from $x$.

Due to the time-reverse invariance and scale invariance of time-space white noise and the scaling invariance of Brownian motion, $\e B_{\e^{-2} s} \overset{d}{=} B_s$,  $\{u_\e(t,x):x\in\mathbb{R}^2\}$ has the same distribution as
\begin{align}
&\DE_x\left[\exp\left(\b_\e\int_0^t \int_{\R^2}\phi_\e(B_s-y)\dot{\xi}(s,\dd y)\dd s-\frac{\b_\e^2tV(0)}{2\e^2}\right)u_0\left(B_t\right)\right]\notag\\
&=\DE_{\frac{x}{\e}}\left[\exp\left(\b_\e\int_0^{\frac{t}{\e^2}} \int_{\R^2}\phi\left(B_s-\frac{y}{\e}\right)\dot{\xi}(s,\dd y)\dd s-\frac{\b_\e^2tV(0)}{2\e^2}\right)u_0\left({ \e}\,B_{\frac{t}{\e^2}}\right)\right],\label{eq:FKform}
\end{align}
where $(B_s)_{s\geq 0}$ is a Brownian motion path and $\DE_x$ the expectation associated to Brownian motion started at $x\in\mathbb R^2$,  $\b\geq 0$. In particular, for the flat initial condition, $u_\e$ has the same distribution as  {\it partition function} $\sZ_{\frac{t}{\e^2}}\left(\frac{x}{\e}\right)$   of continuum directed polymers, where $\sZ_t(x)$ is given by \begin{align*}
\sZ_t(x)&=\DE_x\left[\Phi_t^{\b_\e}\right]\\
&=\DE_{{x}}\left[\exp\left(\b_\e\int_0^{t} \int_{\R^2}\phi\left(B_s-{y}\right)\dot{\xi}(s,\dd y)\dd s-\frac{\b_\e^2tV(0)}{2}\right)\right],
\end{align*}
where \begin{align*}
\Phi_t^\b=\Phi_t^\b(B,\xi):=\exp\left(\beta \int_0^t \int_{{\mathbb R}^2} \phi(y-B_s) \ \xi(\dd s,\dd y)-\frac{\b^2 V(0)t}{2}\right)
\end{align*}
for $t\geq 0$, $\b\in (0,\infty)$. Thus, we can reduce the problem on the laws of $u_\e$ to the {partition function} of continuum directed polymers. Such connections between SHE (and KPZ equation) and directed polymers have been already pointed out in \cite{KPZ} and used in a lot of researches on SHE and KPZ equation \cite{BC95,BG97,MSZ16,GRZ18,MU17,DGRZ20,CSZ17a,CSZ17b,CSZ19a,CSZ19b,CSZ20,CCM20,CCM19b,CNN20,LZ20}.

\begin{rem} We give our contribution to the problem shortly. Edward-Wilkinson type fluctuations for KPZ equation for $d=2$ have been obtained in \cite{CSZ20} and \cite{G20} with the flat initial condition. On the other hand, we obtain the Gaussian fluctuations for the general initial conditions and multi-dimensional parameters. Also, our proof uses a less technical method, ``{\it martingale CLT}" (see Theorem \ref{thm:JS}) { via It\^o's lemma and homogenization argument.}, with onerous  calculation.
In \cite{CSZ20}, the problem was reduced to the case in the solution of SHE via approximating $\log u_\e$ by $``u_\e-1"$. In \cite{G20}, Gaussian fluctuation was obtained by Malliavin calculus and the second order Poincar\'e inequality. 
\end{rem}

%Recently, there have been a few papers studying the asymptotic behavior $u_\e$ and $h_\e$ as the mollification parameter is switched off ($\e\to 0$) \cite{MSZ16,GRZ18,MU17,DGRZ20,DGRZ18b,CCM20,CCM19b}. However, most of the results that have been obtained are not optimal in the sense that they were not proved to hold in the optimal $\b$-region where they should hold. In this paper, we bring new tools inspired by the polymer literature to extend some of the results to their optimal $\b$-regions. Let us first introduce these $\b$-regions of interest.

\begin{rem} The Gaussian fluctuations for partition functions \cite{LZ20} and solutions of SHE \cite{MSZ16,GRZ18,DGRZ18b,CNN20} and KPZ equation \cite{MU17,DGRZ20,CNN20} in $d\geq 3$ have been proved as well as two dimensional case, where the disorder strength is given by  $\hat{\b}\e^{\frac{d-2}{2}}$  for $d\geq 3$. Also, the Gaussian fluctuations for a nonlinear stochastic heat equation with Gaussian multiplicative noise that is white in time and smooth in space \cite{GL20} and the counterpart for $d=2$ is stated in \cite{DG20} without detailed proof.
\end{rem}

\textbf{Note:} Throughout the paper and if clear from the context, the constant $C$ that appears in successive upper-bounds may take different values.

\textbf{Organization of the article} The main idea of Gaussian fluctuation is the same as in \cite{CNN20}. Section \ref{sec2} is devoted to proving key properties of partition functions of directed polymers, $L^2$-boundedness, boundedness of negative moments, and local limit theorem. Section \ref{sec:proofEWKPZ} is dedicated to the proof of Theorem \ref{th:EWlimit}. In subsection \ref{IdeaG}, we give a rough proof strategy and explain a heuristic idea of  Gaussian fluctuation. The rigorous proof starts from subsection \ref{sub:4.2}.

\section{Some Estimates for Partition Functions}\label{sec2}
{ In this section, we discuss some properties of partition functions of directed polymers in random environment.}

Hereafter, we set
\begin{equation}\label{Def: beta gamma T}
  \begin{split}
&T=T_\e:=\e^{-2}, \\
&\b=\b_\e:=\hat{\b}\sqrt{\frac{2\pi}{\log \e^{-1}}}=\hat{\b}\sqrt{\frac{4\pi}{\log T}}, \text{ and } \\
    &\gamma=\gamma_\e=\hat{\gamma}\sqrt{\frac{2\pi}{\log \e^{-1}}}=\hat{\gamma}\sqrt{\frac{4\pi}{\log T}}.
   \end{split}
\end{equation}

{ Throughout the paper, we write the subscript $\e$ in $T_\e$ and $\b_\e$ in each statement to emphasize its dependence but we often omit the subscript $\e$ in the proofs for simplicity.}

\subsection{$L^p$-bound of partition functions}
%\begin{comment}
%It is known that  $\DE_x\left[\Phi_{tT}^\b\right]$ has the following Wiener chaos expansion {\SN ref:}{\MN \cite[(8.1)]{CSZ19b} (see also \cite[Theorem 7.26]{J97} for Wiener chaos expansion)}: For $t>0$, $x\in \R^2$\begin{align*}
%\DE_x\left[\Phi_{tT}^\b\right]&=\sum_{n=0}^\infty I_{tT,\b}^{(n)}(x)\\ 
%&=\sum_{n=0}^\infty \b^{n}\int_{0<{\SN s_1<\cdots<s_n}<tT}\int_{(\R^2)^n}\xi_1(\dd s_1,\dd y_1)\cdots \xi_1(\dd s_{n},\dd y_n)\prod_{i=1}^n\rho_{s_i-s_{i-1}}(y_{i-1},y_{i}),
%\end{align*}
%where $\displaystyle I_{tT,\b}^{(0)}(x)=\IE\left[\DE_x\left[\Phi^\b_{tT}\right]\right]=1$ and we set $x_0=x$, $s_0=0$ and $\rho_t(x)=(2\pi t)^{-1}e^{-|x|^2/(2t)}$ is the heat kernel and $\rho_t(x,y)=\rho_t(x-y)$.
%{\MN Hypercontractivity is used in Lemma \ref{lem:firstStep}. So Wiener chaos expansion has to be left. }

First, we remark that  for $x,y\in \R^2$
\begin{align}
\IE \left[\DE_x[\Phi^\b_{tT}]\DE_y[\Phi^\gamma_{tT}]\right]%&=\DE^{\otimes 2}\left[\exp\left(\b^2\int_0^tV(B_s-\tilde{B}_s)\dd s\right)\right]\\
%&=1+\sum_{n=1}^\infty \frac{\b^{2n}}{n!}\DE\left[\left(\b^2\int_0^T V(B_s-\tilde{B}_s)\dd s\right)^n\right]\\
%&=1+\sum_{n=1}^\infty \b^{2n}\DE^{\otimes 2}\left[\int_{0<s_1<\cdots<s_n<T} V(B_s-\tilde{B}_s)\dd s_1\cdots \dd s_n\right],
%&=\sum_{n=0}^\infty \IE\left[I_{tT,\b}^{(n)}(x)I_{tT,\gamma}^{(n)}(y)\right]\notag\\
&%&\hspace{-5em}
=1+\sum_{n=1}^\infty \b^{n}\gamma^n\int_{0<s_1<\cdots<s_n<tT}\int_{\left(\R^2\right)^n}\prod_{i=1}^n\left(V(\sqrt{2}x_i)\rho_{s_i-s_{i-1}}(x_{i-1},x_{i})\right)\dd \mathbf{s} \dd\mathbf{x}\label{eq:L2decomn}
\end{align}
where   $\IE$ and $\IP$ denote the expectations and probability with respect to the white noise $\xi$ and we set $x_0=\frac{x-y}{\sqrt{2}}$, $s_0=0$ and $\dd\mathbf{s}=\dd s_1\cdots \dd s_n$, $\dd \mathbf{x}=\dd x_1\cdots \dd x_n$.
%$B$, $\tilde{B}$ are independent Brownian motions starting from $0$ and $\IE^{\otimes 2}$ denotes the expectation in $B$ and $\tilde{B}$.
%\end{comment}
This representation is obtained from  the general property of the white noise:
\begin{align}
\IE\left[ \exp{\left(\int_{0}^t \int_{\R^2} f(t,x)\xi(\dd s, \dd x) \right)}\right]= \exp{\left(\frac{1}{2}\int_{0}^t\int_{\R^2} f(t,x)^2 \dd s \dd x \right)}.\label{eq:L2white}
\end{align}
Indeed,  we have
\aln{
  \IE \left[\DE_x[\Phi^\b_{tT}]\DE_y[\Phi^\gamma_{tT}]\right]&= \DE_{x}\otimes \DE_{y}\left[\exp\left(\int_0^t ( \b\phi(B_s-y)+\gamma\phi(\tilde{B}_s-y))^2 \dd s \dd y -  \frac{(\b^2+\gamma^2) V(0)t}{2}\right)\right] \notag\\
  &=\DE_x\otimes \DE_y\left[\exp\left(\b\gamma\int_0^t V(B_s-\tilde{B}_s)\dd s\right)\right]\notag\\
  &= \DE_{\frac{x-y}{\sqrt{2}}}\left[\exp\left(\b \gamma\int_0^t V(\sqrt{2} B_s)\dd s\right)\right],\notag
}
%where we have used $B_s-\tilde{B}_s \overset{(d)}{=} \sqrt{2}B_s$ for two independent Brownian motions $B_s$ and $\tilde{B}_s$.
and  the Taylor expansion $e^x=\sum_{n=0}^\infty \frac{x^n}{n!}$ gives \eqref{eq:L2decomn}% \eqref{eq:L2decomn} is equal to
\begin{comment}
  \al{
&\quad 1+\sum_{n=1}^\infty \frac{(\b\gamma)^{n}}{n!}\DE_{\frac{x-y}{\sqrt{2}}}\left[\left(\int_0^T V(\sqrt{2}B_s)\dd s\right)^n\right]\\
&=1+\sum_{n=1}^\infty \b^{n} \gamma^n\DE_{x_0}\left[\int_{0<s_1<\cdots<s_n<T} V(\sqrt{2}B_s)\dd s_1\cdots \dd s_n\right]\\
&=1+\sum_{n=1}^\infty \b^{n}\gamma^n\int_{0<s_1<\cdots<s_n<tT}\int_{\left(\R^2\right)^n}\prod_{i=1}^n\left(V(\sqrt{2}x_i)\rho_{s_i-s_{i-1}}(x_{i-1},x_{i})\right)\dd \mathbf{s} \dd\mathbf{x},
}
where $\IE$ and $\IP$ denote the expectations and probability with respect to the white noise $\xi$ and we set $x_0=\frac{x-y}{\sqrt{2}}$ and $\dd\mathbf{s}=\dd s_1\cdots \dd s_n$, $\dd \mathbf{x}=\dd x_1\cdots \dd x_n$.
\end{comment}  

\begin{lemma}\label{lem:L2bdd}
 Suppose $\hat{\b},\hat{\gamma}\in (0,1)$ and fix $t>0$. Then,
  \begin{align}
&\lim_{{ \e\to 0}}\IE \left[\DE[\Phi_{tT_\e}^{{ \b_\e}}]\DE[\Phi_{tT_\e}^{{\gamma_\e}}]\right]=\frac{1}{1-\hat{\b}\hat{\gamma}},\label{eq:L2bdd}\\
&\sup_{\e\leq 1}\sup_{s\leq t T_\e}\IE \left[\DE_{0,0}^{s,0}\left[\Phi^{{ \b_\e}}_s\right]^2\right]<\infty,\label{eq:L2ppbdd}
\end{align}
where $\DP_{0,x}^{t,y}$ and $\DE_{0,x}^{t,y},$ denote the probability measure and expectation of the Brownian bridge from $(0,x)$ to $(t,y)$ in $\R^2$.
\end{lemma}
 \begin{rem}
The proof of Theorem 2.15 in \cite{CSZ17b}, where \eqref{eq:L2bdd} for $\b_\e=\gamma_\e$ was proved by reducing the problem to discrete directed polymers in random environment, can be mollified for $\b_\e\not=\gamma_\e$, but we will give a direct proof in this paper.
\end{rem}

\begin{rem}
We call \begin{align*}
\sZ_{0,x}^{t,y}={\sZ_{0,x;t,y}^{(\b)}:=}\DE_{0,x}^{t,y}\left[\Phi_t^{\b}\right]
\end{align*} 
the point-to-point partition function of continuum directed polymers.
\end{rem}
\begin{proof}[Proof of \eqref{eq:L2bdd}]
  We have from
\eqref{eq:L2decomn}
\begin{align}
&\IE \left[\DE\left[\Phi_{tT}^\b\right]\DE\Big[\Phi_{tT}^\gamma\Big]\right]= 1+\sum_{n=1}^\infty \b^{n}\gm^n\int_{0<s_1<\cdots<s_n<tT}\int_{\left(\R^2\right)^n}\prod_{i=1}^n\left(V(\sqrt{2}x_i)\rho_{s_i-s_{i-1}}(x_{i-1},x_{i})\right)\dd \mathbf{s} \dd\mathbf{x},\label{white noise exp}
\end{align}
with $x_0=0$. 
  We first consider the upper bound: $$\varlimsup_{\e\to 0}\IE \left[\DE[\Phi_{tT_\e}^\b]\DE[\Phi_{tT}^\gamma]\right]\leq \frac{1}{1-\hat{\b}\hat{\gamma}}.$$
Let us consider the function
\ben{\label{Def: rs}
  r_s=\sup_{x\in \R^2}\int_{\R^2}V(\sqrt{2}y)\rho_{s}(x,y)\dd y\geq 0.
  }Since $\int_\R \rho_s(x,y)\dd y=1$, $\sup_{s> 0} |r_s|\leq \|V\|_{\infty}$. Moreover, using $\int V(x) \dd x=1$, we obtain
{%\SN
\begin{align}
  r_s&= \frac{1}{4\pi s } \sup_{x\in \R^2}\int_{\R^2}V(y) e^{-\frac{|x-y|^2}{4s}}\dd y\leq \frac{1}{4\pi s } \sup_{x\in \R^2} \int_{\R} V(y)\dd y= \frac{1}{4\pi s }.\label{eq:rorder}
\end{align}
}

Hence, \eqref{white noise exp} is bounded from above by
\begin{align*}
  &%1+\sum_{n=1}^\infty {\b}^n{\gm}^n\int_{T_n}\int_{\left(\R^2\right)^n}\prod_{i=1}^n\left(V(\sqrt{2}x_i)\rho_{s_i-s_{i-1}}(x_{i-1},x_{i})\right)\dd \mathbf{s} \dd\mathbf{x}\notag\\
  1+\sum_{n=1}^\infty \b^{n}\gm^n \left(\int_0^{tT}r_s\dd s\right)^{n}\\
  &\leq 1+\sum_{n=1}^\infty \b^{n}\gm^n \left(\frac{\log (tT)}{4\pi  }+\|V\|_{\infty}\right)^{n}=1+\sum_{n=1}^\infty \hat{\b}^{n}\hat{\gm}^n \left(1+\frac{\log{t}+4\pi \|V\|_{\infty}}{\log{T}  }\right)^{n}\notag\\
  & \to \sum_{n=0}^\infty  (\hat{\b}\hat{\gm})^n= \frac{1}{1-\hat{\beta}\hat{\gamma}},
\end{align*}
as $\e\to 0$, where the convergence is absolute since $\hat{\b}\hat{\gamma}<1$.

%Since the joint law of $\left(\frac{B+\tilde{B}}{\sqrt{2}},\frac{B-\tilde{B}}{\sqrt{2}}\right)$ is the same as the one of $(B,\tilde{B})$, we have  that \begin{align}
%&\DE^{\otimes 2}\left[\int_{0<s_1<\cdots<s_n<T} V(B_s-\tilde{B}_s)\dd s_1\cdots \dd s_n\right]\notag,
%\end{align}
%where
% Then, we have that \begin{align*}
%\int_{T_n\backslash \tilde{T}_n}\int_{\left(\R^2\right)^n}\prod_{i=1}^n\left(V(\sqrt{2}x_i)\rho_{s_i-s_{i-1}}(x_{i-1},x_{i})\right)\dd \mathbf{s} \dd\mathbf{x}\leq n\int_0^{1}r_s\dd s\left(\int_0^{tT}r_s\dd s\right)^{n-1}.
%\end{align*}

%where ${\rm Supp } V$ is the support of the function $V$ and we have used $e^{\frac{|x-y|^2}{4s}}-1\leq C s^{-1}$ with some constant $C$ independent of $x,y\in \R,\,s>1$ in the last line.

%Then, it is easy to see that \begin{align*}
%\int_{\R^2}V(\sqrt{2}y)p_{t}(y-x)\dd y=\int_{\R^2}\frac{1}{2\pi t }\exp\left(-\frac{|y-x|^2}{2t}\right)V(\sqrt{2}y)dy=
%\end{align*}

Next, we consider the lower bound. Let %$\mathbf{T}_n=\{0<s_1<\cdots<s_n<tT\}$,

$$\mathbf{T}_n=\{0<s_1<\cdots<s_n<tT,\, s_i-s_{i-1}>1,\,\forall i\in\{1,\cdots, n\} \}.$$ Then, since each term is non-negative, it is enough to show for fixed $L\in\N$,
\begin{align}
 \varliminf_{\e\to 0} \left(1+\sum_{n=1}^L \b^{n}\gm^n\int_{\mathbf{T}_n}\int_{\left(\R^2\right)^n}\prod_{i=1}^n\left(V(\sqrt{2}x_i)\rho_{s_i-s_{i-1}}(x_{i-1},x_{i})\right)\dd \mathbf{s} \dd\mathbf{x}\right)\geq  \sum_{n=0}^L (\hat{\b}\hat{\gamma})^n.\label{eq:l2tildeT}
\end{align}

For fixed $n\in\N$, we can find that
\begin{align}
&\int_{\mathbf{T}_n}\int_{\left(\R^2\right)^n}\prod_{i=1}^n\left(V(\sqrt{2}x_i)\rho_{s_{i}-s_{i-1}}(x_{i-1},x_{i})\right)\dd \mathbf{s} \dd\mathbf{x}\notag\\
%&=\int_{\mathbf{T}_n^\e}\int_{\left(B(0,R_V)\right)^n}\prod_{i=1}^n\left(V(\sqrt{2}x_i)p_{s_i}(x_i-x_{i-1})\right)\dd \mathbf{s} \dd\mathbf{x}\\
&=\int_{\mathbf{T}_n}\int_{\left(\R^2\right)^n}\prod_{i=1}^n\left(\frac{1}{2\pi (s_i-s_{i-1})}V(\sqrt{2}x_i)-\bar{r}_{s_i-s_{i-1}}(x_{i-1},x_{i})\right)\dd \mathbf{s}\dd \mathbf{x}\notag\\
&=\int_{\mathbf{T}_n}\prod_{i=1}^n \frac{1}{4\pi (s_i-s_{i-1})}\dd \mathbf{s}+A_\e^n,\label{poppoppo}
\end{align}
where 
\begin{align*}
\bar{r}_s(x,y)&=\frac{1}{2\pi s}V(\sqrt{2}y)-V(\sqrt{2}y)\rho_{s}(x,y)\\
&=\frac{1}{2\pi s }\left(1-\exp\left(-\frac{|y-x|^2}{2s}\right)\right)V(\sqrt{2}y)\geq 0,
\end{align*}
\begin{align*}
  A_\e^n&= \sum_{k=1}^n (-1)^k \sum_{j_1<j_2<\cdots<j_k}\int_{\mathbf{T}_n}\int_{\left(\R^2\right)^n}\prod_{i\not=j_1,\cdots,j_k}\frac{V(\sqrt{2}x_i)}{4\pi (s_i-s_{i-1})} \prod_{j=j_1,\cdots,j_k}\bar{r}_{s_{j}-s_{j-1}}(x_{j-1},x_j)\dd \mathbf{s} \dd \mathbf{x}.
  \end{align*}
Let $D_V=\overline{\{x\in \R^d:~V(\sqrt{2}x)\neq 0\}}$, which is compact. Note that
\al{
  \sup_{x\in D_V}\int_{\R^2}\bar{r}_s(x,y)\dd y&= (2\pi s)^{-1} \sup_{x\in D_V}\int_{\R^2} \left(1-\exp\left(-\frac{|y-x|^2}{2s}\right)\right)V(\sqrt{2}y) \dd y\\
  &\leq (2\pi s)^{-1} |D_V| \sup_{x,y \in D_V} \left(1-\exp\left(-\frac{|y-x|^2}{2s}\right)\right) \\
  &\leq C(s^{-2}\wedge 1),
}
with some $C=C(V)\geq 1\lor \|V\|_\infty$. In particular, we have \begin{align*} 
|A_\e^n|
  &\leq C^{n+1}\sum_{k=1}^n \sum_{j_1<j_2<\cdots<j_k}\int_{\mathbf{T}_n}\int_{\left(\R^2\right)^n}\prod_{i\not=j_1,\cdots,j_k}\frac{\dd s_i}{s_i-s_{i-1}} \\
&\leq C^{n+1} \sum_{k=1}^n n^k \left(\log{( { tT})}\right)^{n-k}\\
&\leq (Cn)^{n+1}(\log{( { tT})})^{n-1}.
\end{align*}
Thus, we have for any fixed $L>0$\begin{align}
\sum_{n=1}^L \beta^{n}\gm^nA_\e^n= \sum_{n=1}^L \hat{\beta}^{n}\hat{\gm}^n \left(\frac{4\pi}{\log{T}}\right)^n A_\e^n \to 0,\label{eq:L2bdderror}
\end{align}
as $\e \to 0$. Also,
\begin{align*}
\int_{\mathbf{T}_n}\prod_{i=1}^n \frac{1}{4\pi (s_i-s_{i-1})}\dd \mathbf{s}\geq \left(\int_{1}^{ \frac{tT}{n}}\frac{1}{4\pi s}\dd s\right)^n=\left(\frac{\log{(tT/n)}}{4\pi}\right)^n,  
%\leq \left(\int_{1}^{tT}\frac{1}{4\pi s}\dd s\right)^M,
\end{align*}
and hence we have \begin{align}
\varliminf_{\e\to 0} \sum_{n=1}^L \b^{n}\gm^n\int_{\mathbf{T}_n}\prod_{i=1}^n \frac{1}{4\pi (s_i-s_{i-1})}\dd \mathbf{s}\geq \sum_{n=1}^L \hat{\b}^n\hat{\gm}^n.\label{eq:L2bddlead}
\end{align}
Then, \eqref{poppoppo}, \eqref{eq:L2bdderror} and \eqref{eq:L2bddlead} yield \eqref{eq:l2tildeT}.

%\eqref{eq:l2tildeT} and \eqref{eq:L2bdd} are concluded from \eqref{eq:L2mainl} and \eqref{eq:L2bddrem}.

%\begin{align}
%\sum_{n=1}^L \b^{n}\gm^n\int_{\mathbf{T}_n}\int_{\left(\R^2\right)^n}\prod_{i=1}^n\left(V(\sqrt{2}x_i)\rho_{s_i-s_{i-1}}(x_{i-1},x_{i})\right)\dd \mathbf{s} \dd\mathbf{x}\to \sum_{n=1}^L \hat{\b}^n\hat{\gamma}^n\label{eq:L2mainl}
%\end{align}
%as $T\to \infty$.

%On the other hand, we can easily obtain from \eqref{eq:rorder} that \begin{align*}
%\int_{\widetilde{T}_M}\int_{\left(\R^2\right)^M}\prod_{i=1}^M\left(V(\sqrt{2}x_i)\rho_{s_i}(x_{i-1},x_{i})\right)\dd \mathbf{s} \dd\mathbf{x}&\leq \left(\int_{1}^{tT}r_s\dd s\right)^M\\
%&\leq \left(\frac{\log tT}{4\pi }+C\right)^M.
%\end{align*}
%Thus, we know that \begin{align}
%\varlimsup_{L\to \infty}\varlimsup_{T\to\infty}\sum_{n=L+1}^\infty \b^n\gm^n\int_{\mathbf{T}_n}\int_{\left(\R^2\right)^n}\prod_{i=1}^n\left(V(\sqrt{2}x_i)\rho_{s_i-s_{i-1}}(x_{i-1},x_{i})\right)\dd \mathbf{s} \dd\mathbf{x}=0.\label{eq:L2bddrem}
%\end{align}
% Combining \eqref{eq:L2decomn}, \eqref{eq:L2bddapp}, \eqref{eq:L2bdderror}, \eqref{eq:L2bddlead}, and \eqref{eq:L2bddrem} yields \eqref{eq:L2bdd}. 

\end{proof}
\begin{proof}[Proof of \eqref{eq:L2ppbdd}]
  %Almost the same argument as \eqref{eq:L2bdd} can be applied so we will omit the detail with some comments.
  % For simplicity, we set $t=1$, though the same argument work for general $t>0$.a
  We obtain by  the same manner as \eqref{eq:L2decomn} that
\begin{align}
&\IE \left[\DE_{0,0}^{tT,0}\left[\Phi^\b_{tT}\right]^2\right]= \DE_{0,0}^{tT,0}\left[ \exp{\left( \b^2\int_0^{t T} V(\sqrt 2  B_u) \dd u\right)}\right] \notag\\
&=1+\sum_{n=1}^\infty \b^{2n}\int_{0<s_1<\cdots<s_n<s}\int_{\left(\R^2\right)^n}\left(\prod_{i=1}^nV(\sqrt{2}x_i)\rho_{s_{i}-s_{i-1}}(x_{i-1},x_{i})\right)\frac{\rho_{s-s_n}(x_n)}{\rho_{tT}(0)}\dd \mathbf{s}\dd \mathbf{x},\label{eq:L2ppdecom}
\end{align}
where we  use the orthogonal transformation invariance  of Brownian bridges.  We have for $s,t>0$, by Markov property of Brownian motions,
\al{
   \int_{\R}\rho_{s}(x,y) \rho_t(y) V(\sqrt{2} y)\dd y&\leq \|V\|_\infty  \int_{\R}\rho_{s}(x,y) \rho_t(y) \dd y\leq \|V\|_{\infty}  \rho_{s+t}(x),  
}
 and  by $\int_\R V(\sqrt{2} y)\dd y=1/2$,
\aln{
  \int_{\R}\rho_{s}(x,y) \rho_t(y) V(\sqrt{2} y)\dd y&= \frac{1}{4\pi s t}\int_\R V(\sqrt{2} y) \exp{\left(-\frac{|x-y|^2}{2s}-\frac{|y|^2}{2t}\right)}\notag\\
  &=\rho_{s+t}(x)\,\frac{s+t}{2\pi s t} \int_\R V(\sqrt{2} y) \exp{\left(-\frac{(s+t)|y+*|^2}{2st}\right)}\dd y\notag\\
  &\leq \rho_{s+t}(x)\,\frac{s+t}{2\pi s t} \int_\R V(\sqrt{2} y) \dd y=\frac{s+t}{4\pi st} \rho_{s+t}(x).\label{pp recursive inequality}
  }
Putting things together with $C=4\pi \|V\|_\infty$, we have
\al{
  \int_{\R}\rho_{s}(x,y) \rho_t(y) V(\sqrt{2} y)\dd y\leq \frac{1}{4\pi }\left(C \land \frac{s+t}{st} \right)  \rho_{s+t}(x).
}
 Using this successively, we can bound each term of \eqref{eq:L2ppdecom} as
\aln{
 &  \b^{2n}\int_{0<s_1<\cdots<s_n<{tT}}\int_{\left(\R^2\right)^n}\left(\prod_{i=1}^nV(\sqrt{2}x_i)\rho_{s_{i}-s_{i-1}}(x_{i-1},x_{i})\right)\frac{\rho_{s-s_n}(x_n)}{\rho_{tT}(0)}\dd \mathbf{s}\dd \mathbf{x}\notag\\
  &\leq \left(\frac{\b^{2}}{4\pi}\right)^n \int_{0<s_1<\cdots<s_n<{tT}}\prod_{i=1}^n\left(C \land \frac{{tT}-s_{i-1}}{(s_i-s_{i-1})({tT}-s_{i})}\right) \dd \mathbf{s}\notag\\
    &= \left(\frac{\b^{2}}{4\pi}\right)^n \int_{0<s_1<\cdots<s_n<{tT}} \prod_{i=1}^n\left(C\land \left(\frac{1}{s_i-s_{i-1}}+\frac{1}{{tT}-s_{i}}\right)\right) \dd \mathbf{s},\label{PP partition estimate}
}
where we set $s_0=0$ and $s_{n+1}={tT}$ and we have used $\frac{1}{{tT}-s_i}+\frac{1}{s_i-s_{i-1}}=\frac{{tT}-s_{i-1}}{(s_i-s_{i-1})({tT}-s_{i})}$ in the last line.
% and $\frac{1}{s_{i+1}-s_i}$
We write $\log_+(x)=\log{x}\lor 0$ and $C_1=2C$. We use the following integral estimate: for $s<{tT}$ and $k\geq 0$,
  \al{
   & \int_{s}^{tT} (C_1+\log_+{({tT}-t)})^k \left(C\land \left(\frac{1}{t-s}+\frac{1}{{tT}-t}\right)\right) \dd t\\
    &\leq 2 C(C_1+\log_+{({tT}-s)})^k+ \int_{s+1}^{{tT}-1} (C_1+\log{({tT}-t)})^k  \left(\frac{1}{t-s}+\frac{1}{{tT}-t}\right) \dd t\\
    &\leq C_1(C_1+\log_+{({tT}-s)})^k+ (C_1+\log_+{({tT}-s)})^k \int_{s+1}^{{tT}-1}  \frac{1}{t-s} \dd t - (k+1)^{-1}\left[ (C_1+\log_+{({tT}-t))^{k+1}}\right]_{s+1}^{{tT}-1}\\
    &\leq  C_1(C_1+\log_+{({tT}-s)})^k+ (C_1+\log_+{({tT}-s)})^k \log_+{({tT}-s)}+ (k+1)^{-1} (C_1+\log_+{({tT}-s))^{k+1}}\\
    &= \frac{k+2}{k+1}(C_1+\log_+{({tT}-s)})^{k+1}.
    }
  Using this, \eqref{PP partition estimate} can be successively bounded from above as
  \al{
    &\quad \int_{0<s_1<\cdots<s_n<{tT}} \prod_{i=1}^n\left(C\land \left(\frac{1}{s_i-s_{i-1}}+\frac{1}{{tT}-s_{i}}\right)\right) \dd \mathbf{s}\leq (n+1)(C_1+\log{{tT}})^n.
    %2 \int_{0<s_1<\cdots<s_n<{tT}} \prod_{i=1}^{n-1}\left(C\land \left(\frac{1}{s_i-s_{i-1}}+\frac{1}{{tT}-s_{i}}\right)\right) \times \left( C_1 + \log_+ ({tT}-s_{n-1})\right)\dd \mathbf{s}\\
   % &\leq \int_{0<s_1<\cdots<s_n<{tT}} \prod_{i=1}^{n-2}\left(C\land \left(\frac{1}{s_i-s_{i-1}}+\frac{1}{{tT}-s_{i}}\right)\right) \times \left( (C_1)^2 +4(C_1) \log_+ ({tT}-s_{n-2})+3(\log_+ ({tT}-s_{n-2}))^2\right)\dd \mathbf{s}\\
  }
  Together with  \eqref{eq:L2ppdecom} and  \eqref{PP partition estimate}, using $\b=\hat{\beta}\sqrt{\frac{4\pi}{\log{T}}}$ with $\hat{\beta}<1$, we have
  \al{
 \varlimsup_{\e\to 0}   \IE \left[\DE_{0,0}^{{tT},0}\left[\Phi^\b_{tT}\right]^2\right]& \leq  \varlimsup_{\e\to 0} \sum_{n=0}^\infty \left(\frac{\beta^2}{4\pi}\right)^n (n+1)(C_1+\log{{tT}})^n= \sum_{n=0}^\infty (n+1)\hat{\beta}^{2n}<\infty.
  }
%The equation
%\begin{align*}
%\frac{s}{s_1(s_2-s_1)\cdots (s_{n}-s_{n-1})(s-s_n)}=\prod_{i=1}^n\left(\frac{1}{s_i-s_{i-1}}+\frac{1}{s-s_{i}}\right)
%\end{align*}
%is useful to estimate \eqref{eq:L2ppdecom}.  
\end{proof}

%The following lemma is a consequence of Lemma \ref{lem:L2bdd} combining hypercontractivity of chaos expansion:
\begin{lemma}\cite[(5.11)]{CSZ20}\label{lem:lp}
Fix $\hat{\b}\in (0,1)$. Then, there exists $p_{\hat{\b}}>2$ such that for any $2\leq p<p_{\hat{\b}}$  and for $t\geq 0$\begin{align*}
\varlimsup_{\e\to 0}\IE\left[\DE_x\left[\Phi_{tT}^{{ \b_\e}}\right]^p\right]<\infty.
\end{align*}
\end{lemma}
%The following lemma is a corollary of \eqref{eq:L2ppbdd}.
%The following lemma is a consequence of the %argument in \eqref{eq:L2ppbdd}.
\begin{lemma}\label{lem:p2pge}
Suppose $\hat{\b}\in (0,1)$ and fix $t>0$. Then, \begin{align}
%&\lim_{T\to \infty}\IE \left[\DE\left[\Phi_{tT}^\b\right]\DE\Big[\Phi_{tT}^\gamma\Big]\right]=\frac{1}{1-\hat{\b}\hat{\gamma}},\label{eq:L2bdd}\\
&\sup_{\e \leq 1}\sup_{x\in \R^2}\sup_{s\leq tT_\e}\DE_{0,0}^{s,x}\left[\exp\left(\beta_\e^2\int_0^s V(\sqrt{2}B_u)\dd u\right)\right]<\infty.\label{eq:p2pge}
\end{align}
\end{lemma}
%{\SN I moved the remark here since it is used in the proof below.
%}
\begin{proof}
   By \eqref{eq:L2ppbdd}, \eqref{eq:L2ppdecom}, for $s\leq t T$,
   \al{
     \DE_{0,0}^{s,z}\left[ e^{\b^2\int_0^{s} V(\sqrt 2  B_u) \dd u} \right]&= \IE[\DE_{0,0}^{s,z} [\Phi^\b_{s}]\DE_{0,0}^{s,0} [\Phi^\b_{s}]]\\
     &\leq \IE[\DE_{0,0}^{s,0} [\Phi^\b_{s}]^2]= \DE_{0,0}^{s,0}\left[ e^{\b^2\int_0^{s} V(\sqrt 2  B_u) \dd u} \right]\\
     &\leq \sup_{\e \leq 1}\left[\exp\left(\beta_\e^2\int_0^{tT_\e} V(\sqrt{2}B_u)\dd u\right)\right]<\infty,
   }
   where we have used the remark below.
  %We omit the detail with a remark \begin{align*}
%&\DE_{0,0}^{s,x}\left[\exp\left(\beta^2\int_0^s V(\sqrt{2}B_u)\dd u\right)\right]=\DE_{0,0}^{s,0}\left[\exp\left(\beta^2\int_0^s V\left(\sqrt{2}\left(B_u-\frac{ux}{s}\right)\right)\dd u\right)\right]\\
%&=1+\sum_{n=1}^\infty \b^{2n}\int_{0<s_1<\cdots<s_n<s}\int_{\left(\R^2\right)^n}\left(\prod_{i=1}^nV\left(\sqrt{2}\left(x_i-\frac{s_ix}{s}\right)\right)\rho_{s_{i}-s_{i-1}}(x_{i-1},x_{i})\right)\frac{\rho_{s-s_n}(x_n)}{\rho_s(0)}\dd \mathbf{s}\dd \mathbf{x}.
   %  \end{align*}
 %  Hence, for $x\in \R^2$ and $s\geq 1$,
%   $$\DE_{0,0}^{s,x}\left[\exp\left(\beta^2\int_0^s V(\sqrt{2}B_u)\dd u\right)\right]\leq $$
\end{proof}
\begin{rem}
By the shear invariance of environment, we have that \begin{align*}
\DE_{0,x}^{t,y}\left[\Phi_t^\b\right]\eqlaw \DE_{0,0}^{t,0}\left[\Phi_t^\b\right]
\end{align*}
for any $t>0$ and $x,y\in\R^2$.
\end{rem}

We end this subsection by %\subsection{Negative moment}\label{sec:negativemoment}
 presenting the boundedness of negative moments of  partition functions:

\iffalse{
   \begin{align}
\overline{\sZ}_{t,r}(z)&={ \overline{\sZ}_{t,r}^{\b}(z)}=\DE_{z}\left[\Phi_{t}^\b:\mathtt{F}_{t,r}(B,z) \right],\label{eq:barZ}
  \end{align}
 where $\mathtt{F}_{t,r}(B,z)$ is the event that Brownian motion $B$ does not escape from the open ball $B(z,r)=\{x\in \R^2:|x-y|<r\}$  up to times $t$, \begin{align*}
\mathtt{F}_{t,r}(B,z)=\{B_s\in B(z,r)\ \text{for any }s\in [0,t]\}.
\end{align*}
}\fi

% Given $z\in\R^2$ and  $t>0,\,r>0$, we denote by $\mathtt{C}_{t,r}(z)$ the cylinder in $[0,\infty)\times \R^2$
%  \begin{align}
%{\mathtt{C}_{t,r}(z)=[0,t]\times B(z,r),} \label{Def:CUBE}
%\end{align}
%and denote by  We define the constrained partition function as
 
\begin{lemma}\label{lem:negativemoment}\cite[(5.12), (5.13), (5.14)]{CSZ20}
Let $\hat{\b}\in (0,1)$ and fix $t>0$. For any $p\geq 0$   and $x\in\R^2$,
\begin{align*}
&\sup_{s\in [0,t]}\IE\left[\left({\sZ}_{sT_\e}^{{\beta_\e}}(x)\right)^{-p}\right]<\infty.%, \quad \sup_{(s,r){\SN \in} \Lambda_t}\IE\left[\left(\overline{\sZ}_{sT_\e,r\sqrt{T_\e}}^{{\beta_\e}}(x)\right)^{-p}\right]<\infty,
\end{align*}
%where $\Lambda_t$ is a  subset of $[0,t]\times \R^2$ such that $\DP_x\left(\mathtt{F}_{s,r}(B,x)\right)\geq \frac{1}{2}$.
\end{lemma}
%{\SN Can we remove the stuff below?}{\MN Yes!}
%\begin{rem}
%In \cite{CSZ20}, the boundedness of negative moments for partition functions of continuum directed polymers in ``constrained" environment  is concluded from the concentration inequality for partition functions \cite[Proposition 5.5]{CSZ20}.   The same argument  shows the concentration inequality for a partition function. Indeed, we already know the boundedness of $L^2$-moment of ``$\nabla_{\xi} \log {\sZ}_{s}(x)$" which follows from Lemma \ref{lem:L2bdd} for $\hat{\beta}$ and  $\IP\left({\sZ}_{s}(x)\geq \frac{1}{4}\right)\geq c_{\hat{\b}}$ for some constant $c_{\hat{\b}}>0$ which follows from the Payley-Zygmund inequality. 
%\end{rem}

\subsection{Local limit theorem}
%\begin{lemma}\label{lem:L2ppbdd}
%Suppose $\hat{\b}\in (0,1)$. Then, \begin{equation}\label{eq:L2ppbdd}
%\sup_{t\leq T}\IE \left[\DE_{0,0}^{t,0}\left[\Phi_t\right]^2\right]<\infty.
%\end{equation}
%\end{lemma}

%\begin{proof}
%By the usual $L^2$-argument, we know that \begin{align*}
%\IE \left[\DE_{0,0}^{t,0}\left[\Phi_t\right]^2\right]&=\left(\DE_{0,0}^{t,0}\right)^{\otimes 2}\left[\exp\left(\b^2\int_0^tV(B_s-\tilde{B}_s)\dd s\right)\right]\\
%&=1+\sum_{n=1}^\infty \frac{\b^{2n}}{n!}\left(\DE_{0,0}^{t,0}\right)^{\otimes 2}\left[\left(\b^2\int_0^t V(B_s-\tilde{B}_s)\dd s\right)^n\right]\\
%&=1+\sum_{n=1}^\infty \b^{2n}\left(\DE_{0,0}^{t,0}\right)^{\otimes 2}\left[\int_{0<s_1<\cdots<s_n<t} V(B_s-\tilde{B}_s)\dd s_1\cdots \dd s_n\right],
%\end{align*}
%where $\left(\DE_{0,0}^{t,0}\right)^{\otimes 2}$ is the probability measure of  two  independent Brownian  bridges $(B,\tilde{B})$ from the origin at  time $0$ to the origin at time $t$. 
%The joint law of $\left(\frac{B+\tilde{B}}{\sqrt{2}},\frac{B-\tilde{B}}{\sqrt{2}}\right)$ is the same as the one of $(B,\tilde{B})$ so we have  that \begin{align}
%&\left(\DE_{0,0}^{t,0}\right)^{\otimes 2}\left[\int_{0<s_1<\cdots<s_n<t} V(B_s-\tilde{B}_s)\dd s_1\cdots \dd s_n\right]\notag\\
%&=\int_{0<s_1<\cdots<s_n<t}\int_{\R^n}\prod_{i=1}^n\left(V(\sqrt{2}x_i)p_{s_i}(x_i-x_{i-1})\right)\frac{p_{t-s_n}(x_n)}{p_t(0)}\dd \mathbf{s} \dd\mathbf{x}\label{eq:L2decomn}
%\end{align}
%\end{proof}

%Recall the definition of $\overset{\leftarrow}{\sZ}_{T,\ell}(z)$ in \eqref{eq:timeRevPartitionf}.

In this subsection, we give an estimate of local limit theorem for partition functions.

To describe the statement, we introduce the time-reversed partition function of time horizon $\ell$, $\overleftarrow{\sZ}^\b_{T,\ell}(z)$ :
\begin{equation}% \label{eq:timeRevPartitionf}
\overleftarrow{\sZ}^\b_{t,\ell}(z) = \DE_{z} \left[\exp\left\{\b\int_{t-\ell}^t \int_{\mathbb R^2} \phi(B_{t-s}-y)  \xi(\dd s,\dd y) -\frac{\b^2  V(0) \ell}{2}\right\} \right].\notag
\end{equation}

\begin{theorem}[Local limit theorem for polymers] \label{th:errorTermLLT}
{ Fix $t>0$.} Let $0<\ell_{{T_\e}}^a<\ell_{{T_\e}}^b<{ L({ T_\e})}\leq t{T_\e}$ be functions with $\displaystyle \lim_{\e\to0 }\ell_{{ T_\e}}^a\to \infty$, $\displaystyle \lim_{\e\to 0}\frac{\ell_{T_{{ T_\e}}}^b}{L(T_{{ T_\e}})}=0$, { $\displaystyle \lim_{\e\to 0}\frac{\log L({ T_\e})}{\log {T_\e}}=1$}. Then,  for all ${\hat{\b}}<1$ there exists $C=C(\hat{\b})$ such that for all positive $\ell>0$ verifying $\ell_{{ T_\e}}^a\leq \ell\leq \ell_{{ T_\e}}^b$ and for all $x,y\in\mathbb R^d$,
\begin{align*}
&\IE\left(\DE_{0, 0}^{{ L({ T_\e})},x} [\Phi_{{ L({ T_\e})}}^{{ \b_\e}}] - \sZ_\ell^{{\b_\e}}( 0) \overleftarrow{\sZ}^{{ \b_\e}}_{{ L({ T_\e})},\ell}( x)\right)^2 \\
&\leq \begin{cases}
C\frac{\ell}{{ L({T_\e})}}+C\b_\e^2\left( \log \frac{{ L({T_\e})}}{\ell}+\frac{|x|\log \ell}{{ L(T)}}+\frac{|x|^2 \ell}{{L(T)}^2}\right)\quad &|x|\leq \sqrt{{ L({T_\e})}\log { L({T_\e})}}\\
C&|x|\geq \sqrt{{ L({T_\e})}\log { L({ T_\e})}}
\end{cases}.
\end{align*}
\end{theorem}

\begin{rem}
Theorem states that the point-to-point partition function from $(0,x)$ to $({ L(T)},y)$ is approximated by the product of partition function from $(0,x)$ with length $\ell$ and time-reversed partition function from $({ L(T)},y)$ with length $\ell$ in $L^2$-sense. For $d\geq 3$, the reader may refer to \cite{CNN20,Si95,V06}.  
\end{rem}
%{\MN The statement in this  proposition is stronger than the one in Vargas (LLT holds if $l(T)=o(T)$.) Most parts in his proof for continuous case is omitted, so more rigorous estimate may be recovered in this note. Probably the difference from discrete case is the error term in LLT for transition probability.} 
%\begin{rem}
%Note that in addition to giving a control on the error term, the theorem states that the error term still vanishes for starting and terminal points that can be distant up to sub-linear scale (i.e.\ $\sqrt T |x-y| = o(T)$), provided that $\ell$ is chosen accordingly. This is a significant improvement compared to the result of \cite{V06,Si95}. %Another slight improvement is that we can let $\ell$ go up to $o(T)$ instead of $o(T^{1/2})$. %\CC{?? See if Khanin has put his paper}.
%\end{rem}
%Although we follow in our proof the two main steps of \cite{V06}, some new arguments are needed in order to obtain good and uniform estimates on the error. The main difference in our approach is the use of uniform bound on the second moments of point-to-point partition functions.
%\subsection{First Step}

\begin{notation}
  Fix $R_V>0$ such that  ${\rm supp}\,V\subset B(0,R_V)$.
  \end{notation}

The proof is composed of three steps.

\begin{lemma}[Step 1]\label{lem:LLT1}  Fix $t>0$. There exists a constant $C=C(\tilde{\b})>0$ such that  for $\ell_{{ T_\e}}^a\leq \ell\leq \ell_{{ T_\e}}^b$,
\[
\sup_{x\in\mathbb R^2} \IE\left[\left(\DE_{0,0}^{{ L({ T_\e})},x} [\Phi^\b_{{ L({ T_\e})}}] - \DE_{0,0}^{{ L({ T_\e})}, x}[\Phi^\b_{\ell}\Phi^\b_{{ L({ T_\e})}-\ell,{ L({ T_\e})}}]\right)^2\right] \leq C\b_\e^2{\log \frac{{L({ T_\e})}}{\ell}}, 
\]
where \begin{align*}
\Phi_{s,t}^\b=\exp\left(\b\int_s^t\int_{\R^2}\phi(y-B_u)\xi(\dd u,\dd y)-\frac{\b^2 V(0)(t-s)}{2}\right).
\end{align*}
\end{lemma}
\begin{proof}
Since $B_s^{(1)}-B_s^{(2)} \eqlaw \sqrt 2 B_s$ for two independent Brownian motions, by $1-e^{-x}\leq x$ for $x\geq 0$, we have
\begin{align*}
&\IE\bigg[\bigg(\DE_{0,0}^{{ L(T)}, x} [\Phi^\b_{L(T)} - \Phi^\b_{\ell}\Phi^\b_{{L(T)}-\ell,{ L(T)}}]\bigg)^2\bigg] \nonumber\\
& = \DE_{0,0}^{{ L(T)},0} \left[\mathrm e^{\b^2 \int_0^{{ L(T)}} V(\sqrt 2 B_s)ds} -\mathrm e^{\b^2 \int_0^{\ell} V(\sqrt 2 B_s)ds}\mathrm e^{\b^2 \int_{{ L(T)}-\ell}^{{ L(T)}} V(\sqrt 2 B_s)ds}\right]\nonumber\\
& \leq  \DE_{0,0}^{{ L(T)},0}\left[ e^{\b^2\int_0^{{ L(T)}} V(\sqrt 2  B_s)ds}\,\b^2\int_\ell^{{ L(T)}-\ell} V(\sqrt 2 B_s)\dd s\right]. 
\end{align*}
The last expectation equals
\begin{equation*}%\label{Application LLT KPZ}
  \begin{split}
 &  \b^2\int_{[\ell,L(T)-\ell]\times \mathbb R ^2} V(\sqrt 2 z)  \frac{\rho_s(z)\rho_{L(T)-s}(z)}{\rho_{L(T)}(0)} \DE_{0,0}^{s,z} \left[e^{\b^2\int_0^{s} V(\sqrt 2  B_u)\dd u} \right] \DE_{0,z}^{{ L(T)}-s,0} \left[ e^{\b^2\int_0^{tT-s} V(\sqrt 2  B_u)\dd u}\right]  \dd z \dd s\\
 & \leq \b^2\left( \sup_{s\leq { L(T)}} \sup_{|z|\leq R_V} \DE_{0,0}^{s,z}\left[ e^{\b^2\int_0^{s} V(\sqrt 2  B_u) \dd u} \right]\right)^2 \int_\ell^{{ L(T)}-\ell} \int_{\mathbb R^2} V(\sqrt 2 z)  \frac{\rho_s(z)\rho_{{L(T)}-s}(z)}{\rho_{{ L(T)}}(0)} \dd s \dd z,
  \end{split}
  \end{equation*}
where  the supremum on the last line is finite by Lemma~\ref{lem:p2pge}.
Finally, by \eqref{pp recursive inequality},{
\begin{align*}
  \int_\ell^{{L(T)}-\ell} \int_{\mathbb R^2} \frac{\rho_{s}(z)\rho_{{ L(T)}-s}(z)}{\rho_{{ L(T)}}(0)} V\left(\sqrt 2 z\right)\dd s \dd z
& \leq \frac{1}{4\pi} \int_\ell^{{ L(T)}-\ell}  \frac{{L(T)}}{s({ L(T)}-s)}\dd s \\
& \leq  \frac{1}{2\pi} \log \frac{{ L(T)}-\ell}{\ell}\leq  \log \frac{{ L(T)}}{\ell},
\end{align*}}
and the statement of the lemma follows.
\end{proof}

By translation invariance of Brownian bridge, Brownian motion and noise, we have
\begin{align*}
&\DE_{0,0}^{{ L(T)}, x}\left[\Phi^\b_{\ell}\Phi^\b_{{ L(T)}-\ell,{ L(T)}}\right]- \sZ_\ell^\b(0) \overleftarrow{\sZ}^\b_{{L(T)},\ell}( x)\\
&\eqlaw \DE_{0,0}^{{L(T)},0}\left[\Phi^\b_{\ell}\Phi^\b_{{L(T)}-\ell,{L(T)}}\right]\\
&-\DE_0\left[\Phi^\b_\ell\left(B_\cdot+\frac{ x}{{{L(T)}}}\cdot\right)\right]\DE_0\left[\exp\left(\b\int_{{ L(T)}-\ell}^{{L(T)}}\phi\left(B_{{L(T)}-s}+\frac{({L(T)}-s)x}{{{L(T)}}}-y\right)\xi(\dd s,\dd y)-\frac{\b^2 V(0)\ell}{2}\right)\right],
\end{align*}
where $\dis B_\cdot+\frac{ x}{{{L(T)}}}\cdot$ is a Brownian motion with drift $\dis \frac{x}{{ L(T)}}$.

Define 
\begin{align*}
&A_{{ L(T)},\ell} := \DE_{0,0}^{{L(T)},0}\left[\Phi^\b_{\ell}\Phi^\b_{{ L(T)}-\ell,{ L(T)}}\right] - \DE_0\left[\Phi^\b_{\ell}\right] \overleftarrow{\sZ}^\b_{{L(T)},\ell}(0),\\
&B_{{ L(T)},\ell,x}:=\DE_0\left[\Phi^\b_{\ell}\right]-\DE_0\left[\Phi^\b_\ell\left(B_\cdot+\frac{ x}{{L(T)}}\cdot\right)\right].
\end{align*}

%{\MN There are errors in the proof. The estimate might be $C\sqrt{\frac{\ell}{T}}$.}
\begin{lemma}[Step 2]
There exists a constant $C=C(\hat{\b})$ such that for all positive $\ell>0$  with $\ell_{{ T_\e}}^a\leq \ell \leq \ell_{{T_\e}}^b$,
\begin{align*}
\IE\left[A_{{ L({ T_\e})},\ell}^2\right]\leq C\frac{\ell}{{L({ T_\e})}}.
\end{align*}
\end{lemma}
\begin{proof}
\eqref{eq:L2white} yields that
\begin{align*}
&\IE\left[\DE_{0,0}^{{ L(T)},0}\left[\Phi^\b_{\ell}\Phi^\b_{{ L(T)}-\ell,{L(T)}}\right]^2\right]=\DE_{0,0}^{{L(T)},0}\otimes \DE_{0,0}^{L(T),0}\left[\exp\left(\b^2\int_0^\ell V(B_s-\widetilde{B}_s)ds+\b^2\int_{L(T)-\ell}^{L(T)}V(B_s-\widetilde{B}_s)ds\right)\right]\\
&=\int_{\R^2\times \R^2}\dd x\dd y\, \rho_{\ell}(x)\rho_\ell(y)\frac{\rho_{{L(T)}-2\ell}(y-x)}{\rho_{{L(T)}}(0)}\int_{\R^2\times \R^2}\dd z\dd w\rho_{\ell}(z)\rho_\ell(w)\frac{\rho_{{L(T)}-2\ell}(z-w)}{\rho_{{L(T)}}(0)}\\
&\hspace{2.5em}\times \DE_{0,0}^{\ell,x}\otimes \DE_{0,0}^{\ell,z}\left[\exp\left(\b^2\int_0^\ell V(B_s-\widetilde{B}_s)ds\right)\right]
\DE_{0,0}^{\ell,y}\otimes \DE_{0,0}^{\ell,w}
\left[\exp\left(\b^2\int_0^\ell V(B_s-\widetilde{B}_s)ds\right)\right],
\intertext{and}
%&\IE\left[\DE_0\left[\Phi^\b_{\ell}\right]^2 \overleftarrow{\sZ}_{tT,\ell}(0)^2\right]=\DE\left[\exp\left(\b^2\int_0^\ell V(\sqrt{2}B_s)ds\right)\right]^2\\
&\IE\left[\DE_{0,0}^{{ L(T)},0}\left[\Phi^\b_{\ell}\Phi_{{L(T)}-\ell,{L(T)}}\right]\DE_0\left[\Phi^\b_{\ell}\right]\overleftarrow{\sZ}_{{L(T)},\ell}(0)\right]\\
&=\int_{\R^2\times \R^2}\dd x\dd y\, \rho_{\ell}(x)\rho_\ell(y)\frac{\rho_{{ L(T)}-2\ell}(y-x)}{\rho_{{L(T)}}(0)}\int_{\R^2\times \R^2}\dd z\dd w \rho_\ell(z)\rho_\ell(w)\\
&\hspace{2.5em}\times \DE_{0,0}^{\ell,x}\otimes \DE_{0,0}^{\ell,z}\left[\exp\left(\b^2\int_0^\ell V(B_s-\widetilde{B}_s)ds\right)\right]\DE_{0,0}^{\ell,y}\otimes \DE_{0,0}^{\ell,w}\left[\exp\left(\b^2\int_0^\ell V(B_s-\widetilde{B}_s)ds\right)\right],
\end{align*}
where $B=\{B_s^{B}:0\leq s\leq \ell\}$ and $\widetilde{B}=\{\widetilde{B}_s:0\leq s\leq \ell\}$ are independent  Brownian bridges with the law $\DP_{0,0}^{\ell,u}$ ($u=x,y,z,w$).
Then, it is easy to see that \begin{align*}
&\IE\left[\DE_{0,0}^{{L(T)},0}\left[\Phi^\b_{\ell}\Phi^\b_{{L(T)}-\ell,{L(T)}}\right]^2\right]-\IE\left[\DE_{0,0}^{{L(T)},0}\left[\Phi^\b_{\ell}\Phi_{{L(T)}-\ell,{L(T)}}\right]\DE_0\left[\Phi^\b_{\ell}\right]\overleftarrow{\sZ}_{{ L(T)},\ell}(0)\right]\\
&=\int_{\R^2\times \R^2}\dd x\dd y\, \rho_{\ell}(x)\rho_\ell(y)\frac{\rho_{{ L(T)}-2\ell}(y-x)}{\rho_{{ L(T)}}(0)}\int_{\R^2\times \R^2}\dd z\dd w \rho_\ell(z)\rho_\ell(w)\left(\frac{\rho_{{ L(T)}-2\ell}(z-w)}{\rho_{{ L(T)}}(0)}-1\right)\\
&\hspace{2.5em}\times\DE_{0,0}^{\ell,x}\otimes \DE_{0,0}^{\ell,z}\left[\exp\left(\b^2\int_0^\ell V(B_s-\widetilde{B}_s)ds\right)\right]\DE_{0,0}^{\ell,y}\otimes \DE_{0,0}^{\ell,w}\left[\exp\left(\b^2\int_0^\ell V(B_s-\widetilde{B}_s)ds\right)\right].
\end{align*}
Combining with \eqref{eq:p2pge}, \begin{align*}
\left|\IE\left[\DE_{0,0}^{{ L(T)},0}\left[\Phi^\b_{\ell}\Phi^\b_{{ L(T)}-\ell,{L(T)}}\right]^2\right]-\IE\left[\DE_{0,0}^{{ L(T)},0}\left[\Phi^\b_{\ell}\Phi_{{ L(T)}-\ell,{ L(T)}}\right]\DE_0\left[\Phi^\b_{\ell}\right]\overleftarrow{\sZ}_{{ L(T)},\ell}(0)\right]\right|\leq C\frac{\ell}{{L(T)}}.
\end{align*}
Also, the same argument holds for $\IE\left[\DE_0\left[\Phi^\b_{\ell}\right]^2 \overleftarrow{\sZ}_{{ L(T)},\ell}(0)^2\right]$.
\end{proof}

%\subsection{Second step}
\begin{lemma}[Step 3]\label{lem: estimate B}
  Fix $t>0$. There exists a positive constant $C$ such that for all ${ \hat{\b}}< 1$ there exists  a positive constant $C=C(\hat{\b})$  such that for all positive $\ell>0$ with $\ell_{{T_\e}}^a\leq \ell\leq \ell_{{ T_\e}}^b$ and all $x\in\mathbb R^d$, 
\[
 \IE\left[B_{{ L({ T_\e})},\ell,x}^2\right]\leq  \begin{cases}
 C\b_\e^2\left( \frac{|x|\log \ell}{{ L(T)}}+\frac{|x|^2 \ell}{{ L(T)}^2}\right)
 %\frac{C|x|}{{\MN L({\MN T_\e})}\log {\MN L({\MN T_\e})}}\left( \log \ell+\frac{|x| \ell}{{\MN L({\MN T_\e})}}\right)\quad 
 &|x|\leq \sqrt{{ L({ T_\e})}\log { L({ T_\e})}}\\
 C&|x|\geq \sqrt{{ L({ T_\e})}\log { L({ T_\e})}}
 \end{cases}.
\]
\end{lemma}
\begin{proof}
For $|x|\geq \sqrt{{L(T)}\log { L(T)}}$, it is trivial from \eqref{eq:L2bdd}. 

Combining \eqref{eq:L2white} and transformation of Brownian motions yield that %{\SN Check if the first equality is correct.}
\begin{align*}
 \IE\left[B_{{L(T)},\ell,x}^2\right]&={ 2}E\left[\exp\left(\b^2\int_0^\ell V(\sqrt{2}B_s)ds\right)-\exp\left(\b^2\int_0^\ell V\left(\sqrt{2}B_s+\frac{xs}{{ L(T)}}\right)ds\right)\right]\\
 &={ 2}\sum_{n=1}^\infty \b^{2n}\int_{0<t_1<\cdots<t_n<\ell}\int_{{\R^2}^n}\dd \mathbf{s}\dd \mathbf{x}\prod_{i=1}^n V(\sqrt{2}x_i)\\
 &\hspace{12mm}\times \left(\prod_{i=1}^n\rho_{s_i-s_{i-1}}(x_{i-1},x_{i})-\prod_{i=1}^n \rho_{s_i-s_{i-1}}\left(x_i-x_{i-1}+\frac{x(s_i-s_{i-1})}{{ L(T)}}\right)\right),
\end{align*}
{ where we set $x_0=0$.} When we use the relation \begin{align*}
&\prod_{i=1}^na_i-\prod_{i=1}^nb_i=\sum_{j=1}^n \left(\prod_{i=1}^{j-1} b_i\right)(a_j-b_j)\left(\prod_{k=j+1}^n a_k\right),
\end{align*}
and recall the notation $r_s$ from \eqref{Def: rs}, we have \begin{align}
& \IE\left[B_{{ L(T)},\ell,x}^2\right]\notag\\
&\leq { 2}\sup_{z\in B(0,R_V)}\b^2 \int_0^\ell \int_{\R^2} V(\sqrt{2}y)\left|\rho_s(y-z)-\rho_s(y-z+\frac{xs}{{ L(T)}})\right| \dd y \dd s \notag\\
&\hspace{12mm} \times \sum_{n=1}^\infty \sum_{k=1}^{ n}
 \b^{2n-2} \left(\sup_{z\in \R}\int_{0}^{\ell} \int_{\R^2} V(\sqrt{2}y)\rho_s(y-z)\dd y \dd s\right)^{n-1}\notag\\
% &\hspace{5em}\times \left(\sup_{z\in B(0,R_V)}\%int_{0}^{\ell} \int_{\R^2} V(\sqrt{2}y)\rho_s(y-z+%\frac{xs}{{\MN L(T)}})\dd y\right)^{n-1-k}\%\
&\leq  { 2}\sum_{n=1}^\infty { n
 \b^{2(n-1)}  \left(\int_{0}^{\ell}r_s  \dd s\right)^{n-1}} \sup_{z\in B(0,R_V)}\b^2 \int_0^\ell \int_{\R^2} V(\sqrt{2}y)\left|\rho_s(y-z)-\rho_s\left(y-z+\frac{xs}{{ L(T)}}\right)\right| \dd y  \dd s\notag\\
 &\leq C\b^2 \sup_{z\in B(0,R_V)} \int_0^\ell \int_{\R^2} V(\sqrt{2}y)\left|\rho_s(y-z)-\rho_s\left(y-z+\frac{xs}{{ L(T)}}\right)\right| \dd y \dd s,\label{eq:Best}
\end{align}
where we have used the estimate \eqref{eq:rorder} in the last line. Also, we have that for $y,z\in B(0,R_V)$, and for $s>0$\begin{align*}
&\left|\rho_s(y-z)-\rho_s\left(y-z+\frac{xs}{{{ L(T)}}}\right)\right|\\
&=\rho_{s}(y-z)\left|1-\exp\left(-\frac{\langle y-z,x\rangle}{{L(T)}}-\frac{|x|^2s}{2{L(T)}^2}\right)\right|\\
  &\leq \rho_s(y-z)\left(\frac{R_V|x|}{{L(T)}}\exp\left(\frac{{2}R_V|x|}{{L(T)}}-\frac{|x|^2s}{2{L(T)}^2}\right)+\frac{R_V|x|}{{L(T)}}+\frac{|x|^2s}{2{L(T)}^2}\right)\\
  &\leq C  \rho_s(y-z)\left( \frac{|x|}{L(T)}+\frac{|x|^2 s}{L(T)^2}\right)
\end{align*}
where we denote by $\langle x,y\rangle$ the inner product of $x$ and $y\in \R^2$ and we use $e^x-1\leq xe^x$ if  $x\geq 0$ and $1-e^x\leq -x$ if $x<0$ in the last line.
For $|x|\leq \sqrt{{ L(T)}\log{ L(T)}}$, \begin{align*}
 \IE\left[B_{{L(T)},\ell,x}^2\right]\leq&C\b^2\left(\frac{|x|\log \ell}{{ L(T)}}+\frac{|x|^2 \ell}{{L(T)}^2}\right).
\end{align*}
\end{proof}
{ Putting things together, we conclude the proof of Theorem \ref{th:errorTermLLT}.}\\

We also use the following lemma later.
\begin{lemma}\label{lem:covpart}
  For fixed $t>0$ and $\hat{\b}\in (0,1)$, there exists a constant $C=C_{\hat{\b},t}$ such that for $x\in \R^2$ and for $1\leq \ell \leq t{ T_\e}$ with $t>0$
  \begin{align}
\IE\left[\left(\DE_x\left[\Phi_\ell^{{ \b_\e}}\right]-\DE_0\left[\Phi_\ell^{{ \b_\e}}\right]\right)^2\right]\leq \begin{cases}
\displaystyle C\b_\e^2(1+|x|^2)\quad &|x|\leq \sqrt{\log { \ell}}\\
C&|x|>\sqrt{{\SN }\log { \ell}}.
\end{cases}
\label{eq:covpart}
  \end{align}
 % Moreover, for $L(T_\e)\leq t T_\e$  with  $\displaystyle \lim_{\e\to 0}\frac{\log L({ T_\e})}{\log {T_\e}}=1$. Then,  for all ${\SN\hat{\b}}<1$ there exists $C=C(\hat{\b})$ such that for all positive $\ell>0$ verifying $\ell_{{ T_\e}}^a\leq \ell\leq \ell_{{ T_\e}}^b$ and for all $x,y\in\mathbb R^d$,
%\begin{align*}
%&\IE\left(\DE_{0, 0}^{{ L({ T_\e})},x} [\Phi_{{\SN L({ T_\e})}}^{{ \b_\e}}] - \sZ_\ell^{{\b_\e}}( 0) \overleftarrow{\sZ}^{{ \b_\e}}_{{ L({ T_\e})},\ell}( x)\right)^2 \\
%&\leq \begin{cases}
%C\frac{\ell}{{ L({T_\e})}}+C\b_\e^2\left( \log \frac{{ L({T_\e})}}{\ell}+\frac{|x|\log \ell}{{ L(T)}}+\frac{|x|^2 \ell}{{L(T)}^2}\right)\quad &|x|\leq \sqrt{{ L({T_\e})}\log { L({T_\e})}}\\
%C&|x|\geq \sqrt{{ L({T_\e})}\log { L({ T_\e})}}
%\end{cases}.
%\end{align*}
\end{lemma}
\begin{proof}
  For $|x|\geq \sqrt{ \log \ell}$, it is trivial from \eqref{eq:L2bdd}. We suppose $|x|< \sqrt{\log \ell}$. % We use the  argument in Lemma~\ref{lem: estimate B} as follows.
  Using the same argument as in \eqref{eq:Best}, with the convention $x_0=0$,
  \begin{align*}
\IE\left[\left(\DE_x\left[\Phi_\ell^{{ \b_\e}}\right]-\DE\left[\Phi_\ell^{{\b_\e}}\right]\right)^2\right]&=E\left[\exp\left(\b^2\int_0^\ell V(\sqrt{2}B_s)ds\right)-\exp\left(\b^2\int_0^\ell V\left(x+\sqrt{2}B_s\right)ds\right)\right]\\
 &=\sum_{n=1}^\infty \b^{2n}\int_{0<t_1<\cdots<t_n<\ell}\int_{{\R^2}^n}\dd \mathbf{s}\dd \mathbf{x}\prod_{i=1}^n V(\sqrt{2}x_i)\\
 &\hspace{12mm}\times \left(\prod_{i=1}^n\rho_{s_i-s_{i-1}}(x_{i}-x_{i-1})-\rho_{s_i-s_{i-1}}\left(x_1-x\right) \prod_{i=2}^n \rho_{s_i-s_{i-1}}\left(x_i-x_{i-1}\right)\right)\\
%&\leq \sup_{z\in B(0,R_V)}\b^2 \int_0^\ell \int_{\R^2} V(\sqrt{2}y)\left|\rho_s(y-z)-\rho_s(y-z+\frac{xs}{{\MN L(T)}})\right| \dd y \\
%&\hspace{12mm} \times \sum_{n=1}^\infty \sum_{k=1}^{\SN n}
% \b^{2n-2} \left(\sup_{z\in \R}\int_{0}^{\ell} \int_{\R^2} V(\sqrt{2}y)\rho_s(y-z)\dd y\right)^{n-1}\\
% &\hspace{5em}\times \left(\sup_{z\in B(0,R_V)}\%int_{0}^{\ell} \int_{\R^2} V(\sqrt{2}y)\rho_s(y-z+%\frac{xs}{{\MN L(T)}})\dd y\right)^{n-1-k}\%\
%&\leq  \sum_{n=1}^\infty {\SN n
% \b^{2(n-1)}  \left(\int_{0}^{\ell}r_s  \dd s\right)^{n-1}} \sup_{z\in B(0,R_V)}\b^2 \int_0^\ell \int_{\R^2} V(\sqrt{2}y)\left|\rho_s(y-z)-\rho_s(y-z+x)\right| \dd y \\
 &\leq C\b^2\left(1+\sup_{z\in B(0,R_V)} \int_1^\ell \int_{\R^2} V(\sqrt{2}y)\left|\rho_s(y-z)-\rho_s(y-z+x)\right| \dd y \dd s\right).
\end{align*}
Also, we have that for $y,z\in B(0,R_V)$ and $s\geq 1$, \begin{align*}
\left|\rho_s(y-z)-\rho_s(y-z+x)\right|&=\rho_{s}(y-z)\left|1-\exp\left(-\frac{2\langle y-z,x\rangle+|x|^2}{2s}\right)\right|\\
&\leq \frac{C}{s}\rho_s(y-z)\left(\exp\left(\frac{R_V^2}{s}\right)+{|x|+|x|^2}\right)\\
&\leq \frac{C}{s}\rho_s(y-z)\left(1+|x|^2\right).
%&\leq C \rho_s(y-z)\left(\frac{1+|x|^2}{s}\right).
\end{align*}

Thus, we have %{\SN How to prove the estimate in the lemma?}
\begin{align*}
\IE\left[\left(\DE_x\left[\Phi_\ell^{{ \b_\e}}\right]-\DE_0\left[\Phi_\ell^{{ \b_\e}}\right]\right)^2\right]\leq&C \b^2\left(1+|x|^2\right).%\log{\ell}\left( 1+|x|^2\right).
\end{align*}
  
\end{proof}

\section{Proofs of Theorem \ref{th:EWlimit}}\label{sec:proofEWKPZ}
%We use a constant $C>0$ depending on $d$ and $u$, which may change from line to line.

For fixed $t>0$ and for $u_0\in \mathfrak{C}$, let us define the martingale
\begin{equation*}
    s \to \kW_{s}(x)= \kW_{s}^{(t,T,\hat{\b},u_0)}(x)=\DE_x\left[ \Phi_{s}^\b(B) \, u_0\left(\frac{B_{t T}}{\sqrt{T}}\right)\right]
\end{equation*}
with respect to the filtration $\{\mathcal{F}_s:0\leq s\leq tT\}$  associated to the white noise $\xi$.  Then, it follows from Feynman-Kac formula (see \eqref{eq:FKform}) that for each $(t,x)\in [0,\infty)\times \R^2$
\begin{equation} \label{eq:hepsWtT}
u_{\e}^{(T,\hat{\b},u_0)}(t,x)\eqlaw {\kW_{t T}^{(t,T,\hat{\b},u_0)}(\sqrt{T}x)}.
\end{equation}
We omit some superscripts $t$, $T$, $\hat{\b}$, and $u_0$  to make notation simple for several notations when it is easily understood from the context. 

  Since both $\|u_0^{-1}\|_{\infty}$ and $\|u_0\|_{\infty}$ are finite, 
  $$\|u_0^{-1}\|^{-1}_{\infty} \sZ_{s}(x)\leq \kW_{s}(x)\leq \|u_0\|_{\infty} \sZ_{s}(x).$$
  Hereafter, we use this without any comment.

%Put together, these two propositions imply 1-dimensional in time convergence for the EW limits and GFF limits (theorems \CC{???}).  (or for the multi-time correlation as in Section \ref{sec:conclusionOfProofOfSHE}, \textcolor{blue}{think about joint in time}).
%and
%\begin{equation}
% T^{\frac{d-2}{2}} \int_{\mathbb{R}^d\times \mathbb{R}^d} f(x) f(y)\, \langle H (\sqrt T x), H (\sqrt T y) \rangle_{t T} \dd x \dd y \cvLone \langle N(f) \rangle_\tau.
%\end{equation}
\begin{comment}
To prove Proposition \ref{prop:mainpropKPZ}, we rely again on the CLT for martingales. This time, because of the presence of $\sZ_\tau$ in the denominator of \eqref{eq:diffOfH}, the bracket of $H$ is slightly more delicate to deal with compared to the previous sections and our strategy is to replace $H$ in \eqref{eq:CLTmartH} by another martingale with no denominator.
\end{comment}

%{ In this subsection, we give a heuristic idea of the proof of Gaussian limit. }

It\^o's formula yields that for each $x\in \R^2$
\begin{align}
\kW_s^{(t,T,\hat{\b},u_0)}(x)&=\bar{u}(t,x)+\int_0^s \dd \kW^{(t,T,\hat{\b},u_0)}_u(x)\label{eq:Itozdeco}\\
\kW_s^{(t,T,\hat{\gm},v_0)}(x)&=\bar{v}(t,x)+\int_0^s \dd \kW^{(t,T,\hat{\gm},v_0)}_u(x)\label{eq:Itozdeco2}
\end{align}
with \begin{align}
\langle \kW^{(\hat{\b},u_0)}(x),\kW^{(\hat{\gm},v_0)}(y)\rangle_s=\int_0^s \b\gm\DE_{x}{\otimes }\DE_y\left[V(B_u-\widetilde{B}_u)\Phi^\b_u(B)\Phi^\gm_u(\widetilde{B})u_0\left(\frac{B_{t T}}{\sqrt{T}}\right)v_0\left(\frac{\widetilde{B}_{t T}}{\sqrt{T}}\right)\right]\dd u	\label{eq:QuadVarW}
\end{align}
for each $x,y\in\R^2$, where $\DE_{x}{\otimes }\DE_y$ denotes the expectation in two independent Brownian motions $B$ and $\widetilde{B}$ starting from $x$ and $y$.

Then,  we find by It\^{o}'s formula that for ${F}\in \mathfrak{F}$, $F(\kW_s(x))$ has the following semimartingale representation   \begin{align}
F(\kW^{(\hat{\b},u_0)}_s(x))&=F(\bar{u}(t,x))+\int_0^s F'(\kW^{(\hat{\b},u_0)}_u(x))\dd \kW^{(\hat{\b},u_0)}_u(x)\notag\\
&\hspace{6em}+\frac{1}{2}\int_0^s F''(\kW^{(\hat{\b},u_0)}_u(x))\dd\langle \kW^{(\hat{\b},u_0)}(x)\rangle_u\label{eq:ItoDecoF}
\end{align}
and we denote by \begin{align*}
&G_s^{(t,T,F,\hat{\b},u_0)}(x)=G_s(x)=\int_0^s F'({\kW_u^{(\hat{\b},u_0)}(x)})\dd \kW_u^{(\hat{\b},u_0)}(x)\\
&H^{(t,T,F,\hat{\b},u_0)}_s(x)=H_s(x)=\int_0^s F''(\kW^{(\hat{\b},u_0)}_u(x))\dd\langle \kW^{(\hat{\b},u_0)}(x)\rangle_u.
\end{align*}
%Due to the finiteness of negative moment, $G^F_s(x)$ is a continuous martingale and $F(\kW_s(x))$ is a continuous supermartingale for each $F\in\mathfrak{F}$.
%{\MN Finiteness of negative moment will be discussed later.}

%\begin{equation}\label{eq:Itodecomposition}
%\log \kW_\tau(x) =\int_0^\tau \frac{1}{\kW_s(x)}\dd \kW_s(x) - \frac{1}{2} \int_0^\tau \frac{1}{\kW_s(x)^2}\dd \langle \kW(x) \rangle_\tau.
%\end{equation}
%Let \begin{align*}
%H_\tau(x)=\int_0^\tau \frac{1}{\kW_s(x)}\dd \kW_s(x) .
%\end{align*}

First, we will prove the fluctuations of martingale parts converge to centered Gaussian random variables.

\begin{proposition} \label{prop:mainpropKPZ} Suppose $u_0^{(1)},\cdots,u_0^{(n)}\in \mathfrak{C}$, $\hat{\b}^{(1)},\cdots,\hat{\b}^{(n)}\in (0,1)$ and $F_{1},\cdots,F_{n}\in \mathfrak{F}$.

For any test function $f_1,\cdots,f_n\in C_c^\infty(\R^2)$,  as $T\to \infty$
%\begin{equation}
%T^{\frac{d-2}{2}} \int_{\mathbb{R}^d\times \mathbb{R}^d} f(x) f(y)\, \langle H (\sqrt T x), H (\sqrt T y) \rangle_{t T} \dd x \dd y \cvLone \langle N(f) \rangle_\tau.
%\end{equation}
%As a consequence,
\begin{equation} \label{eq:CLTmartH}
\left\{\frac{1}{\b^{(i)}_\e}\int_{\mathbb{R}^2} f_i(x) G^{\left(F_i,\b^{(i)},u_0^{(i)}\right)}_{{T_\e} t} (\sqrt {T_\e} x) \dd x\right\}_{i=1,\cdots,n} \cvlaw \left\{\mathscr U(t,f_i,F_i,\hat{\b}^{(i)},u_0^{(i)})\right\}_{i=1,\cdots,n},
\end{equation}
where $\left\{\mathscr U(t,f_i,F_i,\hat{\b}^{(i)},u_0^{(i)})\right\}_{i=1,\cdots,n}$ are Gaussian random variables with zero means and covariance 
\begin{align*}
&\mathrm{Cov}\left(\mathscr U(t,f_i,F_i,\hat{\b}^{(i)},u_0^{(i)}),\mathscr U(t,f_j,F_j,\hat{\b}^{(j)},u^{(j)}_0)\right)\\
&={\frac{1}{1-\hat{\b}^{(i)}\hat{\beta}^{(j)}}\int_{0}^t\dd s\int_{(\R^2)^2}\dd x\dd yf_i(x)f_j(y)
%\notag\\
%&\hspace{3em}\times 
I^{\left(t,F_i,\hat{\b}^{(i)},u_0^{(i)}\right)}(x)I^{\left(t,F_j,\hat{\b}^{(j)},u^{(j)}_0\right)}(y)}
%\notag\\
%\IE\left[F_i'\left(\frac{e^{X_{\hat{\b}_i}}\bar{u}^{(i)}(t,x)}{\sqrt{1-\hat{\b}_i^2}}\right)\right]
%\IE\left[F_j'\left(\frac{e^{X_{\hat{\b}_j}}\bar{u}^{(j)}(t,y)}{\sqrt{1-\hat{\b}_j^2}}\right)\right]
%\notag\\
%&\hspace{3em}\times 
\\
&\hspace{14em}{ \times \int_{\R^2}\dd z\rho_\sigma(x,z)\rho_\sigma(y,z)\bar{u}^{(i)}(t-\sigma,z)\bar{u}^{(j)}(t-\sigma,z)}.
\end{align*}
\end{proposition}

Then, we will prove that the It\^o correction term can be neglected in the limit: 
\begin{proposition}\label{prop:VanishBracket}
 For any $t>0$, $\hat{\b}\in(0,1)$, $u_0\in \mathfrak{C}$ and $F\in \mathfrak{F}$, as $\e\to 0$,
\begin{equation}
\frac{1}{\b_\e} \int_{\mathbb{R}^2} f(x)\,\left( H_{{T_\e}t}^{(F)}(\sqrt {T_\e} x)  - \IE\left[ H_{ {T_\e}t}^{(F)}(\sqrt  {T_\e} x)  \right]\right)  \cvLone 0.
\end{equation}
\end{proposition}

Proposition \ref{prop:mainpropKPZ}  and Proposition \ref{prop:VanishBracket} combined {with \eqref{eq:hepsWtT} and \eqref{eq:ItoDecoF}} imply Theorem~\ref{th:EWlimit} for $1$-dimensional in time.  Thus, Gaussian limit comes from the martingale part of $\dis \int f(x)F(\kW_s(\sqrt{ {  T_\e}}x))\dd x$.

%Hereafter, we will abbreviate the superscript $(F_i,\b_i,u_0^{(i)})$ to the index $i$ as $\dis G^i_t(x)=G_t^{(F_i,\b_i,u_0^{(i)})}(x)$. 

%The proofs of the results for the GFF limit and the EW limits are then concluded as in Section \ref{subsec:conclusionGFFKPZ} (or for the multi-time correlation see Section \ref{sec:KPZmultidim}).

%The proof for Proposition \ref{prop:mainpropKPZ} follows the martingale strategy of Section \ref{sec:EWproof}. This time, because of the presence of $\kW_\tau$ in the denominator of \eqref{eq:diffOfH}, the  quadratic variation of $H$ is slightly more delicate to deal with compared to the previous sections and we wish to replace $H$ in \eqref{eq:CLTmartH} by another martingale with no denominator. 

\subsection{Proof of Proposition  \ref{prop:mainpropKPZ} and heuristics}\label{IdeaG}

In the following, we give a heuristic idea of the proof of Proposition  \ref{prop:mainpropKPZ}.
%
%\vspace{1em}

First, we introduce  the key theorem to prove the convergence of martingale to Gaussian process in this paper: %{\SN We should introduce this later?}
\begin{theorem}\label{thm:JS}{\cite[Theorem 3.11 in Chap.\,8]{JS87}, \cite[Theorem 1.4 in Chap.\,7]{EK86}}

For each $n\geq 1$, let $\mathcal{F}^{n}=\{\mathcal{F}_t^n:t\geq 0\}$ be a filtration and let $X^{(n)}=(X_{t}^{(n,d)},\dots,X_{t}^{(n,d)})$ be an $\R^d$-valued continuous $\mathcal{F}^n$-martingale with $X_0^n=0$. Suppose that there exists a $d\times d$ positive definite matrix-valued  continuous function $c=\{c_{ij}(t)\}_{i,j=1}^d$ such that for each $t\geq 0$, $\langle X^{(n,i)},X^{(n,j)}\rangle_t\to c_{ij}(t)$ in probability. Then, $X^{(n)}\cvlaw X$, where $X=(X_t^{(1)},\cdots,X_t^{(d)})$ is an $\R^d$-valued Gaussian process with $\langle X^{(i)},X^{(j)}\rangle_t=c_{ij}(t)$. 
\end{theorem} 
\begin{rem}
Theorem \ref{thm:JS} is simplified from the original one for our convenience.
\end{rem}

Thus,  we will focus our analysis on  the cross-variation of martingales.

By the local limit theorem (Theorem \ref{th:errorTermLLT}), we may expect that for large $s$
\begin{align}
  \kW_s(x)&= \DE_x\left[ \Phi_s^\b \, u_0\left(\frac{B_{t T}}{\sqrt{T}}\right)\right]\notag\\
  &=\int_{\R^2}\rho_s(z-x) \DE_{0,x}^{s,z}[\Phi^\b_s] \DE_z\left[ u_0\left(\frac{B_{t T-s}}{\sqrt{T}}\right)\right] \dd z		\label{eq:Wexpan}\\  
  &\approx \int_{\R^2}\rho_s(z-x) \sZ^\b_{s{\ell(T)}}(x) \overleftarrow{\sZ}^\b_{s,{ s\ell(T)}}(z) \DE_z\left[ u_0\left(\frac{B_{t T-s}}{\sqrt{T}}\right)\right] \dd z\notag,
  \end{align}
  { where for fixed $\delta\in (0,\frac{1}{100})$, we set
\begin{align}\label{def: ell(T)}
  \ell(T)=\exp\left(-\left(\log T\right)^{\frac{1}{2}-\delta}\right).
  \end{align}}
{   Moreover, we may expect  that the last term is approximated in some sense by }
  \begin{align}
  & \sZ^\b_{{ s\ell(T)}}(x)  \int_{\R^2}\rho_s(z-x)\IE\left[\overleftarrow{\sZ}^\b_{s,{s\ell(T)}}(z)\right] \DE_z\left[ u_0\left(\frac{B_{t T-s}}{\sqrt{T}}\right)\right] \dd z=  \sZ^\b_{{ s\ell(T)}}(x) \bar{u}(t,T^{-\frac{1}{2}}\,x)\notag
\end{align}
since Lemma \ref{lem:covpart} may imply that $(\overleftarrow{\sZ}_{s,{ s\ell(T)}}(x))_{x\in\R^2}$ are asymptotically independent and homogenization occurs.

Therefore, one may observe for $F\in \mathfrak{F}$ that  for large $s$
\begin{align*}
F'(\kW_s(x)){\dd \kW_s(x)} & = {\b}F'({\kW_s(x)})  \int_{\mathbb R^2} \xi(\dd s,\dd b) \DE_x\left[ \phi(B_s-b) \Phi_s^\b   { \, u_0\left(\frac{B_{t T}}{\sqrt{T}}\right)} \right]\\
& = {\b}F'({\kW_s(x)})  \int_{\mathbb R^d} \xi(\dd s,\dd b) \int_{\mathbb R^2} \rho_s(z-x) \phi(z-b) \DE_{0,x}^{ s,z}\left[\Phi^\b_s\right] { \, \DE_z\left[ u_0\left(\frac{B_{t T-s}}{\sqrt{T}}\right)\right]} \dd z\\
&\approx {\b}F'(\sZ^\b_{{ s\ell(T)}}(x)\bar{u}(t,T^{-\frac{1}{2}}\,x)) \sZ^\b_{{ s\ell(T)}}(x)\\
&\hspace{4em}\times  \int_{\R^2} \xi(\dd s, \dd b) \int_{\R^2}\rho_s(z-x) \phi(z-b) \overleftarrow{\sZ}^\b_{s , { s\ell(T)}}(z)
 \, 
 \DE_z\left[ u_0\left(\frac{B_{t T-s}}{\sqrt{T}}\right)\right]\dd z,%\\
%&=: \dd \kM_\tau(x),
\end{align*}
where we have used the local limit theorem in the third line. 
We denote by \begin{align*}
\dis I^{(T)}_s(x)=I_{s}^{(t,T,F,\hat{\b},u_0)}(x)=F'\left(\sZ_{{s\ell(T)}}^{\b}(\sqrt{T}x)\bar{u}(t,x)\right)\sZ_{{ s\ell(T)}}^{\b}(\sqrt{T}x).
\end{align*}
Also, we have for $F_1,F_2\in \mathfrak{F}$ that  for $s=T\sigma$ and $x=\sqrt{T}x'$ and $y=\sqrt{T}y'$, \begin{align*}
&F_1'(\kW^{(\hat{\b},u_0)}_s(x))F_2'(\kW^{(\hat{\gm},v_0)}_s(y)){\dd \langle \kW^{(\hat{\b},u_0)}(x),\kW^{(\hat{\gm},v_0)}(y)\rangle}_s \\
&=\dd s\b\gm F'_1(\kW_s^\b(x))
F'_2(\kW_s^\gm(y))
\int_{(\R^2)^2}\dd z_1\dd z_2 \rho_s(x,z_1)\rho_s(y,{z_2})V(z_1-z_2)
\DE_{0,x}^{s,z_1}\left[\Phi_s^\b\right]
\DE_{0,x}^{s,z_2}\Big[\Phi_s^\gm\Big]\\
&\hspace{15em}\times
\DE_{z_1}\left[ u_0\left(\frac{B_{t T-s}}{\sqrt{T}}\right)\right]\DE_{z_2}\left[ v_0\left(\frac{\widetilde{B}_{t T-s}}{\sqrt{T}}\right)\right]\\
&\approx \dd s{\b\gm} 
F'_1\left(\sZ_{{s\ell(T)}}^\b(x)\bar{u}(t,x')\right)
F'_2\left(\sZ_{{ s\ell(T)}}^\gm(y)\bar{v}(t,y')\right)
\sZ_{{ s\ell(T)}}^\b(x)\sZ_{{s\ell(T)}}^\gm(y)\\
&\times \int_{(\R^2)^2}\dd z_1\dd z_2 \rho_s(x,z_1)\rho_s(y,{z_2})V(z_1-z_2)
\overleftarrow{\sZ}^\b_{s ,{ s\ell(T)}}(z_1)
\overleftarrow{\sZ}^\gm_{s , { s\ell(T)}}(z_2)
\DE_{z_1}\left[ u_0\left(\frac{B_{t T-s}}{\sqrt{T}}\right)\right]\DE_{z_2}\left[ v_0\left(\frac{\widetilde{B}_{t T-s}}{\sqrt{T}}\right)\right]\\
&\approx \dd \sigma{\b\gm} 
%F'_1\left(\sZ_{{ T\sigma \ell(T)}}^\b(x)\bar{u}(t,x')\right)
%F'_2\left(\sZ_{{ T\sigma\ell(T)}}^\gm(y)\bar{v}(t,y')\right)
%\sZ_{{ T\sigma\ell(T)}}^\b(x)\sZ_{{ T\sigma\ell(T)}}^\gm(y)
I_s^{(t,T,F_1,\hat{\b},u_0)}(x')I_s^{(t,T,F_2,\hat{\gm},v_0)}(y')\\
&\times \int_{(\R^2)^2}\dd z\dd w \rho_\sigma(x',z)\rho_\sigma(y',z-T^{-\frac{1}{2}}w)V(w)
\overleftarrow{\sZ}^\b_{T\sigma ,{ T\sigma\ell(T)}}(z)
\overleftarrow{\sZ}^\gm_{T\sigma , { T\sigma\ell(T)}}(z-T^{-\frac{1}{2}}w)\\
&\hspace{5em}\times \bar{u}(t-\sigma,z)
\bar{v}(t-\sigma,z-T^{-\frac{1}{2}}w)
\end{align*}
and thus by homogenization,  $\overleftarrow{\sZ}^\b\overleftarrow{\sZ}^\gm$ would be replaced by $\IE\left[\overleftarrow{\sZ}^\b\overleftarrow{\sZ}^\gm\right]$, and $F'(\sZ \bar{u}) \sZ$ terms would be replaceed by its expectation so that the cross variation would be approximated by 
\begin{align*}
&\dd \sigma{\b\gm} 
\IE\left[I_s^{(t,T,F_1,\hat{\b},u_0)}(x')\right]
\IE\left[I_s^{(t,T,F_2,\hat{\gm},v_0)}(y')\right]
\\
&\times \int_{\R^2}\dd z\rho_\sigma(x',z)\rho_\sigma(y',z)
\IE\left[\overleftarrow{\sZ}^\b_{T\sigma ,{ T\sigma\ell(T)}}(z)
\overleftarrow{\sZ}^\gm_{T\sigma , { T\sigma\ell(T)}}(z) \right]\bar{u}(t-\sigma,z)
\bar{v}(t-\sigma,z).
\end{align*}
Due to Theorem \ref{thm:CSZ17b} and \eqref{eq:L2bdd}, we have\begin{align}
&\int_{0}^{sT}\int_{(\R^2)^2}f(x)g(y)F_1'(\kW^\b_u(\sqrt{T}x))F_2'(\kW^\gm_u(\sqrt{T}y)){\dd \langle \kW^\b(\sqrt{T}x),\kW^\gm(\sqrt{T}y)\rangle}_{u} \notag\\
&\approx \frac{1}{1-\hat{\b}\hat{\gm}}\IE\left[F'(e^{X_{\hat{\b}}-\frac{1}{2}\sigma^2(\hat{\b})}\bar{u}(t,x))e^{X_{\hat{\b}}-\frac{1}{2}\sigma^2(\hat{\b})}\right]
\IE\left[F'(e^{X_{\hat{\gm}}-\frac{1}{2}\sigma^2(\hat{\gm})}\bar{v}(t,x))e^{X_{\hat{\gm}}-\frac{1}{2}\sigma^2(\hat{\gm})}\right]\notag\\
&\times \int_0^s\dd \sigma\int_{(\R^2)^2}\dd x \dd y f(x)g(y)\int_{\R^2}\dd z\rho_{\sigma}(x,z)\rho_{\sigma}(y,z)\bar{u}(t-\sigma,z)\bar{v}(t-\sigma,z)\label{eq:limitcov}
\end{align}
and Theorem \ref{thm:JS} implies that the limit process is the Gaussian process with covariance function \eqref{eq:limitcov}.

For simplicity of notations in the proof, we will focus on the quadratic variation of $\dis \int_{\mathbb{R}^2} f(x) G^{\left(F,\b,u_0\right)}_{{ T_\e} t} (\sqrt {T_\e} x) \dd x$. The reader can easily recover the proof for the cross-bracket from the above argument.

\vspace{1em}
To make this rough idea rigorous, we introduce a martingale increment $\dd\kM^{(t,T,F,\hat{\b},u_0)}_s(x)$ for fixed $t>0$, $x\in \R^2$, $\hat{\b}\in (0,1)$, $F\in \mathfrak{F}$, and $u_0\in \mathfrak{C}$ as %  We define a martingale increment $\dd\overline{M}_\tau(x)$ as
\begin{align}
&\dd \kM_s(x)=\dd \kM_s^{(t,T,F,\beta,u_0)}\nonumber \\
&={\b}F'(\sZ^\b_{{ s\ell(T)}}(x)\bar{u}(t,T^{-\frac{1}{2}}\,x)) \sZ^\b_{{ s\ell(T)}}(x)\notag\\
&\hspace{4em}\times  \int_{\R^2} \xi(\dd s, \dd b) \int_{\R^2}\rho_s(z-x) \phi(z-b) \overleftarrow{\sZ}^\b_{s, { s\ell(T)}}(z) \, \DE_z\left[ u_0\left(\frac{B_{t T-s}}{\sqrt{T}}\right)\right]\dd z,\label{eq:defdMbartau}
\end{align}
and set     \[
\kM_s(x)=\kM_{s}^{(t,T,F,\b,u_0)}(x) := \begin{cases}
\displaystyle \int_{t {  T_\e}m( {  T_\e})}^{s} \dd \kM_u(x)		\quad &s\geq t {  T_\e}m( {  T_\e})\\
0&0\leq s\leq t {  T_\e}m( {  T_\e}),
\end{cases}
\]
where { \begin{align*}
m( {  T_\e})=\exp\left(-\left(\log  {  T_\e}\right)^{\frac{1}{2}-\delta}\right).
\end{align*}}
%for   $s\in [\ell_1(T),t T]$ for $T> 0$  large enough.  {\MN $\ell_1$ is a function which will be chosen later.}

The following proposition computes the covariances of  $\kM_s$:

  %Since the martingale $\widetilde M$ is the same as the martingale integrated against $f$ in \eqref{eq:defMtau0tauf} (up to rescaling and the $\sZ_\ell(\sqrt T x)$ factor), we can use again the functionnal central limit theorem to get following result analogue to Proposition \ref{prop:CVbracketSHEalterGeneral}:
  
\begin{proposition} \label{prop:MwidetildeIsGaussianFlat}
 Suppose $u_0^{(1)},\cdots,u_0^{(n)}\in \mathfrak{C}$, $\hat{\b}^{(1)},\cdots,\hat{\b}^{(n)}\in (0,1)$ and $F_1,\cdots,F_n\in \mathfrak{F}$.

For any test function $f_1,\cdots,f_n\in C_c^\infty(\R^2)$,  as $\e\to 0$
\begin{equation} \label{eq:cvMtoU3}
 \frac{1}{\b_\e^{(i)}}\int_{\mathbb R^2} f_i(x) \, \kM^{(i)}_{ { T_\e}t}(\sqrt  {  T_\e} x) \dd x \cvlaw \mathscr U(t,f_i,F_i,\hat{\b}^{(i)},u_0^{(i)}).
\end{equation}
%as $T\to\infty$.
%with $N_{\tau}(f)$ defined in \eqref{eq:defNtau}. \CC{there should be $\mathscr U_3$ right?} {\SN I need to check it.} 
\end{proposition}

 The following proposition states that $\dd G_s(x)$ can be replaced by $\dd \kM_s(x)$, which  concludes the proof of Proposition \ref{prop:mainpropKPZ}: %In fact, we precisely show that
\begin{proposition} \label{prop:replaceByM}
 For any test function $f$ and $s>0$, 
  \al{
  \frac{1}{\b_\e}\IE\left[ \left|\int_{\R^2} f(x) \left(G_{s {  T_\e}}^{(t,T,F,\b,u_0)}(\sqrt  {  T_\e} x)-\kM^{(t,T,F,\b,u_0)}_{s {  T_\e}}(\sqrt{ {  T_\e}}x)\right)\dd x\right| \right] \to 0
  }
as $\e \to 0$.
\end{proposition}

The proof of  Proposition \ref{prop:MwidetildeIsGaussianFlat}  is given in the following subsection and the proof of Proposition \ref{prop:replaceByM} is given in subsection \ref{sub:4.3}.

%We end this subsection by  introducing and gathering some notations used in this section.
%\begin{itemize}
%\item For $x\in \R^2$ and $T>0$, $x_T=\sqrt{T}x$.
%\item .
%\item $\dis I^{(t,F,\hat{\b},u_0)}(x)=I(x)=
%\IE\left[
%F'\left(e^{X_{\hat{\b}}}
%\bar{u}(t,x)\right)
%e^{X_{\hat{\b}}}\right]=
%\IE\left[F'\left(e^{X_{\hat{\b}}+{\sigma^2(\hat{\b})}}
%\bar{u}(t,x)\right)
%\right].$ 
%\item 
%\end{itemize}

\subsection{Proof of Proposition \ref{prop:MwidetildeIsGaussianFlat}}\label{sub:4.2}

We will focus on only the quadratic variation of $G_{sT}(x\sqrt{T})$  to make the argument simple. Readers can easily replace the quadratic variation by the cross variation.

%\begin{proof}[Proof of Proposition \ref{prop:MwidetildeIsGaussianFlat}] 
To prove Proposition \ref{prop:MwidetildeIsGaussianFlat}, we  will show the following two lemmas:

%which combined to the proof of Proposition \ref{prop:CVbracketSHEalterGeneral} (more precisely, the one of Lemma \ref{lem:asymptPsi}) gives that $\langle\overline{M}^{\ssup T}(f)\rangle_\tau$ converges in $L^1$ to $\langle N(f)\rangle_\tau$ as $T\to\infty$ for all $\tau > \tau_0$. The functional CLT for martingales then implies the statement of the proposition.
%\CC{This proof has been quite modified.}
\begin{lemma}\label{lem:CLTmainpart}
Let $0<\tau_0\leq \tau\leq t$. Then, as $\e\to 0$,
\begin{align} \label{eq:CVbracketMt}
&\frac{1}{\b_\e^2} \int_{ { T_\e}\tau_0}^{ {  T_\e}\tau} \int_{\left(\R^2\right)^2}\dd x\dd yf(x)f(y)\dd \left\langle \kM(x_{ {  T_\e}}), \kM(y_{ { T_\e}}) \right\rangle_{s}\dd s\notag\\
& \cvLone \frac{1}{1-\hat{\b}^2}\int_{\tau_0}^\tau\dd s\int_{(\R^2)^2}\dd x\dd yf(x)f(y)
%\notag\\
%&\hspace{3em}\times 
I(x)I(y)%\notag\\
%\IE\left[F_i'\left(\frac{e^{X_{\hat{\b}_i}}\bar{u}^{(i)}(t,x)}{\sqrt{1-\hat{\b}_i^2}}\right)\right]
%\IE\left[F_j'\left(\frac{e^{X_{\hat{\b}_j}}\bar{u}^{(j)}(t,y)}{\sqrt{1-\hat{\b}_j^2}}\right)\right]
%\notag\\
%&\hspace{3em}\times 
\int_{\R^2}\dd z\rho_\sigma(x-z)\rho_\sigma(y-z)\bar{u}(t-\sigma,z)^2,
\end{align} 
where  we set $x_{{T_\e}}=x\sqrt{ { T_\e}}$ for $x\in \R^2$  and
\begin{align*}\dis I(x)=I^{(t,F,\hat{\b},u_0)}(x):=
\IE\left[
F'\left(e^{X_{\hat{\b}}}
\bar{u}(t,x)\right)
e^{X_{\hat{\b}}}\right]=
\IE\left[F'\left(e^{X_{\hat{\b}}+{\sigma^2(\hat{\b})}}
\bar{u}(t,x)\right)
\right]
\end{align*} 
 for $t>0$, $x\in\R^2$, $F\in \mathfrak{F}$, $\hat{\b}\in (0,1)$, and $u_0\in \mathfrak{C}$.
%where $X_{\hat{\b}}$ is a centered Gaussian random variable with variance $\sigma(\hat{\b})=\displaystyle \log \frac{1}{1-\hat{\b}^2}$.
%where we denote by $x_T=\sqrt{T}x$ for $x\in \R^2$ for simplicity.
\end{lemma}

Lemma \ref{lem:CLTmainpart} with Theorem \ref{thm:JS} implies  that the centered martingale $\displaystyle \left(\int_{\R^2}\frac{1}{\b}f(x)\left(\kM_{T \tau} (x_T)- \kM_{T \tau_0} (x_T)\right)\dd x\right)_{\tau_0\leq \tau\leq t}$ converges in distribution to a Gaussian process with covariance given by the RHS of \eqref{eq:CVbracketMt}. 

\begin{lemma}\label{lem:MtildeTau0negl}
\begin{equation} \label{eq:MtildeTau0negl}
 \lim_{\tau_0\to 0}\varlimsup_{\e\to 0}\IE\left[\frac{1}{\b_\e^2}\left(\int f(x) \kM_{ {  T_\e}\tau_0}(x_{ { T_\e}})\dd x\right)^2\right]=0.
\end{equation}
\end{lemma}

%Moreover, by proof of Lemma \ref{lem:unifNeglTau0}, we have as $\tau_0 \to 0$,
%\begin{equation} \label{eq:MtildeTau0negl}
% \lim_{\tau_0\to 0}\varlimsup_{T\to\infty}\IE\left[ T^{\frac{d-2}{2}}\left(\int f(x) \overline{M}_{T\tau_0}(\sqrt T x)\dd x\right)^2\right]=0,
%\end{equation}
Thus, letting $\tau_0\to 0$ and $\tau=t$,  the RHS of \eqref{eq:CVbracketMt} is exactly the covariance function of the Gaussian process { $\kU_t(f,F,\hat{\b},u_0)$.}

% we thus obtain that \eqref{eq:cvMtoU3} holds with $\mathscr U_3(t,x) = \bar u(t,x)^{-1} \mathscr U_1(t,x)$. Let us now justify why defined this way, $\mathscr U_3(t,x)$  is a solution of the SPDE \eqref{eq:defU3}.

% In other words, if we define $\overline{U}_{t}(\sqrt T x)=\bar{u}(x,t)\overline{M}_{t}(\sqrt T x)$, then
%        \begin{equation} \label{eq:CVbracketUt}
%          T^{-\frac{d-2}{2}} \left\langle \overline{U}(\sqrt T x), \overline{U}(\sqrt T y) \right \rangle_{T\sigma } \cvLone \gamma^2(\b)  \int_{\R^d}\rho_\sigma(x-z)\rho_\sigma(y-z)\bar{u}(t-\sigma,z)^2\dd z ,
%        \end{equation}
% where the limit on the RHS is exactly the covariance of $\int \dd x\, f(x) \kU_1(t,x)$. 
% 
% 
% 
% This means that the limit of
%        $$T^{\frac{d-2}{4}}\int_{\R^d} f(x)\, \bar{u}(t,x) (\log{\kW_{t T}}-\IE \log{\kW_{t T}})\dd x$$
%       has the same distribution as EW for the SHE for the initial condition $u_0$, i.e. $\int \dd x\, f(x) \kU_1(t,x)$.
%
%       Moreover, this convergence is jointly in time as we saw in the previous section \CC{I don't understand why this is necessary since we only considering pointwise convergence for now}. {\SN You're right; I will erase it.} 

%        Moreover, 
%        $$T^{\frac{d-2}{4}}\int_{\R^d} f(x)\, (\log{\kW_{t T}}-\IE \log{\kW_{t T}})\dd x \longrightarrow \int \dd x \, f(x) \kU_3(t,x),$$
%        and hence, the proof of Proposition \ref{prop:MwidetildeIsGaussianFlat} is completed .
%        \end{proof}
      %  \end{rem}

\subsubsection{Proof of Lemma \ref{lem:CLTmainpart} and Lemma \ref{lem:MtildeTau0negl}}

The proof of Lemma \ref{lem:CLTmainpart} is divided into several steps.

Recall that $x_T=\sqrt{T}\,x$. First of all, we can easily find by Markov property and \eqref{eq:QuadVarW} that 
   \begin{align*}
   &\frac{1}{\b^2} \int_{T\tau_0}^{T\tau} \int_{\left(\R^2\right)^2}\dd x\dd yf(x)f(y)\dd \left\langle \kM(x_T), \kM(y_T) \right\rangle_{s}\dd x\dd y\dd s\notag\\
& = \int_{T\tau_0}^{T\tau}\dd s \int_{(\R^2)^2}
f(x)f(y)
I_s^{(T)}(x)
I_s^{(T)}(y)\\
%F_1'(\sZ^{\b_1}_{{\MN s\ell(T)}}(x_T)\bar{u}^{(1)}(t,x)) \sZ^{\b_1}_{{\MN s\ell(T)}}(x_T)
%F'_2(\sZ^{\b_2}_{\ell(s)}(y_T)\bar{u}^{(2)}(t,y)) \sZ^{\b_2}_{\ell(s)}(y_T)\notag\\
&\hspace{4em}\int_{(\R^2)^2}\dd z_1\dd z_2\rho_s(z_1-x_T)\rho_s(z_2-y_T)V(z_1-z_2)\notag\\
&{\hspace{7em} \times \overleftarrow{\sZ}^{\b}_{s,{ s\ell(T)}}(z_1)
\overleftarrow{\sZ}^{\b}_{s,{s\ell(T)}}(z_2)  
\DE_{z_1}\left[u_0\left(\frac{B_{tT-s}}{\sqrt{T}}\right)\right]
\DE_{z_2}\left[u_0\left(\frac{B_{tT-s}}{\sqrt{T}}\right)\right]}\\
&=T \int_{\tau_0}^{\tau}\dd \sigma \int_{(\R^2)^2}\dd x\dd y
f(x)f(y)
I_{T\sigma}^{(T)}(x)
I_{T\sigma }^{(T)}(y)\\
%F'(\sZ^{\b_1}_{\ell(\tau)}(x_T)\bar{u}(t,x)) \sZ^{\b_1}_{\ell(s)}(x_T)
%F'(\sZ^{\b_2}_{\ell(\tau)}(y_T)\bar{u}(t,y)) \sZ^{\b_2}_{\ell(s)}(y_T)\notag\\
&\hspace{4em}\int_{(\R^2)^2}\dd z\dd w \rho_\sigma(z-x)\rho_\sigma(w-y)V(z_T-w_T)\notag\\
&{\hspace{7em} \times
 \overleftarrow{\sZ}^{\b}_{T\sigma,{T\sigma\ell(T)}}(z_T)
\overleftarrow{\sZ}^{\b}_{T\sigma,{T\sigma\ell(T)}}(w_T)
\bar{u}(t-\sigma,z)\bar{u}(t-\sigma,w)
%  \DE_{z_T}\left[u_0\left(\frac{B_{tT-\sigma T}}{\sqrt{T}}\right)\right]
%\DE_{w_T}\left[u_0\left(\frac{B_{tT-\sigma T}}{\sqrt{T}}\right)\right]
}.
   \end{align*}
%   where we denote by $x_T=\sqrt{T}x$ for $x\in \R^2$ for simplicity of notations. Let 
%   \begin{align*}
%  I_\tau^{(t,T)}(x)=I_\tau(x)= F'\left(\sZ^\b_{\ell(\tau)}(\sqrt{T}x)\bar{u}(t,x)\right) \sZ^\b_{\ell(\tau)}(\sqrt{T}x).
%   \end{align*}
%{\SN (Should we organize the above by using $\bar{u}(t-\sigma,z)=\DE_{z_T}\left[u_0\left(\frac{B_{tT-\sigma T}}{\sqrt{T}}\right)\right]$?)}
We define for $x,y\in \R^2$ and  $\tau_0\leq \sigma \leq \tau$\begin{align*}
\Psi_{{ \sigma}}^T(x,y)
&=T\int_{(\R^2)^2}\dd z\dd w\rho_\sigma(z-x)\rho_\sigma(w-y)V(z_T-w_T)\notag\\
&{\hspace{4em}\times 
\overleftarrow{\sZ}^{\b}_{T\sigma,{T\sigma\ell(T)}}(z_T)
\overleftarrow{\sZ}^{\b}_{T\sigma,{T\sigma\ell(T)}}(w_T) 
\bar{u}(t-\sigma,z)\bar{u}(t-\sigma,w)
% \DE_{z_T}\left[u_0\left(\frac{B_{tT-\sigma T}}{\sqrt{T}}\right)\right]
%\DE_{w_T}\left[u_0\left(\frac{B_{tT-\sigma T}}{\sqrt{T}}\right)\right]
}.
\end{align*}   

\begin{lemma}\label{lem:approxPsi}
  For each $x,y\in\R^2$, $\sigma \in [\tau_0,\tau]$ %{\SN (uniformly?)}{\MN Not uniformly}
  \begin{align*}
\lim_{T\to\infty}\IE\left[\left|\Psi^T_\sigma(x,y)-\Psi_\sigma(x,y)\right|\right]=0,
\end{align*}
where \begin{align*}
\Psi_{{\sigma}}(x,y)=\frac{1}{1-\hat{\b}^2}\int_{\R^2}\dd w\rho_\sigma(w-x)\rho_\sigma(w-y)\bar{u}(t-\sigma,z)^2.
\end{align*}
\end{lemma}

{ Combining this with Lemma \ref{lem:negativemoment} and \eqref{eq:Fass}, it is easy to see by the dominated convergence theorem that \begin{align*}
&\IE\left[\int_{\tau_0}^{\tau}\dd \sigma \int_{(\R^2)^2}
\left|f(x)f(y)
I_{\sigma T}^{(T)}(x)
I_{\sigma T}^{(T)}(y)\right|\left|
\Psi_\sigma^T(x,y)-\Psi_\sigma(x,y)\right|\right]\\
&=\int_{\tau_0}^{\tau}\dd \sigma \int_{(\R^2)^2}
\left|f(x)f(y)\right|
\IE\left[\left|I_{\sigma T}^{(T)}(x)
I_{\sigma T}^{(T)}(y)\right|\right]\IE\left[\left|
\Psi_\sigma^T(x,y)-\Psi_\sigma(x,y)\right|\right]\to 0.
\end{align*}
}

\begin{lemma}\label{lem:approxFF}
For any test function $f\in C_c^\infty(\R^2)$ \begin{align*}
{ \int_{\tau_0}^\tau \dd \sigma}\int_{(\R^2)^2}\dd x\dd y f(x)f(y)I_{\sigma T}^{(T)}(x)I_{\sigma T}^{(T)}(y)	\Psi_\sigma(x,y)
&\approx_{L^1} { \int_{\tau_0}^\tau\dd \sigma }\int_{(\R^2)^2}\dd x\dd y f(x)f(y)
I_{\sigma T}(x)I_{\sigma T}(y)
\Psi(x,y)
%\IE\left[F'\left(\frac{e^{X_{\hat{\b}_1}}\bar{u}^{(1)}(t,x)}{\sqrt{1-\hat{\b}_1^2}}\right)\right]
%\IE\left[F'\left(\frac{e^{X_{\hat{\b}_2}}\bar{u}^{(2)}(t,y)}{\sqrt{1-\hat{\b}_2^2}}\right)\right]\Psi(x,y),
\end{align*}
as $T\to\infty$, where the $\approx_{L^1}$ sign means that the difference between the left and right sides goes to $0$ in $L^1$-sense.
\end{lemma}

%\begin{lemma}
%For each $x,y\in \R^2$, $t>0$, the family $\{\sZ^\b_s(x)\sZ^\b_s(y):0\leq s\leq tT,T\geq 1\}$ is uniformly integrable. 
%\end{lemma}

%\begin{proof}
%Applying the Doob's inequality to the martingale $\{\sZ^\b_s(x):0\leq s\leq tT\}$ with \eqref{eq:L2bdd}, there exist constants $C$ and $C_\b$ \begin{align*}
%\IE\left[\sup_{0\leq s\leq tT}\sZ^\b_s(x)\sZ^\b_s(y)\right]\leq C\IE\left[\left(\sZ^\b_s(x)\right)^2\right]\leq C_\b.
%\end{align*}

%\end{proof}

\begin{proof}[Proof of Lemma \ref{lem:approxPsi}]

  (Step 1) Letting $z=w+\frac{v}{\sqrt{T}}$, %{\SN Should we exchange $z$ and $w$?}
  \begin{align*}
\Psi^T_{ \sigma }(x,y)&=\int_{(\R^2)^2}\dd w\dd v\, \rho_\sigma(w-y)V(v) \overleftarrow{\sZ}^{\b}_{T\sigma,{T\sigma\ell(T)}}(w_T) \bar{u}(t-\sigma,w)  \notag\\
&{\hspace{4em}\times 
\rho_\sigma(w+\frac{v}{\sqrt{T}}-x)\overleftarrow{\sZ}^{\b}_{T\sigma,{T\sigma\ell(T)}}(w_T+v)\bar{u}\left(t-\sigma,w+\frac{v}{\sqrt{T}}\right)
%\DE_{w_T+v}\left[u_0\left(\frac{B_{tT}}{\sqrt{T}}\right)\right]
%\DE_{w_T}\left[u_0\left(\frac{B_{tT}}{\sqrt{T}}\right)\right]
}.
  \end{align*}
 { We note that
  \al{
  &\rho_\sigma\left(w+\frac{v}{\sqrt{T}}-x\right)
\overleftarrow{\sZ}^{\b}_{T\sigma,{T\sigma\ell(T)}}(w_T+v)  
\bar{u}\left(t-\sigma,w+\frac{v}{\sqrt{T}}\right)
 - \rho_\sigma(w-x)\overleftarrow{\sZ}^{\b}_{T\sigma,{ T\sigma\ell(T)}}(w_T)\bar{u}(t-\sigma,w) \\
    &\leq   \left|\rho_\sigma\left(w+\frac{v}{\sqrt{T}}-x\right)-\rho_\sigma(w-x)\right|
\overleftarrow{\sZ}^{\b}_{T\sigma,{T\sigma\ell(T)}}(w_T+v)  
\bar{u}\left(t-\sigma,w+\frac{v}{\sqrt{T}}\right)\\
&\qquad+ \rho_\sigma(w-x)\left|\overleftarrow{\sZ}^{\b}_{T\sigma,{T\sigma\ell(T)}}(w_T+v)-\overleftarrow{\sZ}^{\b}_{T\sigma,{ T\sigma\ell(T)}}(w_T)\right|\bar{u}(t-\sigma,w)\\
&\qquad\qquad+\rho_\sigma(w-x)\overleftarrow{\sZ}^{\b}_{T\sigma,{ T\sigma\ell(T)}}(w_T)\left|\bar{u}\left(t-\sigma,w+\frac{v}{\sqrt{T}}\right)-\bar{u}(t-\sigma,w)\right|,
  }
  all of which converge to $0$ as $T\to\infty$ by Lemma~\ref{lem:covpart}.  Since $\bar{u}$ and $\rho_\sigma(x)$  is bounded for $\sigma\in[\tau_0,\tau]$ and $x\in\R^2$, using \eqref{eq:L2bdd} and $\int V(v) \dd v=1$, we have by the dominated convergence theorem,
 % \begin{align*}
%\Psi_\sigma^T(x,y)&\approx_{L^1}  \int_{(\R^2)^2}\dd w\dd v \rho_\sigma(w-x)\rho_\sigma(w-y)V(v)\notag\\
%&\hspace{4em}\times 
%\overleftarrow{\sZ}^{\b}_{T\sigma,{ T\sigma\ell(T)}}(w_T+v)
%\overleftarrow{\sZ}^{\b}_{T\sigma,{ T\sigma\ell(T)}}(w_T)  
%\bar{u}(t-\sigma,w)^2.%\DE_{w_T}\left[u_0\left(\frac{B_{tT}}{\sqrt{T}}\right)\right]
%\DE_{w_T}\left[u_0\left(\frac{B_{tT}}{\sqrt{T}}\right)\right]
%\end{align*}

  %Combining Lemma \ref{lem:covpart} with , it is easy to obtain that
  \begin{align*}
\Psi_\sigma^T(x,y)&\approx_{L^1}  
\int_{\R^2}\dd w \dd v\,\rho_\sigma(w-x)\rho_\sigma(w-y)V(v)
\overleftarrow{\sZ}^{\b}_{T\sigma,\ell(T\sigma)}(w_T)^2
%\overleftarrow{\sZ}^{\b}_{T\sigma,\ell(T\sigma)}(w_T)  
\bar{u}(t-\sigma,w)^2\\&=  
\int_{\R^2}\dd w\rho_\sigma(w-x)\rho_\sigma(w-y)%\\
%&\hspace{4em}\times
\overleftarrow{\sZ}^{\b}_{T\sigma,\ell(T\sigma)}(w_T)^2
%\overleftarrow{\sZ}^{\b}_{T\sigma,\ell(T\sigma)}(w_T)  
\bar{u}(t-\sigma,w)^2.
  \end{align*}
  }
(Step 2) Since we have \begin{align*}
\overleftarrow{\sZ}^\b_{s,{ s\ell(T)}}(w)\eqlaw {\sZ}^\b_{{ s\ell(T)}}(w)
\end{align*}
for each $w\in \R^2$ and $s>0$, it is enough from \eqref{eq:L2bdd} to show that for each $\sigma\in [\tau_0,\tau]$ and $x,y\in \R^2$\begin{align}
&\int_{\R^2}\dd w\rho_\sigma(w-x)\rho_\sigma(w-y)
{\sZ}^{\b}_{{ T\sigma\ell(T)}}(w_T)^2
%{\sZ}^{\b}_{\ell(T\sigma)}(w_T)  
\bar{u}(t-\sigma,w)^2\notag\\
&\approx_{L^1} \int_{\R^2}\dd w\rho_\sigma(w-x)\rho_\sigma(w-y)
\IE\left[{\sZ}^{\b}_{{ T\sigma\ell(T)}}(w_T)^2\right]
\bar{u}(t-\sigma,w)^2.\label{eq:approxex}
\end{align}
%To prove \eqref{eq:approxex}, we introduce new random variables \begin{align*}
%\bar{\sZ}_{(t,r)}^\b(z):=\DE_z\left[\Phi_{t}^\b: B[0,t]\subset \mathtt{R}_{t,r}(z)\right],
%\end{align*}
%and 
It follows from the approximations: \begin{align}
&\int_{\R^2}\dd w\rho_\sigma(w-x)\rho_\sigma(w-y)\left({\sZ}^\b_{{ T\sigma\ell(T)}}(w_T) \right)^2\bar{u}(t-\sigma,w)^2\notag\\
&\approx_{L^1} \int_{\R^2}\dd w\rho_\sigma(w-x)\rho_\sigma(w-y)\left(\left(\widetilde{\sZ}^\b_{{ T\sigma\ell(T)},{ \ell'(\sigma,T)}}(w_T) \right)^2\wedge \left({{ \ell(T)}}\right)^{-\frac{1}{2}}\right)\bar{u}(t-\sigma,w)^2\label{eq:approxT}\\
&\approx_{L^1}\int_{\R^2}\dd w\rho_\sigma(w-x)\rho_\sigma(w-y)
\IE\left[\left(\widetilde{\sZ}^\b_{{ T\sigma\ell(T)},{ \ell'(\sigma,T)}}(w_T) \right)^2\wedge \left({{ \ell(T)}}\right)^{-\frac{1}{2}}\right]\bar{u}(t-\sigma,w)^2,\label{eq:approxTL}
\end{align}
{ where   we define for $t\geq 0$, $r>0$, $z\in \R^2$ \begin{align}
\widetilde{\sZ}_{t,r}(z)&={\widetilde{\sZ}_{t,r}^{\b}(z)}=\DE_{z}\left[\Phi_{t}^\b(B):\mathtt{F}_{t,r}(B,z) \right],\label{eq:barZ}
  \end{align}
and $\mathtt{F}_{t,r}(B,z)$ is the event that Brownian motion $B$ does not escape from the open ball $B(z,r)=\{x\in \R^2:|x-y|<r\}$  up to times $t$: \begin{align*}
\mathtt{F}_{t,r}(B,z)=\{B_s\in B(z,r)\ \text{for any }s\in [0,t]\}
\end{align*}
and  }we set { \begin{align*}
\ell'(\sigma,T)=\sqrt{T\sigma}\ell(T)^{\frac{1}{4}}
\end{align*}}
and we denote by $\mathcal{V}_{{ T,\sigma}}^\b(w)
=\left(\widetilde{\sZ}^\b_{({ T\sigma\ell(T)},{ \ell'(\sigma,T)})}(w_T) \right)^2\wedge \left({\ell(T)}\right)^{-\frac{1}{2}}$ for simplicity. %(Recall the definition of $\overline{\sZ}$ \eqref{eq:barZ}.)
The following argument yields \eqref{eq:approxT}: We find
\begin{align*}
&\IE\left[\left({\sZ}^\b_{T\sigma \ell(T)}(w_T) \right)^2-\left(\left(\widetilde{\sZ}^\b_{{ T\sigma\ell(T)},{ \ell'(\sigma,T)}}(w_T) \right)^2\wedge \left({{ \ell(T)}}\right)^{-\frac{1}{2}}\right)\right]\\
&\leq  \IE\left[\left({\sZ}^\b_{T\sigma\ell(T)}(w_T) \right)^2-\left(\widetilde{\sZ}^\b_{{ T\sigma\ell(T)},{ \ell'(\sigma,T)}}(w_T) \right)^2\right]\\
&+\IE\left[\left(\widetilde{\sZ}^\b_{{ T\sigma\ell(T)},{ \ell'(\sigma,T)}}(w_T) \right)^2: \left(\widetilde{\sZ}^\b_{{ T\sigma\ell(T)},{ \ell'(\sigma,T)}}(w_T) \right)^2\geq  \left({{ \ell(T)}}\right)^{-\frac{1}{2}}\right]
\end{align*}
and the last term tends to $0$ as $T\to\infty$ by Lemma \ref{lem:lp}. Furthermore,
\begin{align*}
&\IE\left[\left({\sZ}^\b_{T\sigma\ell(T)}(w_T) \right)^2-\left(\widetilde{\sZ}^\b_{{T\sigma\ell(T)},{ \ell'(\sigma,T)}}(w_T) \right)^2\right]\\
&\leq 4\IE\left[\left({\sZ}^\b_{{T\sigma\ell(T)}}(w_T) \right)^2\right]
\IE\left[\DE_{w_T}\left[\Phi_{{ T\sigma \ell(T)}}(B):\mathtt{F}_{{T\sigma\ell(T)},{ \ell'(\sigma,T)}}(B,w_T)^c\right]^2\right]\\
&\leq C\DE_{w_T}\otimes \DE_{w_T}\left[\exp\left(\b^2\int_0^{{ T\sigma \ell(T)}}V(B_s-B_s')\dd s\right): \mathtt{F}_{{ T\sigma\ell(T)},{ \ell'(\sigma,T)}}(B,w_T)^c\right]\\
&\leq C\DE_{w_T}\otimes \DE_{w_T}\left[\exp\left(\b^2p\int_0^{
{ T\sigma\ell(T)}}V(B_s-B_s')\dd s\right)\right]^{1/p}\DP_{w_T}\left(\mathtt{F}_{{T\sigma\ell(T)},{ \ell'(\sigma,T)}}(B,w_T)^c\right)^{1/q}\\
&\leq C\DP_{w_T}\left(\mathtt{F}_{{ T\sigma\ell(T)},{ \ell'(\sigma,T)}}(B,w_T)^c\right)^{1/q}\to 0
\end{align*}
where $B$ and $B'$ are independent Brownian motions starting from $w_T$ and we have used the Cauchy-Schwarz inequality in the first line and the  H\"{o}lder inequality for $p,q>1$ with $\frac{1}{p}+\frac{1}{q}=1$ and the fact that there exists a { constant} $p>1$ such that \begin{align}
\varlimsup_{T\to\infty}\DE_{x}\otimes \DE_x\left[\exp\left(\b^2p\int_0^{{ T\sigma \ell(T)}}V(B_s-B_s')\dd s\right)\right]<\infty. \label{eq:collisionLp}
\end{align}
%On the other hand, it is easy to see from Lemma \ref{lem:lp} that \begin{align*}
%\IE\left[\left(\bar{\sZ}^\b_{\ell(T),\sqrt{T}}(w_T) \right)^2-\left(\left(\bar{\sZ}^\b_{\ell(T),\sqrt{T}}(w_T) \right)^2\right)\wedge \sqrt{T}\right]
%\end{align*}
(Step 3)
We end the proof by  showing \eqref{eq:approxTL}. First, we remark that  if  $|w_T-w_T'|>2({ \ell'(\sigma,T)}+R_\phi)$, then
\begin{align*}
\textrm{Cov}\left(\mathcal{V}_{{ T,\sigma}}^\b(w_T),\mathcal{V}_{{ T,\sigma}}^\b(w'_T)\right)=0,
\end{align*}
 where $R_\phi$ is a constant with $\textrm{supp}\phi\subset B(0,R_\phi)$.

Therefore, \begin{align*}
&\IE\left[\left(\int_{\R^2}\dd w\rho_\sigma(w-x)\rho_\sigma(w-y)\left(\mathcal{V}_{{ T,\sigma}}^\b(w_T)-\IE\left[\mathcal{V}_{{ T,\sigma}}^\b(w_T)\right]\right)\bar{u}(t-\sigma,w)^2\right)^2\right]\\
&= \int_{(\R^2)^2}\dd w\dd w'\rho_\sigma(w-x)\rho_\sigma(w-y)\rho_\sigma(w'-x)\rho_\sigma(w'-y)\\
&\hspace{3em}\times \textrm{Cov}\left(\mathcal{V}_{{ T,\sigma}}^\b(w_T),\mathcal{V}_{{ T,\sigma}}^\b(w'_T)\right)
\bar{u}(t-\sigma,w)^2\bar{u}(t-\sigma,w')^2\\
&\leq \int_{|w_T-w_T'|\leq 2(\ell'(\sigma,T)+R_V)}\dd w\dd w'\rho_\sigma(w-x)\rho_\sigma(w-y)\rho_\sigma(w'-x)\rho_\sigma(w'-y)\\
&\hspace{3em}\times \IE\left[\mathcal{V}_{{ T,\sigma}}^\b(0)^2\right]
\bar{u}(t-\sigma,w)^2\bar{u}(t-\sigma,w')^2\\
&\leq C{ {\ell(T)}}^{\frac{1}{2}} \IE\left[\mathcal{V}_{{ T,\sigma}}^\b(0)^2\right].
\end{align*}

Thus, it is enough to show that %{\SN (Forgot $\ell(T)^{\frac{1}{2}}$?)}{\MN Corrected}
\begin{align*}
\lim_{T\to \infty}\ell(T)^{\frac{1}{2}}\IE\left[\mathcal{V}_{{T,\sigma}}(0)^2\right]=0,
\end{align*}
which follows from Lemma \ref{lem:lp} and the following:
\begin{lemma}{\cite[Lemma 3.3]{CN19}}
  Let $(X_k)_{k\in\N}$ be a non-negative, uniformly integrable family of random variables. Then, for any sequence $a_k\to\infty$, $a_k^{-1} \IE[(X_k\land a_k)^2]\to 0$  as $k\to\infty$.
  \end{lemma}
\end{proof}

\begin{proof}[Proof of Lemma \ref{lem:approxFF}]
The proof is essentially the same as  in Lemma \ref{lem:approxPsi}. 
Indeed, we can approximate $I_{\sigma T}(x)$ by \begin{align*}
\left(F'\left({\widetilde{\sZ}_{{ T\sigma\ell(T)},{ \ell'(\sigma,T)}}(x_T)}{\bar{u}(t,x)}\right){\widetilde{\sZ}_{{T\sigma\ell(T)},{ \ell'(\sigma,T)}}(x_T)}\right)\wedge {\ell(T)}^{-\frac{1}{2}}
\end{align*}
due to the same argument as (Step 2) and (Step 3) in the proof of Lemma \ref{lem:approxPsi}. In particular, we remark that its expectation converges to $I(x)$ due to Theorem \ref{thm:CSZ17b} and assumption of $F'$. We omit the detail.
\end{proof}

\begin{proof}[Proof of Lemma \ref{lem:MtildeTau0negl}]
 % The same arguments as in the beginning with Lemma \ref{lem:L2bdd} and (Step 1) in the proof of Lemma \ref{lem:approxPsi} give 
 {By \eqref{eq:Fass}, for $s\leq t T_\e,$
  $$\IE \left[(I_s^{(T)}(x))^2\right] \leq C \IE\left[\left(\left|\log \sZ_{{s\ell (T)}}^{\b}(x_T)\right|+\sZ_{{s\ell(T)}}^{\b}(x_T)\right)^2\right]\leq C_t,$$
  with some constants $C_t>0$ independent of $x,s,\e$. Hence,  we have:}  
  \begin{align*}
&\IE\left[\frac{1}{\b^2}\left(\int f(x) \kM_{T\tau_0}(x_T)\dd x\right)^2\right]\\
&= \int_{t{m(T)}}^{\tau_0}\dd \sigma \int_{(\R^2)^2}\dd x\dd y
f(x)f(y)
\IE\left[I_s^{(T)}(x)I_s^{(T)}(y)
%F'(\sZ^\b_{\ell(T\sigma)}(x_T)\bar{u}(t,x)) \sZ^\b_{\ell(T\sigma)}(x_T)
%F'(\sZ^\b_{\ell(T\sigma)}(y_T)\bar{u}(t,y)) \sZ^\b_{\ell(T\sigma)}(y_T)
\right]%\notag\\
%&\hspace{4em}
\int_{\R^2}\dd w\int _{|v|\leq R_V}\dd v \rho_\sigma\left(w+\frac{v}{\sqrt{T}}-x\right)\rho_\sigma(w-y)V(v)\notag\\
&{\hspace{6em} \times \IE\left[\overleftarrow{\sZ}^\b_{T\sigma,{ T\sigma\ell(T)}}(w_T+v)\overleftarrow{\sZ}^\b_{T\sigma,{ T\sigma\ell(T)}}(w_T) \right] 
\bar{u}\left(t-\sigma,w+\frac{v}{\sqrt{T}}\right)\bar{u}(t-\sigma,w)
%\DE_{w_T+v}\left[u_0\left(\frac{B_{tT-\sigma T}}{\sqrt{T}}\right)\right]
  % \DE_{w_T}\left[u_0\left(\frac{B_{tT-\sigma T}}{\sqrt{T}}\right)\right]
   }\\
&\leq C \int_{t{m(T)}}^{\tau_0}\dd \sigma \int_{(\R^2)^2}\dd x\dd y
|f(x)f(y)|
%\IE\left[I_s^{(T)}(x)I_s^{(T)}(y)\right]
%F'(\sZ^\b_{\ell(T\sigma)}(x_T)\bar{u}(t,x)) \sZ^\b_{\ell(T\sigma)}(x_T)
%F'(\sZ^\b_{\ell(T\sigma)}(y_T)\bar{u}(t,y)) \sZ^\b_{\ell(T\sigma)}(y_T)
%%\notag\\
%&\hspace{4em}
\int_{\R^2}\dd w \int_{|v|\leq R_V}\dd v\rho_\sigma\left(w+\frac{v}{\sqrt{T}}-x\right)\rho_\sigma(w-y)\notag\\
&\leq C\left(\tau_0-t{m(T)}\right),
  \end{align*}
  where we have used  {  $\overleftarrow{\sZ}^\b_{T\sigma,{ T\sigma\ell(T)}}(x)\eqlaw {\sZ}^\b_{T\sigma,{ T\sigma\ell(T)}}(x)$ with \eqref{eq:L2bdd}%, $\int V(v)\dd v=1$, $\rho_\sigma\left(z+\frac{v}{\sqrt{T}}\right)\leq 2\rho_\sigma\left(z\right)$ for $z\in\R^2$, $v\in {\rm supp}(V)$ and $T>0$ large enough, 
  and $\|u_0\|_\infty<\infty$ in the second line, and
    \begin{align*}
 &   \int_{(\R^2)^2}\dd x\dd y
|f(x)f(y)|
%\IE\left[I_s^{(T)}(x)I_s^{(T)}(y)\right]
%F'(\sZ^\b_{\ell(T\sigma)}(x_T)\bar{u}(t,x)) \sZ^\b_{\ell(T\sigma)}(x_T)
%F'(\sZ^\b_{\ell(T\sigma)}(y_T)\bar{u}(t,y)) \sZ^\b_{\ell(T\sigma)}(y_T)
%%\notag\\
%&\hspace{4em}
\int_{\R^2}\dd w \int_{|v|\leq R_V}\dd v\rho_\sigma\left(w+\frac{v}{\sqrt{T}}-x\right)\rho_\sigma(w-y)\\
&=\int_{(\R^2)^2}\dd x\dd y |f(x)f(y)|\int_{|v|\leq R_V}\dd v \rho_{2\sigma}\left(y-x+\frac{v}{\sqrt{T}}\right)\leq \|f\|_1^2
%\int_{(\R^2)^2}\dd x\dd y |f(x)f(y)| \int_{(\R^2)^2}\dd w \rho_\sigma\left(w-x\right)\rho_\sigma(w-y)= \int_{(\R^2)^2}\dd x\dd y |f(x)f(y)|\rho_{2\sigma}(x-y)\leq \|f\|_{\infty} \|f\|_1,
\end{align*}
    in the last line.
  }
\end{proof}

%The next section is divided in two parts: one is dedicated to the proof of Proposition \ref{prop:replaceByM} (which in turn will conclude the proof of Proposition \ref{prop:mainpropKPZ}),
% and one to the proof of Proposition \ref{prop:VanishBracket}.

\subsection{Proof of Proposition \ref{prop:replaceByM}}\label{sub:4.3}
%{We denote by $\langle M\rangle_s':=\frac{\dd \langle M\rangle_s}{\dd s}$ and $\langle M,N\rangle_s':=\frac{\dd \langle M,N\rangle_s}{\dd s}$ the derivatives of quadratic variation and the cross-bracket of continuous  martingale $M$ and $N$ in this section.}

%\subsubsection{Surgery of martingales (proof of Proposition \ref{prop:replaceByM})}

First, we will show that the fluctuation of martingale term is negligible at short time regime.
\begin{lemma} \label{lem:vanishingSmalltSurgery}
\begin{equation*}% \label{eq:vanishingSmalltSurgery}
 \lim_{\ve\to 0} \frac{1}{\b_\e^2}\IE\left[\left(\int f(x) G_{{ tT_\e m(T_\e)}}(x_{T_\e})\dd x \right)^2\right]=0.
 \end{equation*}
\end{lemma}
Then, we will prove that the remainder of martingale  can be comparable to $\mathcal{M}$ in the sense:
\begin{lemma}\label{lem:G-M}
\begin{align*}
\lim_{\e\to 0}\frac{1}{\b_{\e}^2}\IE\left[\left(\int f(x) \left(\left(G_{tT_\e}(x_{T_\e})-G_{{ tT_\e m(T_\e)}}(x_{T_\e})\right)-\mathcal{M}_{tT_\e}(x_{T_\e})\right)\dd x \right)^2\right]=0.
\end{align*}
\end{lemma}
Lemma \ref{lem:vanishingSmalltSurgery} and Lemma \ref{lem:G-M} comclude Proposition \ref{prop:replaceByM}.

\subsubsection{Proof of Lemma  \ref{lem:vanishingSmalltSurgery}}

To prove Lemma \ref{lem:vanishingSmalltSurgery}, we will  introduce a new martingale: Let \begin{align}
&n(T_\ve)=\exp\left(-(\log T_\ve)^{\frac{1}{2}-\frac{\delta}{2}}\right)\notag
\intertext{ and} %\begin{align*}
&\widetilde{\kW}_s(x)=\DE_x\left[\Phi_s(B)u_0\left(\frac{B_{tT}}{\sqrt{T}}\right):\mathtt{F}_{{ tTm(T)},{ \sqrt{tTn(T)}}}(B,x)\right].\label{eq:tildeWdef}
\end{align}

Lemma \ref{lem:vanishingSmalltSurgery} is concluded by the following two  lemmas.
\begin{lemma}\label{lem:GGtilde}
For any $f\in C_c^\infty(\R^2)$, $t>0$, $x\in\R^2$, and $\hat{\b}\in (0,1)$, \begin{align*}
\lim_{\e\to 0}\frac{1}{\b_\e}\IE\left[\left|G_{{ tT_\e m(T_\e)}}(x)-\int_0^{{ tT_\e m(T_\e)}}F'(%\widetilde
{ {\kW}_u(x))}\dd \widetilde{\kW}_u(x) \right|\right]=0.
\end{align*} 
\end{lemma}

\begin{lemma}\label{lem:tildeG}
For any $f\in C_c^\infty(\R^2)$, $t>0$, and $\hat{\b}\in (0,1)$, \begin{align*}
\lim_{\e\to 0}\frac{1}{\b_\e}\IE\left[\left|\int_{\R^2}\dd xf(x)\int_0^{{ tT_\e m(T_\e)}}F'(%\widetilde
{ {\kW}_u(x_{T_\e}))}\dd \widetilde{\kW}_u(x_{T_\e}) \right|\right]=0.
\end{align*} 
\end{lemma}

\begin{proof}[Proof of Lemma \ref{lem:GGtilde}]
%We have \begin{align}
%&\IE\left[\left|G_{ tTm(T)}(x)-\int_0^{{ tTm(T)}}F'(\widetilde{\kW}_u(x))\dd \widetilde{\kW}_u(x) \right|^2\right]\notag\\
%&\begin{matrix*}[l]
%&\dis \leq %2
%\IE\left[\left(G_{ tTm(T)}(x)-\int_0^{{ tTm(T)}}F'(%\widetilde
%{\kW}_u(x))\dd{{\kW}}_u(x)\right)^2\right]%\\
%&\dis +2\IE\left[\left(\int_0^{{ tTm(T)}}F'(\widetilde{\kW}_u(x))\dd{{\kW}}_u(x)-\int_0^{{ tTm(T)}}F'(\widetilde{\kW}_u(x))\dd{\widetilde{\kW}}_u(x)\right)^2\right].
%\end{matrix*}
%\label{eq:equationF}
%\end{align}
We have \begin{align*}
&\IE\left[\left|G_{ tTm(T)}(x)-\int_0^{{ tTm(T)}}F'({\kW}_u(x))\dd \widetilde{\kW}_u(x) \right|^2\right]\\
&=\IE\left[\left(\int_0^{tTm(T)}F'\left(\kW_u(x)\right)\dd \kW_u(x)-\int_0^{tTm(T)}F'\left(\kW_u(x)\right)\dd \widetilde{\kW}_u(x)\right)^2\right]\\
&=\IE\left[\int_0^{tTm(T)}F'\left(\kW_u(x)\right)^2\dd \left\langle \kW(x)-\widetilde{\kW}(x)\right\rangle_u\right]
\end{align*}
Then, the last expectation is written by 
\begin{align*}
&\IE\left[\int_0^{tTm(T)}\dd u F'\left(\kW_u(x)\right)^2%\right.%\\
%&\left.\times 
\DE_{x}\otimes \DE_{x}\left[V(B_u-\widetilde{B}_u)\Phi_u(B_u)\Phi_s(\widetilde{B}_u)u_0\left(\frac{B_{tT}}{\sqrt{T}}\right)u_0\left(\frac{\widetilde{B}_{tT}}{\sqrt{T}}\right):A_T(B,\widetilde{B},x)%\mathtt{F}_{{ tTm(T)},{ \sqrt{tTn(T)}}}(B,x)^c\cap \mathtt{F}_{{ tTm(T)},{ \sqrt{tTn(T)}}}(B',x)^c
%:F_{m(T),\sqrt{Tn(T)}}(B,x),F_{m(T),\sqrt{Tn(T)}}(\widetilde{B},x)
\right]\right],
\end{align*}
where we set \begin{align*}
A_T(B,\widetilde{B},x):=\mathtt{F}_{{ tTm(T)},{ \sqrt{tTn(T)}}}(B,x)^c\cap \mathtt{F}_{{ tTm(T)},{ \sqrt{tTn(T)}}}(\widetilde{B},x)^c.
\end{align*}

%The expectation in the first term in \eqref{eq:equationF} can be written by a property of stochastic integral  \begin{align*}
%&\IE\left[\left(\int_0^{{ tTm(T)}}\left(F'(\kW_u(x))-F'(\widetilde{\kW}_u(x))\right)\dd{{\kW}}_u(x)\right)^2\right]\\
%&= \IE\left[\int_0^{{ tTm(T)}}\left(F'(\kW_s(x))-F'(\widetilde{\kW}_s(x))\right)^2\dd \langle {\kW}(x)\rangle_s\right]\\
%&=\IE\left[\int_0^{{ tTm(T)}}\dd s\left(F'(\kW_s(x))-F'(\widetilde{\kW}_s(x))\right)^2\right.%\\
%&
%\left.
%\hspace{3em}\times 
%\DE_{x}\otimes \DE_{x}\left[V(B_s-\widetilde{B}_s)\Phi_s(B_s)\Phi_s(\widetilde{B}_s)u_0\left(\frac{B_{tT}}{\sqrt{T}}\right)u_0\left(\frac{\widetilde{B}_{tT}}{\sqrt{T}}\right)%:F_{m(T),\sqrt{Tn(T)}}(B,x),F_{m(T),\sqrt{Tn(T)}}(\widetilde{B},x)
%\right]\right],
{By using $|V(B_s-\widetilde{B}_s)|\leq \|V\|_\infty$ and \eqref{eq:Fass}, it is estimated  by }\begin{align}
&C\IE\left[\int_0^{tTm(T)}\dd uF'\left(\kW_u(x)\right)^2(\kW_u(x)-\widetilde{\kW}_u(x))^2\right]\notag\\
&\leq C\IE\left[\int_0^{tTm(T)}\dd u\left(\frac{1}{\kW_u(x)}+1\right)^2(\kW_u(x)-\widetilde{\kW}_u(x))^2\right]\notag\\%.
%\end{align*}
%& C\IE\left[\int_0^{{ tTm(T)}}\dd s\left(F'(\kW_s(x))-F'(\widetilde{\kW}_s(x))\right)^2\sZ_{s}^\b(x)^2\right]\\
%&\leq C\int_0^{{ tTm(T)}}\dd s\IE\left[\left|F'(\kW_s(x))-F'(\widetilde{\kW}_s(x))\right|^{2q}\right]^{\frac{1}{q}}\IE\left[\sZ_s^\b(x)^{2p}\right]^{\frac{1}{p}}
%\end{align*}
%Also, it can easily obtain that 
%\begin{align*}
&\leq C\int_0^{tTm(T)}\dd u%\IE\left[\left(\frac{1}{W_u(x)}+1\right)^2(W_u(x)-\widetilde{W}_u(x))^2\right]\leq 
\IE\left[\frac{\kW_u(x)-\widetilde{\kW}_u(x)}{\kW_u(x)}+\kW_u(x)-\widetilde{\kW}_u(x)+\left(\kW_u(x)-\widetilde{\kW}_u(x)\right)^2\right].\label{eq:FWtWbdd}
\end{align}
It is easy to see that \begin{align}
\IE\left[\left(\kW_u(x)-\widetilde{\kW}_u(x)\right)^2\right]&=\DE_{x}\otimes \DE_{x}\left[\exp\left(\int_0^u \b^2V(B_s-\widetilde{B}_s)\dd s\right)u_0\left(\frac{B_{tT}}{\sqrt{T}}\right)u_0\left(\frac{\widetilde{B}_{tT}}{\sqrt{T}}\right):A_T(B,\widetilde{B},x)%\mathtt{F}_{{ tTm(T)},{ \sqrt{tTn(T)}}}(B,x)^c\cap \mathtt{F}_{{ tTm(T)},{ \sqrt{tTn(T)}}}(B',x)^c
%:F_{m(T),\sqrt{Tn(T)}}(B,x),F_{m(T),\sqrt{Tn(T)}}(\widetilde{B},x)
\right]\notag\\
&\leq C\DE_{x}\otimes \DE_{x}\left[\exp\left(\int_0^u \b^2V(B_s-\widetilde{B}_s)\dd s\right):A_T(B,\widetilde{B},x)%\mathtt{F}_{{ tTm(T)},{ \sqrt{tTn(T)}}}(B,x)^c\cap \mathtt{F}_{{ tTm(T)},{ \sqrt{tTn(T)}}}(B',x)^c
%:F_{m(T),\sqrt{Tn(T)}}(B,x),F_{m(T),\sqrt{Tn(T)}}(\widetilde{B},x)
\right].\label{eq:FWtW}
\end{align}
Then, H\"older's inequality yields that \begin{align}
&\DE_{x}\otimes \DE_{x}\left[\exp\left(\int_0^u \b^2V(B_s-\widetilde{B}_s)\dd s\right):A_T(B,\widetilde{B},x)\right]\notag\\
&\leq \DE_{x}\otimes \DE_{x}\left[\exp\left(\int_0^u p\b^2V(B_s-\widetilde{B}_s)\dd s\right)\right]^{\frac{1}{p}}\DP_x\left(\mathtt{F}_{{ tTm(T)},{ \sqrt{tTn(T)}}}(B,x)^c\right)^{\frac{2}{q}},\label{eq:LW}
\end{align}
where $p,q>1$ with $\frac{1}{p}+\frac{1}{q}=1$ are chosen such that \begin{align*}
\varlimsup_{T\to\infty}\DE_{x}\otimes \DE_{x}\left[\exp\left(p\b^2\int_0^{Tt} V(B_u-\widetilde{B}_u)\dd u\right)\right]<\infty.
\end{align*}

\eqref{eq:FWtW} tends  to $0$ as $T\to\infty$ since we know \begin{align}
\DP_x\left(F_{{ tTm(T)},\sqrt{tTn(T)}}(B,x)^c\right)\leq C\exp\left(-\frac{n(T)}{4m(T)}\right),\label{eq:exit}
\end{align}
which decays faster than any polynomial of $T$. 
By using the Cauchy-Schwarz inequality with Lemma \ref{lem:negativemoment}, we find that \eqref{eq:FWtWbdd} converges to $0$ as $T\to \infty$. %the second expectation in \eqref{eq:equationF} is bounded by \begin{align*}
%&2\int_0^{{{ tTm(T)}}}\dd s\IE\left[F'(\widetilde{\kW}_s(x))^2\DE_{x}\otimes \DE_{x}\left[V(B_s-\widetilde{B}_s)\Phi_s(B)\Phi_s(\widetilde{B}): F_{{{ tTm(T)}},{ \sqrt{tTn(T)}}}(B,x)^c\right] \right]\\
%&\leq C\int_0^{{ tTm(T)}}\dd s\IE\left[F'(\widetilde{\kW}_s(x))^2\sZ_s^\b(x)\DE_x\left[\Phi_s(B):F_{{{ tTm(T)}},{ \sqrt{tTn(T)}}}(B,x)^c\right]\right]
%\end{align*}
%which tends to $0$ as $T\to \infty$.
\end{proof}

\begin{proof}[Proof of Lemma \ref{lem:tildeG}]
It is enough to show that \begin{align*}
&\frac{1}{\b^2}\IE\left[\left(\int_{\R^2}\dd xf(x)\int_0^{{{ tTm(T)}}}F'({\kW}_s(x_T))\dd \widetilde{\kW}_s(x_T)\right)^2\right]\\
&=\frac{1}{\b^2}	\int_{\R^2\times \R^2}\dd x\dd yf(x)f(y)\IE\left[\int_0^{{ tTm(T)}}F'({\kW}_s(x_T))F'({\kW}_s(y_T))\dd \langle \widetilde{\kW}(x_T),\widetilde{\kW}(y_T)\rangle_s\right]	\to 0,
\end{align*}
as $T\to \infty$.
Since $\widetilde{\kW}_s(x)$ and $\widetilde{\kW}_s(y)$  is independent and hence~$\dd \langle \widetilde{\kW}(x),\widetilde{\kW}(y)\rangle_u=0 $ if $|x-y|>2(\sqrt{tTn(T)}+R_\phi)$, we have \begin{align}
\IE\left[\left(\int_{\R^2}\dd xf(x)\int_0^{{ tTm(T)}}F'({\kW}_s(x_T))\dd \widetilde{\kW}_s(x_T)\right)^2\right]
&\leq Cn(T)\IE\left[\int_0^{{ tTm(T)}}F'({\kW}_s(x))^2\dd \langle \widetilde{\kW}(x)\rangle_s\right]\notag\\
&\leq Cn(T)\IE\left[\int_0^{{ tTm(T)}}\left(\frac{1}{{\kW}_s(x)^2}+1\right)\dd \langle \widetilde{\kW}(x)\rangle_s\right].\label{eq:L2error1}
\end{align}
By construction of $\widetilde{W}_s(x)$, we have \begin{align*}
&Cn(T)\IE\left[\int_0^{{ tTm(T)}}\left(\frac{1}{{\kW}_s(x)^2}+1\right)\dd \langle \widetilde{\kW}(x)\rangle_s\right]\leq Cn(T)\IE\left[\int_0^{{ tTm(T)}}\left(\frac{1}{{\kW}_s(x)^2}+1\right)\dd \langle {\kW}(x)\rangle_s\right]
\intertext{and furthermore, by applying It\^{o}'s lemma to $\log {\kW}(x)$ and ${\kW}(x)^2$, it is bounded by}
%\begin{align*}
& Cn(T)\IE\left[\log \bar{u}(t,x)-\log {\kW}_{{ tTm(T)}}+{\kW}_{{ tTm(T)}}(x)^2 \right]\leq Cn(T).
\end{align*}
Thus, Lemma \ref{lem:tildeG} is concluded.

\end{proof}

\subsubsection{Proof of Lemma \ref{lem:G-M}}
Define
\begin{align*} %\label{eq:defdL}
&\dd \mathcal{L}_s(x)={\b_\e} \sZ^\b_{{ s\ell(T)}}(x)\int_{\R^d}\dd z \xi(\dd s, \dd b) \int_{\R^d}\rho_s(z-x) \phi(z-b) \overleftarrow{\sZ}^\b_{s , { s\ell(T)}}(z) \, \DE_z\left[ u_0\left(\frac{B_{t T-s}}{\sqrt{T}}\right)\right].
\end{align*}
Then, we remark that \begin{align}
\dd \mathcal{M}_s(x_T)=F'\left(\sZ^\b_{{ s\ell(T)}}\bar{u}(t,x)\right)\dd \mathcal{L}_s(x_T)\label{eq:relLM}
\end{align}
for $s\geq tTm(T)$ and $x\in\R^2$.

Lemma \ref{lem:G-M} follows when the next two lemmas are proved.

% We first replace $\dd W$ by $\dd L$. \CC{in what ?}
\begin{lemma}\label{Diff W and L}
  For all $t>0$,
  \al{
    \lim_{\e\to 0}\frac{1}{\b_\e}\IE\left[\left| \int \dd xf(x) \left(\int_{{ tT_\e m(T_\e)}}^{tT_\e}F'(\kW_s(x_{T_\e}))\dd \kW_{s}(x_{T_\e})-\int_{{tT_\e m(T_\e)}}^{tT_\e}F'(\kW_s(x_{T_\e}))\dd \mathcal{L}_s(x_{T_\e})\right) \right|\right] =0.
  }

\end{lemma}

%\begin{comment}
%\begin{lemma}
%  For sufficietly large $r>0$,
%\al{  \lim_{T\to\infty}T^{\frac{d-2}{2}}\IE\left(\int^T_{T^{1-\ve}} \int f(x) \frac{\dd \kW_s(\sqrt{T}x)}{\kW_s(\sqrt{T}x)}\mathbf{1}(A_\sigma(x))\dd x \right)^2 =0.}
%\end{lemma}
%\begin{proof}
%  By \eqref{uniform-est},
  
%  \al{
%    &\quad \IE\left(\int^T_{T^{1-\ve}} \int f(x) \frac{\dd \kW_\sigma(\sqrt{T}x)}{\kW_\sigma(\sqrt{T}x)}\mathbf{1}(A_\sigma(x))\dd x \right)^2\\
%    &\leq \int^{T}_{T^{1-\ve}} \int f(x) f(y) \IE\left[\left(  \frac{ \lan \kW_\sigma(\sqrt{T}x),\kW_\sigma(\sqrt{T}y) \ran^\prime}{\kW_\sigma(\sqrt{T}x)\kW_\sigma(\sqrt{T}y)}\right)\mathbf{1}(A_\sigma(x))\right]\dd x \dd y \dd \sigma\\
%    &\leq |V|_{\infty}\int^{T}_{T^{1-\ve}} \int f(x) f(y) \IP\left(A_\sigma(x) \right)\dd x \dd y \dd \sigma\\
%    &\leq C T T^{-2d}=o(T^{\frac{d-2}{2}})
%    }
%\end{proof}
%\end{comment}
\begin{lemma}\label{LAST LEMMA}
For $t>0$,
\al{  \lim_{\e\to 0}\frac{1}{\b_\e}\IE\left[\left| \int f(x) \left(\int_{{ tT_\e m(T_\e)}}^{tT_\e}F'(\kW_s(x_{T_\e}))\dd \mathcal{L}_s(x_{T_\e})- \mathcal{M}_{tT_\e}(x_{T_\e})\right) \dd x \right|\right] =0.}
\end{lemma}

\begin{proof}[Proof of Lemma \ref{Diff W and L}] By the Burkholder-Davis-Gundy inequality, we have
  \al{
    &\frac{1}{\b}\IE\left[\left| \int \dd xf(x) \left(\int_{{ tTm(T)}}^{tT}F'(\kW_s(x_T))\dd \kW_{s}(x_T)-\int_{{ tTm(T)}}^{tT}F'(\kW_s(x_T))\dd \mathcal{L}_s(x_T)\right) \right|\right]  \\
&\leq \frac{C}{\b}\int \dd x|f(x)|\IE\left[\left(\int_{{ tTm(T)}}^{tT} F'(\kW_s(x_T))^2\dd \left\langle\kW(x_T)-\mathcal{L}(x_T)\right\rangle_s\right)^{\frac{1}{2}}\right]\\
&\leq \frac{C}{\b}\int \dd x|f(x)|\IE\left[\sup_{{{ tTm(T)}} \leq s\leq T}{\sZ_s(0)^
{-2}}+1\right]^{\frac{1}{2}}\IE\left[\int_{{ tTm(T)}}^{tT} \dd \left\langle\kW(x_T)-\mathcal{L}(x_T)\right\rangle_s\right]^{\frac{1}{2}}\\
&\leq  \frac{C}{\b}\int \dd x |f(x)|\IE\left[\int_{{ tTm(T)}}^{tT} \dd \left\langle\kW(x_T)-\mathcal{L}(x_T)\right\rangle_s\right]^{\frac{1}{2}},
%&\leq \int \dd x|f(x)| \int_{m(T)}^{T}\int_{(\R^2)^2}\dd z_1\dd z_2\rho_s(z_1-x_T)\rho_{s}(z_2-x_T)V(z_1-z_2)\bar{u}(tT-s,z_1)\bar{u}(tT-s,z_2)\\
%&\times \IE\left[F'(\kW_s)(x_T)^2		\left(\DE_{0,x_T}^{s,z_1}\left[\Phi_{s}^\b\right]-\sZ_{\ell(T)}(x_T)\overleftarrow{\sZ}_{s,\ell(T)}(z_1)\right)		\left(\DE_{0,x_T}^{s,z_2}\left[\Phi_{s}^\b\right]-\sZ_{\ell(T)}(x_T)\overleftarrow{\sZ}_{s,\ell(T)}(z_2)\right)		\right]\\
%    &\leq  \int f(x) \IE\left|\int^{T t}_{T^{1-\ve}}\frac{\dd (\kW(\sqrt{T}\,x)- L(\sqrt{T}\,x))_s}{\kW_s(\sqrt{T}\,x)} \right|\dd x \\
%    &\leq C\int |f(x)| \IE\left[\left( \int^{Tt}_{T^{1-\ve}}\frac{\dd \lan \kW(\sqrt{T}\,x)- L(\sqrt{T}\,x)\ran_s }{\kW_s(\sqrt{T}\,x)^2} \right)^{\frac{1}{2}}\right] \dd x,
}
  where we have used  Doob's inequality and Lemma \ref{lem:negativemoment} in the last inequality.
%  , we know that the variable $\sup_{\sigma\geq 0} \sZ_\sigma(\sqrt{T}\,x)^{-1}$ admits any moment. Therefore, by Cauchy-Schwarz inequality and $\sZ_{\sigma}\leq |u_0^{-1}|_{\infty}\, \kW_{\sigma}$, the last expectation is further bounded from above by

By Cauchy-Schwarz inequality and boundedness of $\bar{u}$, we have from definition of $\kW$ and $\mathcal{L}$ that for $x\in \R^2$\begin{align*}
&\frac{1}{\b_\e^2}\IE\left[\int_{{ tTm(T)}}^{tT} \dd \left\langle\kW(x_T)-\mathcal{L}(x_T)\right\rangle_s\right]\\
&= \int_{{tTm(T)}}^{tT}\dd s\int_{(\R^2)^2}\dd z_1\dd z_2 \rho_s(z_1)\rho_s(z_2)V(z_1-z_2)\bar{u}(tT-s,z_1)\bar{u}(tT-s,z_2)\\
&\hspace{2em}\times \IE\left[\left(\DE_{0,0}^{s,z_1}{\left[\Phi_s^\b\right]}-\sZ_{{ s\ell(T)}}(0)\overleftarrow{\sZ}_{s,{ s\ell(T)}}(z_1)\right)\left(\DE_{0,0}^{s,z_2}\left[\Phi_s^\b\right]-\sZ_{{ s\ell(T)}}(0)\overleftarrow{\sZ}_{s,{ s\ell(T)}}(z_2)\right)\right]\\
&\leq C\int_{{ tTm(T)}}^{tT}\dd s\int_{(\R^2)^2}\dd z_1\dd z_2 \rho_s(z_1)\rho_s(z_2)V(z_1-z_2)\\
&\hspace{2em}\times \IE\left[\left|\DE_{0,0}^{s,z_1}\left[\Phi_s^\b\right]-\sZ_{{ s\ell(T)}}(0)\overleftarrow{\sZ}_{s,{ s\ell(T)}}(z_1)\right|\left|\DE_{0,0}^{s,z_2}\left[\Phi_s^\b\right]-\sZ_{{ s\ell(T)}}(0)\overleftarrow{\sZ}_{s,{ s\ell(T)}}(z_2)\right|\right]\\
&\leq C\int_{{{{ tm(T)}}}}^t\dd \sigma\int_{(\R^2)^2}\dd w\dd v
\rho_{\sigma}(w)\rho_\sigma(w+\frac{v}{\sqrt{T}})V(v)\\
&\hspace{2em}\times  \IE\left[\left(\DE_{0,0}^{T\sigma,w_T}\left[\Phi_{T\sigma}^\b\right]-\sZ_{{ s\ell(T)}}(0)\overleftarrow{\sZ}_{T\sigma,{ T\sigma\ell(T)}}(w_T)\right)^2\right]^{\frac{1}{2}}\\
&\hspace{2em}\times \IE\left[\left(\DE_{0,0}^{T\sigma,w_T+v}\left[\Phi_{T\sigma}^\b\right]-\sZ_{{ T\sigma\ell(T)}}(0)\overleftarrow{\sZ}_{T\sigma,{ T\sigma\ell(T)}}(w_T+v)\right)^2\right]^\frac{1}{2}\\
&\leq C\int_{{{{ tm(T)}}}}^t\dd \sigma\int_{\R^2}\dd w
\rho_{\sigma}(w)%\\
%&\hspace{2em}\times 
 \IE\left[\left(\DE_{0,0}^{T\sigma,w_T}\left[\Phi_{T\sigma}^\b\right]-\sZ_{{ s\ell(T)}}(0)\overleftarrow{\sZ}_{T\sigma,{ T\sigma\ell(T)}}(w_T)\right)^2\right]^{\frac{1}{2}}%\\
%&\hspace{2em}\times \IE\left[\left(\DE_{0,0}^{T\sigma,w_T+v}\left[\Phi_{T\sigma}^\b\right]-\sZ_{{ T\sigma\ell(T)}}(0)\overleftarrow{\sZ}_{T\sigma,{ T\sigma\ell(T)}}(w_T+v)\right)^2\right]^\frac{1}{2}
\intertext{and furthermore Lemma \ref{th:errorTermLLT} allows us to  bound it  from above by }
&C\int_{tm(T)}^t\dd \sigma \int_{|w_T|\leq \sqrt{\sigma T\log (\sigma T)}}\dd w\rho_\sigma(w)\left(\ell(T)-\frac{\log \ell(T)}{\log T}+\frac{\sqrt{\log (\sigma T)}\log( \sigma T\log \ell(T))}{\sqrt{\sigma T}\log T}+\frac{\ell(T)\log (\sigma T)}{\log T}\right)\\
&\hspace{3em}+C\int_{tm(T)}^t\dd \sigma \int_{|w_T|\geq \sqrt{\sigma T\log( \sigma T) }}\dd w\rho_{\sigma}(w)\\
&\leq C\left(\ell(T)-\frac{\log \ell(T)}{\log T}+\frac{\sqrt{\log (\sigma T)}\log( \sigma T\log \ell(T))}{\sqrt{\sigma T}\log T}+\frac{\ell(T)\log (\sigma T)}{\log T}\right)+C\int_{t{m(T)}}^t \frac{1}{\sqrt{\sigma T}}\dd \sigma.
\end{align*}
Since both terms in the last line tend to $0$ as $T\to \infty$ from   {definition of $m(T)$ and $\ell(T)$}, Lemma {\ref{Diff W and L}} is concluded. 
\end{proof}

\begin{proof}[Proof of Lemma \ref{LAST LEMMA}]
We have from \eqref{eq:relLM} 
\begin{align}
&\frac{1}{\b}\IE\left[\left| \int f(x) \left(\int_{{ tTm(T)}}^{tT}F'(\kW_s(x_T))\dd \mathcal{L}_s(x_T)- \mathcal{M}_{tT}(x_T)\right) \dd x \right|\right]\notag\\
&\leq \frac{1}{\b}\int |f(x)| \IE\left[\left| \int_{{ tTm(T)}}^{tT}F'(\kW_s(x_T))\dd \mathcal{L}_s(x_T)- \mathcal{M}_{tT}(x_T)\right|\right] \dd x \notag\\
&\leq \frac{1}{\b}\int \dd x|f(x)|\IE\left[\int_{{ tTm(T)}}^{tT}\left(F'(\kW_s(x_T))-F'(\sZ_{{ s\ell(T)}}(x_T)\bar{u}(t,x))\right)^2\dd \langle \mathcal{L}(x_T)\rangle_s\right]^\frac{1}{2}\notag
\end{align}
and \begin{align*}
&\IE\left[\int_{tTm(T)}^{tT}\left(F'(\kW_s(x_T))-F'(\sZ_{{ s\ell(T)}}(x_T)\bar{u}(t,x))\right)^2\dd \langle \mathcal{L}(x_T)\rangle_s\right]\\
&= \b^2\int_{tm(T)}^{t}\dd \sigma \int_{(\R^2)^2}\dd z\dd v\rho_\sigma (z)\rho_\sigma(z+\frac{v}{\sqrt{T}})V(v)\DE_{z_T}\left[u_0\left(\frac{B_{tT-\sigma T}}{\sqrt{T}}\right)\right]\DE_{z_T+v}\left[u_0\left(\frac{B_{tT-\sigma T}}{\sqrt{T}}\right)\right]\\
&\hspace{5em}\times \IE\left[\left(F'(\kW_{T\sigma }(x_T))-F'(\sZ_{{ T\sigma \ell(T)}}(x_T)\bar{u}(t,x))\right)^2\sZ_{{ T\sigma \ell(T)}}(x_T)^2\overleftarrow{\sZ}_{T\sigma,T\sigma \ell(T)}(z_T)\overleftarrow{\sZ}_{T\sigma,T\sigma \ell(T)}(z_T+v)\right].% \\
%&\leq C\b^2{\MN \int_{tm(T)}^{t}\dd \sigma \IE\left[\left(F'(\kW_{T\sigma }(x_T))-F'(\sZ_{{T\sigma \ell(T)}}(x_T)\bar{u}(t,x))\right)^2\sZ_{{ T\sigma \ell(T)}}(x_T)^2\right]},
\end{align*}
By \eqref{eq:Fass} and Lemma \ref{lem:negativemoment}, we may choose  $p,q>1$ such that $\frac{1}{p}+\frac{1}{q}=1$ and there exists a constant $C$ such that\begin{align*}
&\IE\left[\left|F'(\kW_{T\sigma }(x_T))-F'(\sZ_{{ T\sigma \ell(T)}}(x_T)\bar{u}(t,x))\right|^{2q}\right]^{\frac{1}{q}}\IE\left[\sZ_{{ T\sigma \ell(T)}}(x_T)^{2p}\overleftarrow{\sZ}_{T\sigma,T\sigma \ell(T)}(z_T)^p\overleftarrow{\sZ}_{T\sigma,T\sigma \ell(T)}(z_T+v)^p\right]^{\frac{1}{p}}\\
&\leq C
\end{align*}
uniformly in $tm(T)\leq \sigma \leq t$, $x\in \R^2$ and in $0<\e<\frac{1}{2}$.

Also, \eqref{eq:Fass} and  $\dis \left|F'(\kW_{T\sigma }(x_T))-F'(\sZ_{{ T\sigma \ell(T)}}(x_T)\bar{u}(t,x))\right|=\left|\int_{\sZ_{T\sigma \ell(T)}(x_T)}^{\kW_{T\sigma}(x_T)}F''(w)\dd w\right|$ yield
\begin{align*}
&\IE\left[\left|F'(\kW_{T\sigma }(x_T))-F'(\sZ_{{ T\sigma \ell(T)}}(x_T)\bar{u}(t,x))\right|^{2q}\right]\\
&\leq C\IE\left[\left|\frac{1}{\kW_{T\sigma }(x_T)}+\frac{1}{\sZ_{{ T\sigma \ell(T)}}(x_T)\bar{u}(t,x)}+1\right|^{2q-1}\left|F'(\kW_{T\sigma }(x_T))-F'(\sZ_{{ T\sigma \ell(T)}}(x_T)\bar{u}(t,x))\right|\right]\\
&\leq C\IE\left[\sup_{tm(T)\leq \sigma\leq t}\left\{\left|\frac{1}{\sZ_{T\sigma }(x_T)}+\frac{1}{\sZ_{{ T\sigma \ell(T)}}(x_T)}+1\right|^{2q-1}\left|\frac{1}{\sZ_{T\sigma }(x_T)^2}+\frac{1}{\sZ_{{ T\sigma \ell(T)}}(x_T)^2}+1\right|\right\}\right.\\
&\hspace{5em}\times 	\left|\kW_{T\sigma }(x_T)-\sZ_{{ T\sigma \ell(T)}}(x_T)\bar{u}(t,x)	\right|	\Bigg].
\end{align*}
Then, H\"older's inequality and Doob's inequality guarantee with Lemma \ref{lem:negativemoment} and Theorem \ref{th:errorTermLLT} that the last term converges to $0$ as $T\to \infty$ and thus Lemma \ref{LAST LEMMA} follows by the dominated convergence theorem.
\end{proof}

\subsection{Proof of Proposition \ref{prop:VanishBracket}}

Our goal is to prove that: %(we forget the $t$ dependence since it plays no role here):
\begin{equation} \label{eq:goalOfReduction}
\frac{1}{\b_\e}  \int_{\R^2}\dd xf(x) \left( \int^{tT_\e}_0F''(\kW_s(x_{T_\e})){\dd \lan \kW(x_{T_\e}) \ran_s}-\IE \left[ \int^{tT_\e}_0F''(\kW_s(x_{T_\e})){\dd \lan \kW(x_{T_\e}) \ran_s}\right]\right)  \cvLone 0.
\end{equation}

For simplicity of notation, we set $t=1$ hereafter. 

The proof is composed of four steps. %We first deal with the {\MN ``large $s$"} regime.
In the first step, we will  investigate that the influence at large time is negligible in the following sense:

\begin{lemma}[Step 1]\label{lem:firstStep}
%Fix {\MN $\delta>0$} and $t>0$. %Let $Tm(T)=Te^{-(\log T)^{\frac{1}{2}-{\MN \delta}}}$.
% Then,
  \aln{
 \lim_{\e\to 0} \frac{1}{\b_\e}  \IE\left[\int_{\R^2} \dd xf(x) \int^{T_\e}_{T_\e m(T_\e)} F''(\kW_s(x_{T_\e})) \dd \lan \kW (x_{T_\e})\ran_s\right]=0.
  }
  \end{lemma}

Before going to the step 2, we introduce a stopping time\begin{align*}
\tau_{T_\e}=\tau_{T_\e}(x):=\inf\left\{s\geq 0:W_s(x)+\widetilde{\kW}_s(x)^{-1}>\frac{1}{m(T_\e)}\right\}\land T_\e m(T_\e).
\end{align*}
Let us define the event:
$$A_{T_\e}(x)=\{\tau_{T_\e}(x)= T_\e m(T_\e)\}=\left\{\kW_s(x)+\widetilde{\kW}_s(x)^{-1} \leq  m(T_\e)^{-1}\text{ for all }s\leq T_\e m(T_\e), \right\}.$$

In the step 2, we will find that the contribution of ``large" $\kW(x)$ and ``small" $\widetilde{\kW}(x)$ can be negligible.  
\begin{lemma}[Step 2]\label{lem:step2}
  \begin{align*}
\lim_{\e\to 0}\frac{1}{\b_\e}\IE\left[\int_{\tau_{T_{\e}}(x)}^{{  T_\e m(T_\e})} |F''(\kW_s(x))| \dd \lan \kW(x)\ran_s\right]=0.
\end{align*}
\end{lemma}

In the step 3, we will prove that the contributions by  $\kW(x)$ and $\widetilde{\kW}(x)$ are asymptotically identified.

\begin{lemma}[Step 3]\label{lem:step2.5}
For any $x\in\R^2$, \begin{align*}
\lim_{\e\to 0}\frac{1}{\b_\e}\IE\left[\left|\int_0^{\tau_{T_\e}(x)} F''(\kW_s(x)){\dd \langle \kW(x)\rangle_s}-\int_0^{\tau_{T_\e}(x)} F''(\widetilde{\kW}_s(x)){\dd \langle \widetilde{\kW}(x)\rangle_s}\right| \right]=0.
\end{align*}
\end{lemma}

In the last, we will prove the remainder is also negligible.

\begin{lemma}[Step 4]\label{lem:step3} 
\begin{align*}
\lim_{\e\to 0}\frac{1}{\b_\e^2}\IE\left[\left(\int_{\R^2}\dd x f(x)\left(\int_0^{{{\tau_{T_\e}}}} F''(\widetilde{\kW}_u(x_{T_\e})) \dd \langle \widetilde{\kW}(x_{T_\e})\rangle_u-\IE\left[\int_0^{\tau_{T_\e}}  F''(\widetilde{\kW}_u(x_{T_\e})) \dd \langle \widetilde{\kW}(x_{T_\e})\rangle_u\right)\right] \right)^2\right]=0.
\end{align*}
\end{lemma}
 
 Putting these lemmas together, Proposition \ref{prop:VanishBracket} is concluded.
  
\begin{proof}[Proof of Lemma \ref{lem:firstStep}]
With $\delta $ from \eqref{def: ell(T)}, for $C>(\|u_0\|_\infty+\|u_0^{-1}\|_\infty)^4$,
\begin{align*}
&\IE\left[\int^{T}_{Tm(T)} \left|F''(\kW_s(x_{T_\e}))\right| \dd \lan \kW (x_T)\ran_s \right]\\
&\leq C \IE\left[\sup_{{ Tm(T)}\leq s\leq T}\left(1+\sZ_s(x)^{-2}\right)\int^{T}_{{ Tm(T)}} \dd \lan \sZ (x_T)\ran_s \right]	\\
&\leq 2C\IE\left[\sup_{{ Tm(T)}\leq s\leq T}\sZ_s(x)^{-2}\mathbf{1}\left\{\sup_{{ Tm(T)}\leq s\leq T}\sZ_s(x)^{-1}> (\log T)^{\delta/4}\right\}\int^{T}_{{ Tm(T)}} \dd \lan \sZ (x_T)\ran_s \right]\\
&\qquad\qquad+2C(\log{T})^{\delta/2}\IE\left[\mathbf{1}\left\{\sup_{{ Tm(T)}\leq s\leq T}\sZ_s(x)^{-1}\leq (\log T)^{\delta/4}\right\} \int^{T}_{{ Tm(T)}}\dd \lan \sZ (x_T)\ran_s \right].
\end{align*}
By the Burkholder-Davis-Gundy inequality, Doob's inequality and H\"older's inequality,
%{lem:negativemoment}, 
the first expectation is bounded from above by 
\begin{align*}
&\IP\left(\sup_{{ Tm(T)}\leq s\leq T}\sZ_s(x)^{-1}> (\log T)^{\delta/4}\right)^{\frac{1}{2p}} \IE\left[\frac{1}{\sZ_{T}(x)^{4p}}\right]^{\frac{1}{2p}}\IE\left[\left(\int^{T}_{{ Tm(T)}} \dd \lan \sZ (x_T)\ran_s\right)^{q}\right]^{\frac{1}{q}}\\
&{ \leq \frac{1}{(\log T)^{2}}\IE\left[\sZ_{T}(x)^{-\frac{16 p}{\delta}}\right]^{\frac{1}{2p}}\IE\left[\frac{1}{\sZ_{T}(x)^{4p}}\right]^{\frac{1}{2p}}\IE\left[\sZ_{T}(x)^{2q}\right]^{\frac{1}{q}}}\\
&\leq (\log{T})^{-2},
\end{align*}
where $p,q>1$ with $\frac{1}{p}+\frac{1}{q}=1$ and $\sup_{\e<1} \IE\left[\sZ_{T}(x)^{2q}\right]<\infty$ from Lemma~\ref{lem:lp}. On the other hand, the second expectation can be bounded from above by
\begin{align*}
(\log{T})^{\delta/2}\IE\left[\int^{T}_{{ Tm(T)}} \dd \lan \sZ (x_T)\ran_s \right]
&=(\log T)^{\delta/2}\int_{Tm(T)}^{T}\beta^2 \DE_{0}\left[V(\sqrt{2}B_s)\exp\left(\int_0^s\beta^2  V(\sqrt{2}B_u)\dd u\right)\right]\dd s\\
%=C(\log{T})^{\delta/2}\left(\DE_0\left[\exp\left(\b^2\int_{Tm(T)}^{T}V(\sqrt{2}B_s)\right)\right]-1\right)\\
&=(\log T)^{\delta/2}\int_{Tm(T)}^{T}\beta^2 \int_{\R^2}\dd x V(\sqrt{2}x)\rho_s(x)\DE_{0}^x\left[\exp\left(\int_0^s\beta^2  V(\sqrt{2}B_u)\dd u\right)\right]\dd s\\
&\leq \b^2(\log{T})^{\delta/2} (\log (T)-\log ( { Tm(T)}))\\
&\leq -\b^2 (\log{T})^{\delta/2} \log m(T),
\end{align*}
where we have used Lemma \ref{lem:p2pge} in the third line and $\beta^2 (\log{T})^{\delta/2} \log m(T)\to 0$ as $\e\to 0$ as desired.
  \end{proof}

Before the proof of Lemma \ref{lem:step2}, we give an estimate of the probability of $A_T^c(x)$.

%\subsection{New attempt}
\begin{lemma}
  There exists a constant $C>0$ such that for $T_\e>0$ and { for $x\in\R^2$},
  \al{
  \IP\left(A_{T_\e}(x)^c\right)\leq C m(T_\e).
}
\end{lemma}

\begin{proof}
We have 
\al{
  \IP\left( A_T(x)^c\right)\leq \P(\kW_s(x) > (2m(T))^{-1}\text{ for some }  s\in [0, Tm(T)])+\P(\widetilde{\kW}_s(x) < 2m(T)\text{ for some }s\in [0, Tm(T)]).
}
The first term is bounded from above by $2\|u_0\|_\infty m(T)$ using Doob's inequality and $\IE \left[\kW_{Tm(T)}(x)\right]\leq \|u_0\|_\infty$. Using the fact that $B<x$ implies $A<2x$ or $A-B>x$ for $A\geq B>0$ and $x>0$, by Doob's inequality with (sub-)martingales $\kW_s(x)^{-1},\,\kW_s(x)-\widetilde{\kW}_s(x)$ the second term is bounded from above by
\al{
  & \P(\kW_s(x) < 4m(T)\text{ for some } s\in [0, Tm(T)])+\P(\kW_s(x)-\widetilde{\kW}_s(x) > 2m(T)\text{ for some } s\in [0, Tm(T)])\\
  &\leq 4m(T) \E[\kW_{Tm(T)}(x)^{-1}]+  m(T)^{-1} \E[\kW_{Tm(T)}(x)-\widetilde{\kW}_{Tm(T)}(x)]\leq Cm(T).
}
\end{proof}

%Now we are ready to prove \ref{prop:VanishBracket}.  The proof is composed of three steps. 

\begin{proof}[Proof of Lemma \ref{lem:step2}]
  By H\"older's inequality and Minkowski's inequality, the expectation is bounded from above by
\begin{align*}
&C \IE\left[\int_{\tau_T}^{{ Tm( T)}}\left(\frac{1}{\kW_s(x)^2}+1\right)\dd \lan \kW(x)\ran_s\right]\\
&\leq C \IE\left[\int_{0}^{{  Tm(T)}}\left(\frac{1}{\kW_s(x)^2}+1\right)\dd \lan \kW(x)\ran_s;~A_T(x)^c\right]\\
&\leq C' \IE\left[\left(\int_{0}^{{ Tm(T)}}\left(\frac{1}{\kW_s(x)^2}+1\right)\dd \lan \kW(x)\ran_s \right)^p\right]^{\frac{1}{p}} \IP\left(A_T(x)^c\right)^{\frac{1}{q}}\\
&\leq C' \left(\IE\left[\left(\int_{0}^{{ Tm(T)}}\frac{\dd \lan \kW(x)\ran_s}{\kW_s(x)^2} \right)^p\right]^{\frac{1}{p}}+\IE\left[\left(\int_{0}^{{ Tm(T)}}\dd \lan \kW(x)\ran_s \right)^p\right]^{\frac{1}{p}} \right)\IP\left(A_T(x)^c\right)^{\frac{1}{q}},
\end{align*}
where $p,q>1$ with $\frac{1}{p}+\frac{1}{q}=1$ are constant with $2p<p_{\hat{\b}}$.
Then, by applying the Burkholder-Davis-Gundy inequality to the  martingales $\dis \int_0^{s} \dd \kW_u(x)=\kW_s(x)-\bar{u}(1,x)$,
we obtain that % and $\dis \int_0^{s}\frac{\dd \kW_u(x)}{\kW_u(x)}=\log \kW_s(x)-\log \bar{u}(t,x)+\frac{1}{2}\int_0^s \frac{1}{\kW_u(x)^2}\dd \langle \kW(x)\rangle_u$, the expectation is bounded by  that 
\begin{align*}
\IE\left[\left(\int_0^{Tm(T)}\dd \langle \kW(x)\rangle_s\right)^p\right]\leq C\IE\left[\left(\kW_{Tm(T)}(x)-\bar{u}(1,x)\right)^{2p}\right]\leq C.
\end{align*}

We remark that when we apply It\^{o}'s lemma  to $\log {\kW}_s(x)$, we have \begin{align*}
\log  {\kW}_s(x)&=\log \bar{u}(t,x)+\int_0^s \frac{\dd {\kW}_u(x)}{{\kW}_u(x)}-\frac{1}{2}\int_0^s\frac{\dd \langle {\kW}(x)\ran_u}{{\kW}_u(x)^2}\\
&:=\log \bar{u}(t,x)+{G}'_s(x)-\frac{1}{2}{H}'_s(x),
\end{align*} 
with  \begin{align*}
\left\langle {G}'(x)\right\rangle_s={H}'_s(x).
\end{align*}
In particular, we have\begin{align*}
\IE\left[{H}'_s(x)^2\right]&\leq 12\left(\log\bar{u}(t,x)\right)^2+12\IE\left[{G}'_s(x)^2\right]+12\IE\left[\left(\log {\kW}_s(x)\right)^2\right]\\
&=12\left(\log{u}(t,x)\right)^2+24\IE\left[H'(x)\right]+12 \IE\left[(\log {\kW}_s)^2\right]\\
&\leq C
\end{align*}
for  some constant $C>0$.
Putting things together with \eqref{eq:exit}, we have\begin{align*}
\frac{1}{\b_\e}\IE\left[\int_{\tau_T(x)}^{{  Tm(T)}} |F''(\kW_s(x))| \dd \lan \kW(x)\ran_s\right]\leq \frac{C}{\b}\IP\left(A_T(x)^c\right)^{\frac{1}{q}}\to 0.
\end{align*}
%
%Also, 
%\begin{align*}
%  &\leq C'\left(\IE\left[\left|\kW_{{ Tm(T)}}(x)\right|^{2p}\right]^{\frac{1}{p}}+\IE\left[\left|\log \kW_{{ Tm(T)}}(x)-\log \kW_0(x)\right|^{2p}\right]^{\frac{1}{p}}\right)\IP\left(A_T(x)^c\right)^{\frac{1}{q}}\\
%  &\leq C'' \IP\left(A_T(x)^c\right)^{\frac{1}{q}}
%\end{align*}
%
% and we have used and  Doob's inequality  in the third inequlaity.
\end{proof}
%Hence, it suffices to only estimate the integral over $A_T(x)$.

\begin{proof}[Proof of Lemma \ref{lem:step2.5}]
 Since $\kW_s(x)+\widetilde{\kW}_s(x)^{-1}\leq m(T)^{-1}$ for $s\leq \tau_T(x)$, we have that
  \begin{align}
&\left|\int_0^{{\tau_T}} |F''(\kW_s(x))| {\dd \langle \kW(x)\rangle_s}-\int_0^{{\tau_T}}F''({\widetilde{\kW}_s(x)}){\dd \langle \widetilde{\kW}(x)\rangle_s}\right|\notag\\
&\leq \left|\int_0^{{\tau_T}} F''(\kW_s(x))\dd \langle \kW(x)\rangle_s-\int_0^{{\tau_T}} F''(\kW_s(x))\dd \langle \widetilde{\kW}(x)\rangle_s\right|\notag\\
&+\left|\int_0^{{\tau_T}} \left(F''(\kW_s(x))-F''(\widetilde{\kW}_s(x))\right)\dd \langle \widetilde{\kW}(x)\rangle_s\right|\notag\\
&\leq Cm(T)^2 \int_0^{{\tau_T}}\dd s
 \DE_{x}{\otimes }\DE_x\left[V(B_s-\widetilde{B}_s)\Phi_s^\b(B)\Phi_s^\b(\widetilde{B}): \mathtt{F}_{Tm(T),\sqrt{Tn(T)}}(B,x)^c\cup \mathtt{F}_{Tm(T),\sqrt{Tn(T)}}(\widetilde{B},x)^c
\right]%\label{eq:F''1}
\notag\\
&+\b_\e^2\int_0^{{\tau_T}}\dd s
\left|F''(\kW_s(x))-F''(\widetilde{\kW}_s(x))\right|
 \DE_{x}{\otimes }\DE_x\left[V(B_s-\widetilde{B}_s)\Phi_s^\b(B)\Phi_s^\b(\widetilde{B})\right]. \notag%\label{eq:F''}
\end{align}
  Using H\"older's inequality, there exists a $p>2$ such that the first term is bounded from above by
  \begin{align*}
&\IE\left[\int_0^{Tm(T)}\dd s {\DE_{x}{\otimes }\DE_x\left[V(B_s-\widetilde{B}_s)\Phi_s^\b(B)\Phi_s^\b(\widetilde{B}): %{\MN F(B)^c\cup F(B')^c}
\mathtt{F}_{{Tm(T)},\sqrt{Tn(T)}}(B,x)^c\cup
\mathtt{F}_{{Tm(T)},\sqrt{Tn(T)}}(\widetilde{B},x)^c
\right]}\right]\\
&\leq \|V\|_\infty \int_0^{{Tm(T)}}\dd s\IE\left[{\DE_{x}\left[\Phi_s^\b(B)\Phi_s^\b(\widetilde{B}): %{\MN F(B)^c\cup F(B')^c}
\mathtt{F}_{{Tm(T)},\sqrt{Tn(T)}}(B,x)^c\right]}\right]\\
&\leq \|V\|_\infty C \int_0^{{Tm(T)}} \dd s \DP_x\left(\mathtt{F}_{{Tm(T)},\sqrt{Tn(T)}}(B,x)^c\right)^{\frac{1}{p}}.
\end{align*}

For the second term,  we first note that for each $s\leq \tau_T$,
$$|F''(\kW_s(x))-F''(\widetilde{\kW}_s(x))|=\left|\int^{\kW_s(x)}_{\widetilde{\kW}_s(x)}F'''(r)\dd r\right|\leq C(1+m(T)^{-3}) (\kW_s(x)-\widetilde{\kW}_s(x)),$$  and $\dis \frac{\dd \langle \widetilde{\kW}(x)\rangle_s}{ \dd s} \leq  \beta^2 \|V\|_\infty \widetilde{\kW}_s(x)^2 \leq \|V\|_{\infty} m(T)^{-2}$. Hence,
\al{
&\qquad \IE\left[ \left|\int_0^{{ \tau_T}}|F''(\kW_s(x))-F''(\widetilde{\kW}_s(x))|\dd \langle \widetilde{\kW}(x)\rangle_s\right|\right]\\
%&\leq \int_0^{{ Tm(T)}}\IE\left[\left|F''(\kW_s(x))-F''(\widetilde{\kW}_s(x))\right|  \langle %\widetilde{\kW}(x)\rangle'_s\right] \dd s\\
%&\leq \int_0^{{ Tm(T)}}\IE\left[\left|F''(\kW_s(x))-F''(\widetilde{\kW}_s(x))\right|  \langle %\widetilde{\kW}(x)\rangle'_s\right] \dd s\\
  &\leq C\|V\|_\infty (1+m(T)^{-3}) m(T)^{-2}\int_0^{{Tm(T)}} \IE\left[\left|\kW_s(x)-\widetilde{\kW}_s(x)\right|\right] \dd s\\
  & =  C\|V\|_\infty (1+m(T)^{-3}) m(T)^{-2}\int_0^{{ Tm(T)}} \DP_x\left(\mathtt{F}_{{Tm(T)},\sqrt{Tn(T)}}(B,x)^c\right) \dd s.
}
%It is easy to see that there exists a universal constant $c>0$ such that 
%\begin{align*}
%\DP_x\left(\mathtt{F}_{{ Tm(T)},\sqrt{Tn(T)}}(B,x)^c\right)=\exp\left(-\frac{c n(T)}{2{ m(T)}}\right),
%\end{align*}
%, in particular the power of $m(T)$, and therefore 
By \eqref{eq:exit}, the statement holds.
\end{proof}

%For $F\in \mathfrak{F}$, the proof of \eqref{eq:goalOfReduction} is a modification of the one for $F(x)=\log x$.

%The proof of Lemma \ref{lem:firstStep} goes well for $F\in \mathfrak{F}$ without any change.

\begin{proof}[Proof of Lemma \ref{lem:step3}]
%We remark that when we apply It\^{o}'s lemma  to $\log \widetilde{\kW}_s(x)$, we have \begin{align*}
%\log \widetilde{\kW}_s(x)&=\log \widetilde{u}(t,x)+\int_0^s \frac{\dd \widetilde{\kW}_u(x)}{\widetilde{\kW}_u(x)}-\frac{1}{2}\int_0^s\frac{\dd \langle \widetilde{\kW}(x)\ran_u}{\widetilde{\kW}_u(x)^2}\\
%&:=\log\widetilde{u}(t,x)+\widetilde{G}_s(x)-\frac{1}{2}\widetilde{H}_s(x),
%\end{align*} 
%where $\widetilde{u}(t,x)=\displaystyle \DE_x\left[u_0\left(\frac{B_{T}}{\sqrt{T}}\right):\mathtt{F}_{{ Tm(T)},\sqrt{Tn(T)}}(B,x)\right]$
%and \begin{align*}
%\left\langle \widetilde{G}(x)\right\rangle_s=\widetilde{H}_s(x).
%\end{align*}
%In particular, we have\begin{align*}
%\IE\left[\widetilde{H}_s(x)^2\right]&\leq 12\left(\log\widetilde{u}(t,x)\right)^2+12\IE\left[\widetilde{H}_s(x)\right]+12\IE\left[\left(\log \widetilde{\kW}_s(x)\right)^2\right]\\
%&=12\left(\log\widetilde{u}(t,x)\right)^2+24\IE\left[\log \widetilde{u}(t,x)-\log \widetilde{\kW}_s(x)\right]+12 \IE\left[(\log \widetilde{\kW}_s)^2\right]\\
%&\leq C
%\end{align*}
%for a some constant $C>0$.
  We define
  \al{
    \widetilde{H}_s(x)=\int_0^{s}  F''(\widetilde{\kW}_s(x)) \dd \langle \widetilde{\kW}(x)\rangle_s.
    }

We remark that for $|x-y|\geq 3\sqrt{Tn(T)}$\begin{align*}
\textrm{Cov}\left(\widetilde{H}_{{ Tm(T)}}(x),\widetilde{H}_{{ Tm(T)}}(y)\right)=0
\end{align*}
so that \begin{align*}
&\IE\left[\left(\int_{\R^2}\dd xf(x)\left([\widetilde{H}_{{ \tau_T}}(x_T)-\IE\left[\widetilde{H}_{{ \tau_T}}(x_T)\right]\right) \right)^2\right]\\
&=\int_{|x-y|\leq 3\sqrt{n(T)}}\dd x\dd yf(x)f(y)\textrm{Cov}\left(\widetilde{H}_{{ \tau_T}}(x_T),\widetilde{H}_{{ \tau_T}}(y_T)\right)\\
%&\leq \int_{|x-y|\geq T^{-\frac{1}{2}}n(T)}\dd x\dd yf(x)f(y)\textrm{Cov}\left(H_{{\MN \tau_T}}(x_T),H_{{\MN \tau_T}}(y_T)\right)\\
&\leq \int_{|x-y|\leq 3\sqrt{n(T)}}\dd x\dd y|f(x)f(y)|\,\E\left[\widetilde{H}_{{ \tau_T}}(x_T)^2\right]^{\frac{1}{2}}\E\left[\widetilde{H}_{{ \tau_T}}(y_T)^2\right]^{\frac{1}{2}}.
%&\leq Ctn(T).
\end{align*}
Since %for $s\leq \tau_T$,
\al{
  |\widetilde{H}_{\tau_T}(x)|\leq C \int^{\tau_T}_0 (1+\widetilde{\kW}_{s}(x)^{-1})^2 \dd \langle \widetilde{\kW}(x)\rangle_s\leq  C(1+m(T)^{-2})\int^{\tau_T}_0 \dd \langle \widetilde{\kW}(x)\rangle_{s},
}
by the Burkholder-Davis-Gundy inequality, we have
\al{
 \E \left[\widetilde{H}_{\tau_T}(x)^2\right]&\leq C (1+m(T)^{-2})^2 \E\left[\sup_{0\leq s\leq \tau_T}\left(\kW_{s}(x)-\bar{u}(1,x)\right)^4\right]\leq C (1+m(T)^{-2})^2 m(T)^{-4}.  
  }
Putting things together,% we get 
%\begin{align*}
%\textrm{Var}\left(\widetilde{H}_s(x)\right)\leq Cm(T)^{-8},
%\end{align*}
we have \begin{align*}
&\frac{1}{\b^2}\IE\left[\left(\int_{\R^2}\dd xf(x)\left(\int_0^{\tau_T} F''(\widetilde{\kW}_s(x_T))\dd \lan \widetilde{\kW}(x_T)\ran_s-\IE\left[\int_0^{\tau_{T}} F''(\widetilde{\kW}_s(x_T))\dd \lan \widetilde{\kW}(x_T)\ran_s\right]\right) \right)^2\right]\\
&\leq C\frac{n(T)}{\b^2m(T)^8}.
\end{align*} 
\end{proof}

%Combining Lemma \ref{lem:firststep}, Lemma \ref{lem:step2}, Lemma \ref{lem:step2.5}, and Lemma \ref{lem:step3}, we obtain Proposition \ref{prop:VanishBracket}.

\subsection{Multidimensional convergence in the EW limits}
\label{sec:KPZmultidim}
 To ease the presentation, we restrict ourselves to the case where $F(x)=x$, and $\hat{\b}\in (0,1)$ is fixed, although a repetition of the argument would lead to the result for the general initial conditions and the function $F$ that we have been considering. %Note that $I(x)=1$ in this case.

Also, we note that for all $0\leq t_1\leq \cdots\leq t_n= t$, $u_0^{(1)},\cdots,u_0^{(n)}\in C_b(\R^2)$, and $f_1,\cdots,f_n\in C_c^\infty(\R^2)$\begin{align*}
&(u_\e(t_1,u_0^{(1)},f_1),\cdots, u_\e(t_n,u_0^{(n)},f_n))\\
&\eqlaw \left(\kW^{\left(t,T,u_0^{(1)}\right)}_{T(t-t_1)}(Tt,f_1),\cdots,\kW^{\left(t,T,u_0^{(n-1)}\right)}_{T(t-t_{n-1})}(Tt,f_{n-1}),\kW^{\left(t,T,u_0^{(n)}\right)}_{0}(Tt,f_n)\right),
\end{align*}
where we define for fixed $t>0$ that  for $ u, s\geq 0$ and $x\in\R^2$ \begin{align*}
&\kW_{u}(s,x)=\kW_{u}^{(t,T,u_0)}(s,x)=\begin{cases}
\dis \DE_x\left[\Phi_{u,s}(B)u_0\left(\frac{B_{Tt-u}}{\sqrt{T}}\right)\right],\quad &0\leq u\leq s\\
u_0(x),&0\leq s\leq u.
\end{cases}
\intertext{and}
&\kW_{u}(s,f)=\kW_{u}^{(t,T,u_0)}(s,f)=\int_{\R^2}f(x)\kW_{u}^{(t,T,u_0)}(s,x_T)\dd x.
\end{align*}

Thus, it suffices to show is that jointly for finitely many $u\in[0,t]$, $u_0\in C_b(\R^2)$, and $f\in\mathcal C^\infty_c$, as $\e\to 0$, 
\begin{equation}
\frac{1}{\b_\e}  \int f(x) \left( \kW_{T u,T t}^{\left(t,T,u_0\right)}(x_T)- \bar{u}(t-u,x) \right)\dd x 
 \cvlaw \kU_{u}^{(t,u_0)}(t,f),%\quad \int_{\mathbb{R}^d} f(x)\mathscr U_1(t-u,x)\dd x.
 \label{desired convergence}
\end{equation}
where $\dis \left\{\kU_{u}^{(t,u_0)}(s,f):f\in C_c^\infty(\R^2),u_0\in C_b(\R), 0\leq u\leq s\leq t\right\}$ is centered Gaussian field with covariance  \begin{align*}
&\mathrm{Cov}\left(\kU_{u}^{(t,u_0)}(s,f),\kU_{u'}^{(t,u_0')}(s,f')\right)\\
&=\frac{1}{1-\hat{\b}^2}\int_{u\vee u'}^s \dd \sigma \int \dd x\dd y  f(x)f'(y)\int \dd z\rho_{\sigma-u}(x,z)\rho_{\sigma-u'}(y,z) \bar{u}(t-\sigma,z)\bar{u}'(t-\sigma,z).
\end{align*}
%where for any $0\leq U \leq T$ and $X\in \rd$, we set
%\begin{align*}
%\sZ_{U,T}(X;\xi)= E\big[\Phi_{U,T}(B) \big \vert {B_U=X}\big],
%\end{align*}
Following the same strategy as in Subsection \ref{IdeaG}, we are reduced to showing that
\begin{equation}\label{eq:finalGoalKPZmuli}
\frac{1}{\b_\e} \mathcal{M}^{\ssup {t,T,u_0}}_{u}(\tau,f)\cvlaw
 \kU_{u}^{(t,u_0)}(\tau,f)
%\int_{\mathbb{R}^2} f(x)\mathscr U(\tau-u,x)\dd x, 
\quad \text{jointly in } u\in[0,\tau], f\in\mathcal C^\infty_c,
\end{equation}
where (see \eqref{eq:defdMbartau})
\begin{align*}
&\mathcal{M}^{\ssup {t,T,u_0}}_{u}(\tau,f) := \begin{cases}
\dis \int_{\R^2} f(x) \int_{Tu+T(t-u)m(T)}^{T \tau} \dd \mathcal{M}_{u}^{\ssup{t,T,u_0}}(s,x_T)\dd x,\quad &\tau\geq Tu+T(t-u)m(T)\\
0,\quad &\tau \leq Tu+T(t-u)m(T)
\end{cases}
\intertext{and }
&\dd \mathcal{M}_{u}^{\ssup{t,T,u_0}}(s,x)
\\
&:=%\begin{cases}
\b_\e \sZ_{Tu,Tu+(s-Tu)\ell(T)}(x)\int \xi(\dd s,\dd b)\int \rho_{s-Tu}(x,z)\phi(z-b)\overleftarrow {\sZ}_{s,(s-Tu)\ell(T)}(z)\DE_z\left[u_0\left(\frac{B_{Tt-s}}{\sqrt{T}}\right)\right]\dd z.
%\end{cases}
\end{align*}
%As in Section \ref{sec:conclusionOfProofOfSHE},  convergence in \eqref{eq:finalGoalKPZmuli} follows from convergence of the { cross-brackets} $\langle\overline{M}^{\ssup T}(u_1,f),\overline{M}^{\ssup T}(u_2,g)\rangle_\tau$ towards the RHS of \eqref{eq:LimcovStruc} and  the multidimensional functional central limit for martingales (\cite[Theorem 3.11]{JS87}). See the proof of Proposition \ref{lm:CVLone}. Note that convergence of the quadratic variation comes from the argument of Lemma \ref{lem:asymptPsi}, c.f.\ the proof of Proposition \ref{prop:MwidetildeIsGaussianFlat}.
%\end{comment}
%{Multidimensional convergence for the stationary case (Theorem \ref{th:CV_stationaryKPZ}) comes again via exchanging limits as in Section \ref{subsec:conclusionGFFKPZ}.}

%\subsubsection{Preliminaries}
Then, for all $u\geq 0$ and $f\in\mathcal C^\infty_c$, $\tau \to \mathcal{M}^{(t,T)}_{u}(\tau,f)$ is a continuous martingale. In view of the desired convergence \eqref{desired convergence}, we have again in mind the functional CLT for martingales Theorem \ref{thm:JS}, so we are interested in the limit of the cross-bracket $\langle \mathcal{M}^{(t,T)}_{u_1}(\cdot,f_{{ 1}}),\mathcal{M}_{u_2}^{\ssup {t,T}}(\cdot,f_{{ 2}})\rangle_\tau$.
We have:
\begin{proposition}\label{lm:CVLone}
For all test functions $f$ and $f'$ in $\mathcal{C}_c^\infty$, $u_0,u_0'\in C_b(\R^2)$, and $0\leq u_2\leq u_1\leq t$, for all $\tau \geq u_1$,
\begin{align} 
&
\frac{1}{\beta_\e^2} \langle \mathcal{M}^{(t,T,u_0)}_{u_1}(\cdot,f_{{ 1}}),\mathcal{M}_{u_2}^{\ssup {t,T,u'_0}}(\cdot,f_{{ 2}})\rangle_\tau \nonumber\\
&\overset{L^1}{\longrightarrow}\frac{1}{1-\hat{\b}^2}\int_{u_1}^\tau \dd \sigma \int \dd x\dd y  f_1(x)f_2(y)\int \dd z\rho_{\sigma-u_1}(x,z)\rho_{\sigma-u_2}(y,z) \bar{u}(t-\sigma,z)\bar{u}'(t-\sigma,z), \label{eq:LimcovStruc}
\end{align}
as $\e\to 0$. 
\end{proposition}

%\begin{rem}
%The solution of $\mathcal{U}^{(t,f,x,\hat{\b},1)}$ has the Duhamel form \begin{align*}
%\mathcal{U}^{(t,f,x,\hat{\b},1)}=\int f(x)\dd x+\frac{1}{}\int \dd x\int_0^t  \int f(x)p_{t-s}(x,y)\xi(\dd s,\dd y)
%\end{align*}
%\[
%\mathrm{Cov}\big(\mathscr U(t-u_1,x), \mathscr U(t-u_2,y)\big)= \frac{1}{1-\hat{\b}^2} \int_0^{t-u_1} \rho(2\sigma {\MN -(u_1-u_2)},x-y) \dd \sigma.
%\]
%\end{rem}
\begin{proof}
%First observe that by a simple shift in time, we can again restrict ourselves to showing the lemma when $u_2 = 0$. Then,
For all $\tau \geq u_1+(t-u_1)m(T)$,
\begin{align*}
&\frac{1}{\b^2}\langle \mathcal{M}_{u_1}^{(t,T,u_0)}(\cdot,f_{{ 1}}),\mathcal{M}_{u_2}^{\ssup {t,T,u_0'}}(\cdot,f_{{ 2}})\rangle_\tau\\
&= 
\int_{\mathbb{R}^2\times \mathbb R^2} f_1(x)f_2(y) \dd x \dd y \int_{Tu_1+T(t-u_1)m(T)}^{T\tau}\dd  s \sZ_{Tu_1,Tu_1+(s-Tu_1)\ell(T)}(x_T) \sZ_{Tu_2,Tu_2+(s-Tu_2)\ell(T)}(y_T)\\
&\hspace{3em}\times \int_{\R^2\times \R^2}\dd z_1\dd z_2\rho_{s-Tu_1}(z_1-x_T)\rho_{s-Tu_2}(z_2-y_T)V(z_1-z_2)\overleftarrow{\sZ}_{s,(s-Tu_1)\ell(T)}(z_1)\overleftarrow{\sZ}_{s,(s-Tu_2)\ell(T)}(z_2)\\
&\hspace{5em}\times \DE_{z_1}\left[u_0\left(\frac{B_{Tt-s}}{\sqrt{T}}\right)\right]\DE_{z_2}\left[u_0'\left(\frac{B_{Tt-s}}{\sqrt{T}}\right)\right]
%& \times \int_{\mathbb R^d} V(Z) \rho_{\sigma-u}(z-x) \rho_\sigma( z-y+T^{-\frac{1}{2}}Z) \DE_{Tu,\sqrt T x}^{T\sigma,\sqrt T z} \left[\Phi_{Tu,T\sigma}\right] \DE_{0,\sqrt T y}^{T\sigma,\sqrt T z+Z} \left[\Phi_{T\sigma}\right] \dd  z \dd Z.
\end{align*}
By a repetition of the arguments that lead to \eqref{eq:CVbracketMt}, %(namely \eqref{eq:decomposemartingale}-\eqref{eq:centralLimiThMGeneral}, with in particular Lemma \ref{lem:approxPsi} and the proof of Proposition \ref{prop:MwidetildeIsGaussianFlat}),
 we find that
\begin{align*} 
&\frac{1}{\b^2}\langle \mathcal{M}_{u_1}^{(t,T,u_0)}(\cdot,f_{{ 1}}),\mathcal{M}_{u_2}^{\ssup {t,T,u_0'}}(\cdot,f_{{ 2}})\rangle_\tau\\
&\approx_{L^1}\int f_1(x)f_2(y) \dd x \dd y \int_{u_1+(t-u_1)m(T)}^{\tau}  \dd \sigma\IE\left[\sZ_{Tu,Tu+T(\sigma-u)\ell(T)}(x_T)\right] \IE\left[ \sZ_{Tu_2,Tu_2+(s-Tu_2)\ell(T)}(y_T)\right] \Theta_T(x,y),
\end{align*}
where
\begin{align*}
\Theta_T(x,y) & =  \int \dd z\dd v\rho_{\sigma-u_1}(z-x)\rho_{\sigma -u_2}(z-\frac{v}{\sqrt{T}}-y)V(v)\IE\left[\overleftarrow{\sZ}_{\sigma,(\sigma-u_1)\ell(T)}(z_T)\overleftarrow{\sZ}_{\sigma,(\sigma-u_2)\ell(T)}(z_T+v)\right]\\
&\hspace{5em}\times \DE_{z_T}\left[u_0\left(\frac{B_{Tt-T\sigma}}{\sqrt{T}}\right)\right]\DE_{z_T+v}\left[u'_0\left(\frac{B_{Tt-T\sigma}}{\sqrt{T}}\right)\right]\\%\\
% \int  V(v) \rho_{\sigma-u}(z-x) \rho_\sigma( z-y) \IE\left[\overleftarrow{\sZ}_{T\sigma,T(\sigma-u)\ell(T)}(z_T) \overleftarrow{\sZ}_{T\sigma,T\sigma \ell(T)}(z_T + v)\right] \dd  z \dd v \\
& \to  \frac{1}{1-\hat{\b}^2}  \int \dd z\rho_{\sigma -u_1}(x-z)\rho_{\sigma-u_2}(y-z)\bar{u}(t-\sigma,z)\bar{u}'(t-\sigma,z).
\end{align*}
%which gives \eqref{eq:LimcovStruc} when $u_2=0$.

\end{proof}

\appendix

\section*{Acknowledgments}
%We thank Nikos Zigouras, Dimitris Lygkonis,  Ofer Zeitouni and David Belius for interesting discussions and helpful suggestions. The work of S. Nakajima is partially supported by  JSPS KAKENHI 19J00660 and SNSF grant 176918. C. Cosco acknowledges that this project has received funding from the European Research Council (ERC) under the European Union's Horizon 2020 research and innovation program (grant agreement No. 692452). We also thank Chiranjib Mukherjee for bringing to our attention the work \cite{LZ20}.
The work of S. Nakajima is supported by  SNSF grant 176918. M.~Nakashima is supported by JSPS KAKENHI Grant Numbers JP18H01123, JP18K13423.


\begin{thebibliography}{99}%{WWWW98}
\bibitem[AKQ14]{AKQ14}
{\sc T. Alberts}, {\sc K. Khanin} and {\sc J. Quastel},
\newblock{The intermediate disorder regime for directed polymers in dimension 1+1,}
\newblock{\it Ann. Probab.} {\bf 42}  (2014) 1212-1256


\bibitem[AKQ14b]{AKQ14b}
{\sc T. Alberts}, {\sc K. Khanin} and {\sc J. Quastel},
\newblock{The continuum directed random polymer,}
\newblock{\it J. Stat. Phys.} 
{\bf 154} (2014) 305-326

\bibitem[BC20]{BC20}
{\sc E. Bates}, \textsc{S. Chatterjee} \newblock The endpoint distribution of directed polymers, \newblock \textit{Ann. Probab.}\ Volume 48, Number 2 (2020), 817-871.






\bibitem[BT10]{BT10} \textsc{Q. Berger} and  \textsc{F. Toninelli}, \newblock On the critical point of the random walk pinning model in dimension $d = 3$, \newblock \textit{Elect.
J. Prob.} 15, 654-683, (2010).



\bibitem[BC95]{BC95}
{\sc L. Bertini} and {\sc N. Cancrini},
\newblock{The stochastic heat equation: Feynman-Kac formula and intermittence,}
\newblock{\it J. Statist. Phys.} 78(5-6):1377-1401, (1995).

\bibitem[BG97]{BG97}
{\sc L. Bertini} and {\sc G. Giacomin}.
\newblock Stochastic {B}urgers and {KPZ} equations from particle systems.
\newblock {\em Comm. Math. Phys.}, Vol. 183, No.~3, pp. 571--607, 1997.


\bibitem[B99]{B99}
 \textsc{P. Billingsley}:
\newblock {Convergence of probability measures},
\newblock Second.
\newblock New York~: John Wiley \& Sons Inc., 1999
\newblock \emph{Wiley Series in Probability and Statistics: Probability and
  Statistics}. 
\newblock A Wiley-Interscience Publication. 
\newblock ISBN 0-471-19745-9

 \bibitem[B89]{B89}
  {\sc E. Bolthausen}, \newblock{A note on the diffusion of directed polymers in a random environment,} \newblock {\em Commun. Math. Phys.} 123(4), 529-534 (1989) 




\bibitem[BGH11]{BGH11} \textsc{M. Birkner}, \textsc{A. Greven} and \textsc{F. den Hollander}, \newblock Collision local time of transient random walks and intermediate
phases in interacting stochastic systems, \newblock \textit{Elec. J. Probab.} 16, 552-586, (2011)

\bibitem[BS10]{BS10} \textsc{M. Birkner} and \textsc{R. Sun}, \newblock Annealed vs quenched critical points for a random walk pinning model, \newblock \textit{Ann. Henri Poinc.}, Prob. et Stat., Vol. 46, No. 2, pp. 414-441, (2010).

\bibitem[BS11]{BS11} \textsc{M. Birkner} and \textsc{R. Sun}, \newblock Disorder relevance for the random walk pinning model in dimension 3,\newblock \textit{Ann Henri
Poincar\'{e}.} Prob. et Stat., Vol. 47, No. 1, pp. 259-293, (2011) .

\bibitem[BM19]{BM19}
\textsc{Y. Br\"{o}ker} and \textsc{C. Mukherjee}, 
\newblock{Localization of the Gaussian multiplicative chaos in the Wiener space and the stochastic heat equation in strong disorder,} \newblock \textit{Ann. Appl. Probab.}
    29, 6 (2019), 3745-3785.

\bibitem[CSZ17a]{CSZ17a} \textsc{F. Caravenna}, \textsc{R. Sun} and \textsc{N. Zygouras}, \newblock Polynomial chaos and scaling limits of disordered systems, \newblock \textit{J. Eur.
Math.} Soc. 19 (2017), 1-65

\bibitem[CSZ17b]{CSZ17b}
\textsc{F. Caravenna}, \textsc{R. Sun} and \textsc{N. Zygouras}, \newblock {Universality in marginally relevant disordered systems}, \newblock \textit{ Ann. Appl.
Prob.} 27 (2017), 3050-3112.



 
 


\bibitem[CSZ19a]{CSZ19a}
\textsc{F. Caravenna}, \textsc{R. Sun} and \textsc{N. Zygouras}, \newblock The Dickman subordinator, renewal theorems, and disordered systems \newblock \textit{Electron. J. Probab.}
    Volume 24 (2019), paper no. 101, 40 pp.

\bibitem[CSZ19b]{CSZ19b}
\textsc{F. Caravenna}, \textsc{R. Sun} and \textsc{N. Zygouras}, \newblock On the Moments of the $(2+1)$-Dimensional Directed Polymer and Stochastic Heat Equation in the Critical Window \newblock \textit{Communications in Mathematical Physics},  Volume {372}, (2019), No. 2, {385--440},



\bibitem[CSZ20]{CSZ20}\textsc{F. Caravenna}, \textsc{R. Sun} and \textsc{N. Zygouras}, \newblock {The two-dimensional KPZ equation in the entire subcritical regime,} 
\newblock \textit{Ann. Prob.}   Volume {48}, (2020), No. 3,
  {1086--1127}.
  


\bibitem[CH02]{CH02}
{\sc P. Carmona} and {\sc Y. Hu},
\newblock{On the partition function of a directed polymer in a Gaussian random environment,}
\newblock{\it Prob. Th. Rel. Fields.}, {\bf 124} (2002) 431-457

\bibitem[CD20]{CD20} \textsc{S. Chatterjee} and \textsc{A. Dunlap}, \newblock Constructing a solution of the $(2+1)$-dimensional KPZ equation, \newblock \textit{Ann. Prob.,}
48 (2020), no. 2, 1014-1055. 



\bibitem[CL17]{CL17}
{\sc F. Comets} and {\sc Q. Liu},
\newblock{ Rate of convergence for polymers in a weak disorder}, 
\newblock{\it J. Math. Anal. Appl.} {\bf 455} (2017), 312-335



\bibitem[CN95]{CN95}
{\sc F. Comets} and {\sc J. Neveu},
\newblock{The Sherrington-Kirkpatrick model of spin glasses and stochastic calculus: the high temperature case},
\newblock{\it Comm. Math. Phys.} {\bf 166} (1995), 349-364

\bibitem[C17]{C17}
{\sc F. Comets}, 
\newblock{Directed polymers in random environments}, \textit{Lect. Notes Math.} 2175, Springer, 2017.

\bibitem[CC18]{CC18}
\textsc{F. Comets} and \textsc{C. Cosco}, \newblock{Brownian Polymers in Poissonian Environment: a survey}, \newblock{arXiv:1805.10899.}  (2018)



\bibitem[CCM20]{CCM20}
{\sc F. Comets}, {\sc C. Cosco} and {\sc C. Mukherjee},
\newblock{Renormalizing the Kardar-Parisi-Zhang equation in weak disorder in $d\geq 3$},
\newblock{\it Journal of Statistical Physics.} (2020)

\bibitem[CCM19]{CCM19b}
{\sc F. Comets}, {\sc C. Cosco} and {\sc C. Mukherjee},
\newblock{Space-time fluctuation of the Kardar-Parisi-Zhang equation in $d \geq 3$ and the Gaussian free field},
\newblock{\textit{arXiv:1905.03200}} 


\bibitem[CN19]{CN19}
{\sc C. Cosco} and {\sc S. Nakajima},
\newblock{Gaussian fluctuations for the directed polymer partition function for $d\geq 3$ and in the whole $L^2$-region},
\newblock{To appear in \textit{Ann. Inst. Poinc.}, (2020) \emph{arXiv:1903.00997}}


\bibitem[CNN20]{CNN20}
{\sc C. Cosco},  {\sc S. Nakajima}, and {\sc M.~Nakashima}
\newblock{Law of large numbers and fluctuations in the sub-critical and $L^2$ regions for SHE and KPZ equation in dimension $d\geq 3$},
\newblock{\emph{arXiv:2005.12689}, (2020)}




\bibitem[CY06]{CY06}
\textsc{F. Comets} and \textsc{N. Yoshida,}  \newblock{Directed polymers in random environment are diffusive at weak disorder,} \newblock{\textit{Ann. Probab.}} 34 (2006), no. 5, 1746--1770.  



\bibitem[DG20]{DG20}
{\sc A. Dunlap} and {\sc Y. Gu},
\newblock{A forward-backward SDE from the $2D$ nonlinear stochastic heat
equation},
\newblock 
 (2020), arXiv:2010.03541



\bibitem[DGRZ18b]{DGRZ18b}
{\sc A. Dunlap}, {\sc Y. Gu}, {\sc Lenya Ryzhik} and {\sc Ofer Zeitouni}, 
\newblock{The random heat equation in dimensions three and higher: the homogenization viewpoint,}
\newblock 
 (2018), arXiv:1808.07557

\bibitem[DGRZ20]{DGRZ20}
{\sc A. Dunlap}, {\sc Y. Gu}, {\sc Lenya Ryzhik} and {\sc Ofer Zeitouni}, 
\newblock{Fluctuations of the solutions to the KPZ equation in dimensions three and higher,}  \textit{Probab. Theory Related Fields}  176 (2020), no. 3-4, 1217-1258. 


\bibitem[EK86]{EK86}
{\sc S.\,Ethier} and {\sc Thomas G.\,Kurtz}, 
\newblock{Markov Processes Characterization and Convergence,}  John Wiley \& Sons, (1986)


\bibitem[GIP15]{GIP15}
{\sc M. Gubinelli}, {\sc P. Imkeller} and {\sc N. Perkowski}, 
\newblock{Paracontrolled distributions and singular
PDEs},
\newblock{Forum Math. Pi}, 3:e6, 75, 2015.

\bibitem[GP18]{GP18}
{\sc M. Gubinelli} and {\sc N. Perkowski}, 
\newblock{Energy solutions of KPZ are unique},
\newblock{J. Amer. Math. Soc.}, 31(2):427-471, (2018).

\bibitem[GL20]{GL20}
  {\sc Y.~Gu} and {\sc J.~ Li},
\newblock{Fluctuations of a nonlinear stochastic heat equation in dimensions three and higher},
\newblock{SIAM Journal on Mathematical Analysis},
{52}, (2020),
no.~{6},
  {5422--5440},

\bibitem[GQT19]{GQT19} \textsc{Y. Gu}, \textsc{J. Quastel} and \textsc{L.C. Tsai}, \newblock Moments of the 2D SHE at criticality, \newblock arXiv:1905.11310, (2019).

\bibitem[GRZ18]{GRZ18}
{\sc Y. Gu}, {\sc L. Ryzhik} and {\sc O. Zeitouni}, 
\newblock{The Edwards-Wilkinson limit of the random heat equation in dimensions three and higher},
\newblock{Comm. Math. Phys.}, {\bf 363} (2018), No. 2, pp. 351-388


\bibitem[G20]{G20}
{\sc Y. Gu},
\newblock{Gaussian fluctuations of the $2$D KPZ equation,}
\newblock \textit{ Stoch. Partial Differ. Equ. Anal. Comput.} 8 (2020), no. 1, 150-185. 



\bibitem[H13]{H13}
{\sc M. Hairer},
{Solving the KPZ equation},
\textit{ Annals of Mathematics} \textbf{178}  (2013) 558--664



\bibitem[H14]{H14}
{\sc M. Hairer},
\newblock{A theory of regularity structures,}
\textit{Inventiones mathematicae}
{\bf 198:2} (2014) 269--504 

  
\bibitem[IS88]{IS88}
  {\sc J. Imbrie} and {\sc T. Spencer}, \newblock Diffusion of directed polymers in a random environment, \newblock {\it Journal of Statistical Physics.}
   52(3-4), 609-626. (1988)


\bibitem[JS87]{JS87}
{\sc J. Jacod} and {\sc A. Shiryaev},
\newblock{\it Limit theorems for stochastic processes}, Springer-Verlag, Berlin (1987)

\bibitem[J97]{J97}
{\sc S. Janson}, \newblock{\it Gaussian Hilbert Spaces}, Vol.{129}, {Cambridge University Press}
(1997)


\bibitem[KPZ86]{KPZ}
{\sc M.~Kardar, G.~Parisi, Y.C.~Zhang,}
\newblock{Dynamic scaling of growing interfaces},
\newblock{\it Physical Review Letters},
Vol.{56}, No.~{9},  pp.889--892, {1986}


\bibitem[KM17]{KM17}
{\sc A. Kupiainen} and {\sc M. Marcozzi},
Renormalization of generalized KPZ equation. 
\textit{Journal of Statistical Physics} {\bf 166} (2017)  876--902.

\bibitem[LZ20]{LZ20}
\textsc{D. Lygkonis} and \textsc{N. Zygouras}. \newblock Edwards-Wilkinson fluctuations for the directed polymer  in the full $L^2$-regime for dimensions $d \geq 3$, {\emph{arXiv:2005.12706}, (2020)}

\bibitem[MU17]{MU17}
{\sc J. Magnen} and {\sc J. Unterberger}, 
\newblock{The scaling limit of the KPZ equation in space dimension 3 and higher}, 
\newblock{\it Journal of Statistical Physics.} {\bf 171:4}  (2018) 543-598


\bibitem[MSZ16]{MSZ16}
{\sc C. Mukherjee}, {\sc A. Shamov} and {\sc O. Zeitouni},
\newblock{Weak and strong disorder for the stochastic heat equation and the continuous directed polymer in $d\geq 3$}, 
\newblock{\it Electr. Comm. Prob.} {\bf 21}  (2016) 12 pp. 


\bibitem[S95]{Si95}
\textsc{Y. Sinai,}
A remark concerning random walks with random potentials,
{\it Fund. Math.} {\bf 147}  (1995) 173--180.


\bibitem[V06]{V06}
   V. Vargas, \newblock A local limit theorem for directed polymers in random media: the continuous and the discrete case, \newblock {\em Ann. Inst. H. Poincar\'{e} Probab. Stat.} 42(5), 521-534 (2006)

\end{thebibliography}
\end{document}